\documentclass[11pt,reqno]{amsart}
\usepackage{graphicx,amsmath,amssymb,amsfonts,amsthm,enumitem,bm,xcolor}
\usepackage{a4wide,fullpage}
\usepackage{hyperref}
\usepackage{cleveref}

\usepackage{mathtools}
\DeclarePairedDelimiter{\ceil}{\lceil}{\rceil}

\usepackage[normalem]{ulem} 

\usepackage{url}

\usepackage{color}

\renewcommand{\a}{\alpha}

\renewcommand{\l}{\lambda}

\renewcommand{\(}{\left\(}
\renewcommand{\)}{\right\)}

\newcommand{\os}[2]{\overset{#1}{#2}}

\newcommand{\pa}[2]{\left(\frac{#1}{#2}\right)}

\numberwithin{equation}{section}
\theoremstyle{plain}
\newtheorem{theorem}{Theorem}[section]
\newtheorem{lemma}[theorem]{Lemma}

\newtheorem{remark}[theorem]{Remark}

\newtheorem{definition}[theorem]{Definition}

\numberwithin{equation}{section}

\renewcommand{\binom}[2]{\left(\begin{smallmatrix}#1\\\\#2\end{smallmatrix}\right)}
\newcommand{\smallbinom}[2]{(\begin{smallmatrix}#1\\#2\end{smallmatrix})}

\setlist[enumerate]{leftmargin=*,label=\rm{(\arabic*)}}

\definecolor{blaugrau}{rgb}{0.796887, 0.789075, 0.871107}
\newcounter{mmacnt}
\def\restartmma{\setcounter{mmacnt}{0}}
\restartmma \catcode`|=\active
\def|#1|{\mathrm{#1}}
\catcode`|=12
\newenvironment{mma}{
	\par
	\catcode`|=\active
	\parskip=4pt\parindent=0pt 
	\small
	\def\In##1\\{%
		\def\linebreak{\hfill\break\null\qquad}%
		\refstepcounter{mmacnt}
		\hangindent=2.5em\hangafter=0
		\leavevmode
		\llap{\tiny\sffamily In[\arabic{mmacnt}]:=\kern.5em}%
		\mathversion{bold}\footnotesize$\tt\bf\displaystyle##1$\normalsize
		\mathversion{normal}\par
	}%
	\def\Print##1\\{%
		\def\linebreak{\hfill\break}%
		\hangindent=2.5em\hangafter=0
		\leavevmode\scriptsize ##1\par}%
	\def\Out##1\\{%
		\vspace*{-0.2cm}\def\linebreak{$\hfill\break\null\hfill$}%
		\kern\abovedisplayskip\par
		\hangindent=2.5em\hangafter=0
		\leavevmode
		\llap{\tiny\sffamily Out[\arabic{mmacnt}]=\kern.5em}
		\footnotesize$\displaystyle\tt##1$\normalsize\hfill\null\par
		\kern\belowdisplayskip\vspace*{-0.3cm}
	}%
	\def\Warning##1##2\\{%
		\def\linebreak{\hfill\break}%
		\hangindent=2.5em\hangafter=0
		\leavevmode
		{\scriptsize##1 : ##2}\par}%
}{%
	\par\medskip
}

\newcommand{\LoadP}[1]{\fcolorbox{black}{blaugrau}{
		\begin{minipage}[t]{13cm}
			\footnotesize #1
\end{minipage}}}

\newcommand{\myIn}[1]{{\sffamily In[#1]}}
\newcommand{\myOut}[1]{{\sffamily Out[#1]}}

\def\MLabel#1{{\refstepcounter{mmacnt}\label{#1}}\addtocounter{mmacnt}{-1}}

\makeatletter
\@namedef{subjclassname@2020}{%
	\textup{2020} Mathematics Subject Classification}
\makeatother

\allowdisplaybreaks

\title{Asymptotics for the reciprocal and shifted quotient of the partition function}
\author{Koustav Banerjee}
\author{Peter Paule}
\author{Cristian-Silviu Radu}
\author{Carsten Schneider}
\address{Research Institute for Symbolic Computation (RISC), Johannes Kepler University, Altenberger Straße 69, A-4040 Linz, Austria. In addition to RISC, Peter Paule is affiliated as guest professor to the Center for Applied Mathematics (TCAM), Tianjin University, Tianjin 300072, P. R. China. The first author is affiliated at Department of Mathematics and Computer Science, Division of Mathematics, University of Cologne, Weyertal 86-90, 50931 Cologne, Germany}
\email{kbanerj1@uni-koeln.de, Koustav.Banerjee@risc.jku.at}
\email{Peter.Paule@risc.jku.at}
\email{sradu@risc.jku.at}
\email{Carsten.Schneider@risc.jku.at}


\begin{document} 

\begin{flushright}
	\hfill RISC Report Series 24-06\\[0.5cm] 
\end{flushright}

	\maketitle
	\begin{abstract}
		Let $p(n)$ denote the partition function. In this paper our main goal is to derive an asymptotic expansion up to order $N$ (for any fixed positive integer $N$) along with estimates for error bounds for the shifted quotient of the partition function, namely $p(n+k)/p(n)$ with $k\in \mathbb{N}$, which generalizes a result of Gomez, Males, and Rolen. In order to do so, we derive asymptotic expansions with error bounds for the shifted version $p(n+k)$ and the multiplicative inverse $1/p(n)$, which is of independent interest.
	\end{abstract}

\tableofcontents
	
	\section{Introduction and summary of results}
	A partition $\lambda$ of a positive integer $n$ is a weakly decreasing sequence of positive integers $\l=\left(\lambda_1,\dots,\lambda_r\right)$ such that $\sum_{j=1}^{r}\lambda_j=n$. The total number of partitions of $n$ is denoted by $p(n)$ with $p(0):=1$. A rigorous study on $p(n)$ from the analytic point of view began with the seminal work of Hardy and Ramanujan \cite{HR}. In \cite{HR}, developing a widely celebrated tool 
	 known as Circle Method, they established the asymptotic growth of $p(n)$ which states that
	\[
	p(n)\sim \frac{e^{\pi\sqrt{\frac{2n}{3}}}}{4n\sqrt{3}}\ \ \text{as}\ n\rightarrow\infty.
	\]
	Later Rademacher \cite{R} refined the Circle Method to give an exact formula for $p(n)$. Lehmer \cite{L1,L2} estimated the error term after truncating Rademacher's convergent series for $p(n)$. Among many other applications, Rademacher's formula along with Lehmer's estimate have been used to prove inequalities for $p(n)$. The first such instance has been documented in Nicolas' work \cite{N} on the {\it log-concavity} of $p(n)$, where a sequence (of positive real numbers) $(a_n)_{n\ge 0}$ is called {\it log-concave} if $a_n^2\ge a_{n-1}a_{n+1}$ for all $n\ge 1$. We refer to \cite{C1,RPS} for a more detailed study on {\it log-concavity} and its higher order analogues known as {\it higher order Tur\'{a}n inequalities}. After a standstill for more than thirty years, DeSalvo and Pak's work \cite{DP} on the {\it log-concavity} of $p(n)$ and its associated inequalities not only confirmed the conjectures of Chen \cite{C} but also resurrected this particular area of research on $p(n)$. Since then an extensive amount of work on inequalities for $p(n)$ and its variations have been recorded. Among many others, we refer the readers to the work of Bringmann, Kane, Rolen, and Tripp \cite{BKRT} on inequalities for the {\it fractional partition function} $p_{\alpha}(n)$ arising from the famous {\it Nekrasov–Okounkov formula} \cite{NO} and to Ono, Pujahari, and Rolen's work \cite{OPR} on Tur\'{a}n inequalities for the {\it plane partition function} $\text{PL}(n)$. For a detailed study on $\text{PL}(n)$, see \cite{M}.
	
	In the following we shall restrict ourselves to inequalities for $p(n)$. The motivation for doing so is as follows. We call a polynomial with real coefficients {\it hyperbolic} if all roots are real. The sequence of so-called {\it Jensen polynomials} associated with $p(n)$ is defined as
	\begin{equation}\label{Jenpoly}
	J^{d,n}_{p}(x):=\sum_{j=0}^{d}\binom{d}{j}p(n+j)x^j.
	\end{equation} 
	From work of Nicolas \cite{N} and DeSalvo and Pak \cite{DP} we already know that $J^{2,n}_p(x)$ is hyperbolic for $n\ge 26$. The case for $d=3$ has been settled by Chen, Jia, and Wang \cite{CJW}. Soon after, in their fundamental work on a problem entangled with variants of the {\it Riemann hypothesis}, Griffin, Ono, Rolen, and Zagier \cite{GORZ} confirmed that for any $d\in \mathbb{Z}_{\ge 0}$, $J^{d,n}_p(x)$ is eventually hyperbolic which settled the conjecture of Chen, Jia, and Wang \cite{CJW}. In \cite[Theorem 3]{GORZ}, Griffin, Ono, Rolen, and Zagier developed a generic framework to decide the {\it hyperbolicity} of Jensen polynomials associated with a family of sequences satisfying {\it certain} growth conditions. In particular, for $p(n)$ they proved the hyperbolicity of $J^{d,n}_p(x)$ for sufficiently large $n$ by showing the following uniform convergence property:
	\begin{equation}\label{GORZeqn}
	\underset{n\rightarrow\infty}{\lim}\left(\frac{\delta(n)^d}{p(n)}J^{d,n}_p\left(\frac{\delta(n)x-1}{\exp\left(A(n)\right)}\right)\right)=H_d(x),
	\end{equation}
	with $H_d(x)$ the $d$-th Hermite polynomial and where the sequences $(\delta(n))_{n\ge 0}$, $(A(n))_{n\ge 0}$ are determined by
	\begin{equation*}
	\log \left(\frac{p(n+j)}{p(n)}\right)=A(n)j-\delta(n)^2j^2+o\left(\delta(n)^d\right)\ \ \text{as}\ n\rightarrow \infty.
	\end{equation*} 
	Larson and Wagner \cite[Theorem 1.3]{LW} derived an effective cutoff $N(d)$ for $n$ such that $J^{d,n}_p(x)$ is hyperbolic for all $n\ge N(d)$. In order to arrive at their main result, one of the quintessential steps was to get an asymptotic expansion of the quotient $p(n+j)/p(n)$; see \cite[Lemmas 2.2 and 2.3]{LW}. Recently, Gomez, Males, and Rolen \cite[Theorem 1.1]{GMR} obtained a more effective error bound for the asymptotic expansion of the ratio $p(n-j)/p(n)$ truncating after the first three terms of the asymptotic series, which, in turn, has many applications including a shifted convexity property of $p(n)$ (see \cite{A,GO,GU,KK,O}) and new estimates of the $k$-rank partition function $N_k(m,n)$ defined by Garvan \cite{GA}. In this article we generalize the result \cite[Theorem 1.1]{GMR} of Gomez, Males, and Rolen as stated below in Theorem \ref{Mainthm} but with a slight variant; namely, for quotients with positive shifts $p(n+k)/p(n)$. 
	
	Throughout, we write $f(x)=O_{\le c}\left(g(x)\right)$ if $\left|f(x)\right|\le c g(x)$ for a positive function $g$ and for a domain of $x$ which we will specify in each given context. With this convention we can state the main result of this article. Moreover, we use $\mathbb{N}$ to denote the set of positive integers. 
	\begin{theorem}\label{Mainthm}
		Let $(k,N)\in\mathbb{N}^2$. Then for all $n\ge n_N(k)$
		\begin{equation*}
		\frac{p(n+k)}{p(n)}=\sum_{m=0}^{N}\dfrac{c_k(m)}{n^{\frac m2}}+O_{\le E_N(k)}\left(n^{-\frac{N+1}{2}}\right),
		\end{equation*}
		where $c_k(0)=1$, the constants $(c_k(m))_{m\ge1}$, $n_N(k)$, and $E_N(k)$ are determined effectively for any fixed $N$. 
	\end{theorem}
\begin{remark}\label{BPRSremark1}
The $(c_k(m))_{m\ge1}$, $n_N(k)$, and $E_N(k)$ are stated explicitly in \eqref{Mainthmcoeff}, \eqref{Mainthmcutoff}, and \eqref{Mainthmerrorfinal} respectively.
\end{remark}

	\begin{remark}\label{BRSremark2}
		\begin{enumerate}
			\item In their study on asymptotic $r$-log-concavity for $p(n)$, Hou and Zhang \cite[Theorem 1.3]{HZ} determined an asymptotic expansion of $p(n+1)/p(n)$. Theorem \ref{Mainthm} is a straightforward generalization of their result. In particular, now we have estimated effectively the cutoff $n_N(k)$, the error bound $E_N(k)$, and also provided a concrete description of the coefficient sequence $(c_k(m))_{m\ge 0}$.
			\item In light of the remark \cite[Remark 1.3]{GMR} made by Gomez, Males, and Rolen on positivity of $\Delta^r_j(p(n))$, where $\Delta^r_j$ is the $r$-fold application of the difference operator $\Delta_j$ defined by $\left(\Delta_jp\right)(n)=p(n)-p(n-j)$, it seems that by making the shift $n\mapsto n+j$, one might use Theorem \ref{Mainthm} by choosing the truncation point $N(r)$\footnote{For instance, one can choose $N(r)=r+1$, see \cite[Chap. 6, Sec. 6.7.2]{B}.} (depending on $r\in \mathbb{N}$) in the asymptotic expansion of $p(n+k)/p(n)$ appropriately so as to obtain the asymptotic growth of $\Delta^r_j(p(n))$. In addition to that, we can estimate a cutoff $n_{N(r),r}(j)$ such that  $\Delta^r_j(p(n))\ge 0$ for all $n\ge n_{N(r),r}(j)$. The case $r=2$ was settled in \cite[Theorem 1.2]{GMR}. 
		\end{enumerate}
	\end{remark}
	
	We summarize in brief the main tools which will be used in proving Theorem \ref{Mainthm}. First, we will start with the infinite family of inequalities for $p(n)$ given in \cite[Theorem 4.4]{BPRZ} that states: for $m\in \mathbb{Z}_{\ge 2}$ and $n>g(m)$ (see Theorem \ref{BPRSthm1eqn} below),
	\[
	p(n)=\frac{\sqrt{12}e^{\mu(n)}}{24n-1}\left(1-\frac{1}{\mu(n)}+O_{\le 1}\left(\mu(n)^{-m}\right)\right)\ \ \text{with}\ \mu(n):=\frac{\pi}{6}\sqrt{24n-1}.
	\]
	This result has twofold applications in the present context: (i) by considering the shift $n\mapsto n+k$, after Taylor series expansion of the major term $\frac{\sqrt{12}e^{\mu(n+k)}}{24(n+k)-1}\left(1-\frac{1}{\mu(n+k)}\right)$, we extract the coefficient sequence in the asymptotic expansion of $p(n+k)$ and similarly, (ii) for exhibiting the coefficient sequence in the asymptotic expansion of $1/p(n)$, we use Taylor series expansion for the inverse of the major term; i.e., $\frac{24n-1}{\sqrt{12}e^{\mu(n)}}\left(1-\frac{1}{\mu(n)}\right)^{-1}$.
	
	In addition, we shall use the {\it symbolic summation} package {\texttt{Sigma}}, developed by the third author \cite{S}, to simplify the coefficients appearing in the asymptotic expansion for $1/p(n)$ so as to obtain precise error bound estimates for the asymptotic expansion under consideration.
	
	The rest of this paper is organized as follows: in Section \ref{Sec1}, we will recall a list of preliminary results from \cite{BPRS,BPRZ} which will be helpful in carrying out estimates in the subsequent sections. In Sections \ref{Sec2} and \ref{Sec3}, we will provide asymptotic expansions for the shifted partition function $p(n+k)$ (see Theorem~\ref{Mainthmpart1}) and the inverse $1/p(n)$ (see Theorem~\ref{Mainthmpart2}), respectively. The proof of Theorem \ref{Mainthm} is given in Section \ref{Sec4}. Finally in Section \ref{Sec5}, we lay out a detailed exposition on the usage of the symbolic summation package {\texttt{Sigma}}.
	
	\section{Preliminaries}\label{Sec1}
The following results from \cite{BPRS} and \cite{BPRZ} will be needed. Throughout, for $n\in \mathbb{N}$, $$\mu(n):=\frac{\pi}{6}\sqrt{24n-1}.$$
	\begin{theorem}\cite[Theorem 4.4]{BPRZ}\label{BPRSthm1}
		For $m \in \mathbb{Z}_{\geq 2}$, define
		$$\widehat{g}(m):=\dfrac{1}{24}\Biggl(\dfrac{36}{\pi^2}\cdot \nu(m)^2+1\Biggr),$$
		where $$\nu(m):=2\log 6+(2\log 2)m+2m \log m+2m \log \log m+\dfrac{5m \log \log m}{\log m}.$$ Then for all $m\in\mathbb{Z}_{\geq 2}$ and  $n> \widehat{g}(m)$ such that $(n,m)\neq (6,2)$, we have
		\begin{equation}\label{BPRSthm1eqn}
		p(n)=\frac{\sqrt{12}e^{\mu(n)}}{24n-1}\left(1-\frac{1}{\mu(n)}+O_{\le 1}\left(\mu(n)^{-m}\right)\right).
		\end{equation}
	\end{theorem}
\begin{lemma}\cite[Lemma 4.1]{BPRS}\label{BPRSprelimlem1}
	Let $x_1,x_2,\dots,x_n$ and $y_1,\dots,y_n$ be non-negative real numbers such that $x_j\le 1$ for $1\le j\le n$. Then
	$$\frac{(1-x_1)(1-x_2)\cdots (1-x_n)}{(1+y_1)(1+y_2)\cdots(1+y_n)}\geq 1-\sum_{j=1}^nx_j-\sum_{j=1}^ny_j.$$
\end{lemma}
From \cite[eqs. (1) and (2)]{Rob}, we obtain the following inequality for the central binomial coefficients. 
\begin{lemma}\label{BPRSprelimlem2} For $n\in \mathbb{N}$, we have 
\[
\frac{4^n}{\sqrt{\pi n}}\left(1-\frac{1}{8n}\right)<\Large\binom{2n}{n}<\frac{4^n}{\sqrt{\pi n}}.
\]
\end{lemma}
\begin{lemma}\cite[Lemma 4.6]{BPRS}\label{BPRSprelimlem3}
	For $t\geq 1$ and $k\in\{0,1,2,3\}$, with $\a=\frac{\pi}{6}$, we have
	\[\sum_{u=t+1}^{\infty}\frac{u^k\alpha^{2u}}{(2u)!}\leq \frac{C_k}{t^2}\ \ \text{with}\ \ C_k=\frac{\alpha^4 2^k}{18}.\]
\end{lemma}
We generalize Lemma \ref{BPRSprelimlem3} only for the case $k=0$ stated below.
\begin{lemma}\label{BPRSprelimlem4}
For $t\geq 1$ and $0<c<2.2$, we have
	\[\sum_{u=t+1}^{\infty}\frac{c^{2u}}{(2u)!}\leq \frac{c^4}{18\cdot t^2}.\]	
\end{lemma}
\begin{proof}
The statement is obtained by a routine calculation following the proof of Lemma \ref{BPRSprelimlem3} (cf. \cite[Section 8, p. 38]{BPRS}).
\end{proof}

	\section{Asymptotics of $p(n+k)$}\label{Sec2}
In this section we shall prove the following theorem on the asymptotic expansion of $p(n+k)$.

\begin{theorem}\label{Mainthmpart1}
Let $(k, N)\in \mathbb{N}^2$ and $\widehat{g}(m)$ be as in Theorem \ref{BPRSthm1}. For $n>n_N^{[1]}(k)$ we have
\begin{equation}\label{Mainthmpart1eqn1}
p(n+k)=\frac{e^{\pi\sqrt{2n/3}}}{4n\sqrt{3}}\left(\sum_{t=0}^{N}\frac{\omega_k^{[1]}(t)}{n^{\frac t2}}+O_{\le E_N^{[1]}(k)}\left(n^{-\frac{N+1}{2}}\right)\right),
\end{equation}
where
\begin{equation}\label{Mainthmpart1eqn2}
\omega_k^{[1]}(t)=\pa{24k-1}{4\sqrt{6}}^t\sum_{\ell=0}^{\frac{t+1}{2}}\binom{t+1}{\ell}\frac{(t+1-\ell)}{(t+1-2\ell)!}(-1)^{\ell}\pa{\pi}{6}^{t-2\ell}\pa{1}{24k-1}^{\ell},
\end{equation}
and where $n_N^{[1]}(k), E_N^{[1]}(k)$ are determined effectively in \eqref{Mainthmpart1eqn3} and \eqref{Mainthmpart1eqn6} respectively.
\end{theorem}
In order to prove Theorem \ref{Mainthmpart1} we need some preparatory estimates documented below. First, applying Theorem \ref{BPRSthm1} with $m\mapsto N+1$, and taking the shift $n\mapsto n+k$, we have the following.
\begin{lemma}\label{BPRSshiftlem1}
Let $(k,N)\in \mathbb{N}^2$ and $\widehat{g}(m)$ be as in \eqref{BPRSthm1}. For $n>\widehat{g}(N+1)$ we have
\[
p(n+k)=\frac{\sqrt{12}e^{\mu(n+k)}}{24(n+k)-1}\left(1-\frac{1}{\mu(n+k)}+O_{\le 1}\left(\mu(n+k)^{-N-1}\right)\right).
\]	
\end{lemma}
Next, extracting the factor $\frac{e^{\pi\sqrt{2n/3}}}{4n\sqrt{3}}$ from the major term $\frac{\sqrt{12}e^{\mu(n+k)}}{24(n+k)-1}\left(1-\frac{1}{\mu(n+k)}\right)$ and by Taylor series expansion of the residual part, we get the following.
\begin{lemma}\label{BPRSshiftlem2}Let $(\omega_k^{[1]}(t))_{t\ge 0}$ be as in \eqref{Mainthmpart1eqn2}. For $k\in \mathbb{N}$,
\[
\frac{\sqrt{12}e^{\mu(n+k)}}{24(n+k)-1}\left(1-\frac{1}{\mu(n+k)}\right)=\frac{e^{\pi\sqrt{2n/3}}}{4n\sqrt{3}}\sum_{t=0}^{\infty}\frac{\omega_k^{[1]}(t)}{n^{\frac t2}}.
\]	
\end{lemma}	
\begin{proof}
Analogous to the proof of \cite[Proposition 4.4]{OS}.
\end{proof}
Next, we truncate the Taylor series $\sum_{t=0}^{\infty}\frac{\omega_k^{[1]}(t)}{n^{\frac t2}}$ at $N$ and estimate the absolute value of the remainder part by estimating $\left|\omega_k^{[1]}(t)\right|$ independence of $t$ even or odd. 

For $k\in \mathbb{N}$, we define
\begin{equation}\label{BPRSshiftdef1}
\alpha_k:=\mu(k)=\frac{\pi}{6}\sqrt{24k-1}
\end{equation}
and 
\begin{equation}\label{BPRSshiftdef2}
C_1(k):=\frac{3k\left(3\ceil[\big]{\sqrt{k}}+1\right)^2\alpha_k^{6\left\lceil\sqrt{k}\right\rceil+4}}{2\ceil[\big]{\sqrt{k}}\left(6\ceil[\big]{\sqrt{k}}+4\right)!}.
\end{equation}
\begin{lemma}\label{BPRSshiftlem3}
	Let $\alpha_k$ and $C_1(k)$ be as in \eqref{BPRSshiftdef1} and \eqref{BPRSshiftdef2}, respectively. Then for all $(k,t)\in \mathbb{N}^2$,
	\[
	\left|\omega_k^{[1]}(2t+1)\right|\le \pa{24k-1}{24}^t\pa{6}{\pi^3}^{\frac 12}\sqrt{t+1}\left|\cos\left(\alpha_k\right)\right|\left(1+\frac{C_1^*(k)}{t}\right),
	\]
	with 
	\begin{equation}\label{BPRSshiftlem3eqn0}
	C_1^*(k):=\frac{\alpha_k^2 \left(\cosh\left(\alpha_k\right)-1\right)+4\cdot C_1(k)}{4\cdot \left|\cos\left(\alpha_k\right)\right|}.
	\end{equation}
\end{lemma}
\begin{proof}
	By \eqref{Mainthmpart1eqn2} and \eqref{BPRSshiftdef1},
	\begin{align}\nonumber
	\omega_k^{[1]}(2t+1)&=\pa{24k-1}{4\sqrt{6}}^{2t+1}\sum_{\ell=0}^{t+1}\binom{2t+2}{\ell}\frac{(2t+2-\ell)}{(2t+2-2\ell)!}(-1)^{\ell}\pa{\pi}{6}^{2t+1-2\ell}\pa{1}{24k-1}^{\ell}\\\nonumber
	&\hspace{-0.5 cm}\os{(\ell\mapsto t+1-\ell)}{=}(-1)^{t+1}\frac{(24k-1)^t}{(4\sqrt{6})^{2t+1}}\sum_{\ell=0}^{t+1}\binom{2t+2}{t+1-\ell}(t+1+\ell)\frac{(-1)^{\ell}}{(2\ell)!}\pa{\pi}{6}^{2\ell-1}(24k-1)^{\ell}\\\nonumber
	&=\frac{(24k-1)^t}{(4\sqrt{6})^{2t+1}}\frac{6}{\pi}\binom{2t+2}{t+1}(t+1)(-1)^{t+1}\sum_{\ell=0}^{t+1}\frac{\binom{2t+2}{t+1-\ell}(t+1+\ell)}{\binom{2t+2}{t+1}(t+1)}\frac{(-1)^{\ell}}{(2\ell)!}\alpha_k^{2\ell}\\\label{BPRSshiftlem3eqn1}
	&=:\frac{(24k-1)^t}{(4\sqrt{6})^{2t+1}}\frac{6}{\pi}\binom{2t+2}{t+1}(t+1)(-1)^{t+1}S_k^{[1]}(t).
	\end{align}
	Now, we bound $|S_k^{[1]}|(t)$ as
	\begin{align}\nonumber
	\left|S_k^{[1]}(t)\right|&\le \left|\cos\left(\alpha_k\right)\right|+\left|S_k^{[1]}(t)-\cos\left(\alpha_k\right)\right|\\\nonumber
	&=\left|\cos\left(\alpha_k\right)\right|+\left|\sum_{\ell=0}^{t+1}\frac{\binom{2t+2}{t+1-\ell}(t+1+\ell)}{\binom{2t+2}{t+1}(t+1)}\frac{(-1)^{\ell}}{(2\ell)!}\alpha_k^{2\ell}-\sum_{\ell=0}^{\infty}\frac{(-1)^{\ell}}{(2\ell)!}\alpha_k^{2\ell}\right|\\\nonumber
	&\le \left|\cos\left(\alpha_k\right)\right|+\sum_{\ell=0}^{t+1}\left|\frac{\binom{2t+2}{t+1-\ell}(t+1+\ell)}{\binom{2t+2}{t+1}(t+1)}-1\right|\frac{\alpha_k^{2\ell}}{(2\ell)!}+\sum_{\ell=t+2}^{\infty}\frac{\alpha_k^{2\ell}}{(2\ell)!}\\\nonumber
	&=\left|\cos\left(\alpha_k\right)\right|+\sum_{\ell=2}^{t+1}\left|\frac{\binom{2t+2}{t+1-\ell}(t+1+\ell)}{\binom{2t+2}{t+1}(t+1)}-1\right|\frac{\alpha_k^{2\ell}}{(2\ell)!}+\sum_{\ell=t+2}^{\infty}\frac{\alpha_k^{2\ell}}{(2\ell)!}\\\nonumber
	&\hspace{5 cm} \left(\text{as}\ \frac{\binom{2t+2}{t+1-\ell}(t+1+\ell)}{\binom{2t+2}{t+1}(t+1)}=1\ \ \text{for}\ \ell\in \{0,1\}\right)\\\label{BPRSshiftlem3eqn2}
	&=:\left|\cos\left(\alpha_k\right)\right|+S_{k,1}^{[1]}(t)+S_{k,2}^{[1]}(t).
	\end{align}
	Analogous to the proof of \cite[Lemma 6.2.7]{B}, one can derive that for $t\ge 1$,
	\begin{equation}\label{BPRSshiftlem3eqn3}
	S_{k,2}^{[1]}(t)\le \frac{C_1(k)}{t^2}.
	\end{equation}
	To obtain an upper bound of $S_{k,1}^{[1]}(t)$, we first note that
	\[
	\frac{\binom{2t+2}{t+1-\ell}(t+1+\ell)}{\binom{2t+2}{t+1}(t+1)}=\prod_{j=2}^{\ell}\frac{t+2-j}{t+j},
	\]
which implies
	\begin{equation}\label{BPRSshiftlem3eqn4}
	\frac{\binom{2t+2}{t+1-\ell}(t+1+\ell)}{\binom{2t+2}{t+1}(t+1)}-1\le 0\ \ \text{as}\ \ \frac{t+2-j}{t+j}\le 1\ \ \text{for all}\ j\ge 1.
	\end{equation}
Moreover, by using Lemma \ref{BPRSprelimlem1},
	\begin{align*}
	\prod_{j=2}^{\ell}\frac{t+2-j}{t+j}=\prod_{j=2}^{\ell}\frac{1-\frac{j-2}{t}}{1+\frac{j}{t}}&\ge 1-\sum_{j=2}^{\ell}\frac{j-2}{t}-\sum_{j=2}^{\ell}\frac{j}{t}
	=1-\frac{\ell^2-\ell}{t},
	\end{align*}
	we have
	\begin{equation}\label{BPRSshiftlem3eqn5}
	\frac{\binom{2t+2}{t+1-\ell}(t+1+\ell)}{\binom{2t+2}{t+1}(t+1)}-1\ge -\frac{\ell^2-\ell}{t}.
	\end{equation}
	Combining \eqref{BPRSshiftlem3eqn4} and \eqref{BPRSshiftlem3eqn5}, it follows that
	\begin{equation}\label{BPRSshiftlem3eqn6}
	\left|\frac{\binom{2t+2}{t+1-\ell}(t+1+\ell)}{\binom{2t+2}{t+1}(t+1)}-1\right|\le \frac{\ell^2-\ell}{t}.
	\end{equation}
	Now, using \eqref{BPRSshiftlem3eqn6} and $\ell\ge 2$, we bound $S_{k,1}^{[1]}(t)$ against
	\begin{equation}\label{BPRSshiftlem3eqn7}
	S_{k,1}^{[1]}(t)\le \sum_{\ell=2}^{t+1}\frac{\ell^2-\ell}{t}\frac{\alpha_k^{2\ell}}{(2\ell)!}\le \frac{\alpha_k^2}{4t}\sum_{\ell=2}^{\infty}\frac{\alpha_k^{2\ell-2}}{(2\ell-2)!}=\frac{\alpha_k^2 \left(\cosh\left(\alpha_k\right)-1\right)}{4t}.
	\end{equation}
	Applying \eqref{BPRSshiftlem3eqn3} and \eqref{BPRSshiftlem3eqn7} to \eqref{BPRSshiftlem3eqn2} gives
	\begin{align}\nonumber
	\left|S_k^{[1]}(t)\right|&\le \left|\cos\left(\alpha_k\right)\right|+\frac{\alpha_k^2 \left(\cosh\left(\alpha_k\right)-1\right)}{4t}+\frac{C_1(k)}{t^2}\\\nonumber
	&\le \left|\cos\left(\alpha_k\right)\right|+\frac{\alpha_k^2 \left(\cosh\left(\alpha_k\right)-1\right)+4\cdot C_1(k)}{4t}\ \ \left(\text{as}\ t\ge 1\right)\\\label{BPRSshiftlem3eqn8}
	&=\left|\cos\left(\alpha_k\right)\right|\left(1+\frac{\alpha_k^2 \left(\cosh\left(\alpha_k\right)-1\right)+4\cdot C_1(k)}{4\left|\cos\left(\alpha_k\right)\right|\cdot t}\right)=\left|\cos\left(\alpha_k\right)\right|\left(1+\frac{C_1^*(k)}{t}\right),
	\end{align}
where the last equality is by \eqref{BPRSshiftlem3eqn0}. Summarizing, 
	\begin{align*}
	\left|\omega_k^{[1]}(2t+1)\right|&\le \frac{(24k-1)^t}{(4\sqrt{6})^{2t+1}}\frac{6}{\pi}\binom{2t+2}{t+1}(t+1)\left|S_k^{[1]}(t)\right|\ \ \left(\text{by}\ \eqref{BPRSshiftlem3eqn1}\right)\\
	&\le \frac{(24k-1)^t}{(4\sqrt{6})^{2t+1}}\frac{6}{\pi}\binom{2t+2}{t+1}(t+1)\left|\cos\left(\alpha_k\right)\right|\left(1+\frac{C_1^*(k)}{t}\right)\ \ \left(\text{by}\ \eqref{BPRSshiftlem3eqn8}\right)\\
	&\le \frac{(24k-1)^t}{(4\sqrt{6})^{2t+1}}\frac{6}{\pi}\frac{4^{t+1}}{\sqrt{\pi}}\sqrt{t+1}\left|\cos\left(\alpha_k\right)\right|\left(1+\frac{C_1^*(k)}{t}\right)\ \ \left(\text{by Lemma\  \ref{BPRSprelimlem2}}\right)\\
	&=\pa{24k-1}{24}^t\pa{6}{\pi^3}^{\frac 12}\sqrt{t+1}\left|\cos\left(\alpha_k\right)\right|\left(1+\frac{C_1^*(k)}{t}\right),
	\end{align*}
	which concludes the proof of Lemma \ref{BPRSshiftlem3}.
\end{proof}

For $k\in \mathbb{N}$ define
\begin{equation}\label{BPRSshiftdef3}
C_2(k):=\frac{9k\left(3\ceil[\big]{\sqrt{k}}+1\right)^2\alpha_k^{6\left\lceil\sqrt{k}\right\rceil+3}}{\left(6\ceil[\big]{\sqrt{k}}+1\right)\left(6\ceil[\big]{\sqrt{k}}+3\right)!}.
\end{equation}
\begin{lemma}\label{BPRSshiftlem4}
	Let $\alpha_k$ be as in \eqref{BPRSshiftdef1} and $C_2(k)$ be as in and \eqref{BPRSshiftdef3}. Then for all $(k,t) \in \mathbb{N}^2$,
	\[
	\left|\omega_k^{[1]}(2t)\right|\le \pa{24k-1}{24}^t\frac{2\sqrt{t}}{\sqrt{\pi}\alpha_k}\left|\sin\left(\alpha_k\right)\right|\left(1+\frac{C_2^*(k)}{t}\right)
	\]
	with 
	\begin{equation}\label{BPRSshiftlem4eqn0}
	C_2^*(k):=\frac 12\left(1+3\frac{\alpha_k^2 \cdot\sinh\left(\alpha_k\right)+4\cdot C_2(k)}{4\left|\sin\left(\alpha_k\right)\right|}\right).
	\end{equation}
\end{lemma}
\begin{proof}
	By \eqref{Mainthmpart1eqn2} and \eqref{BPRSshiftdef1},
	\begin{align}\nonumber
	\omega_k^{[1]}(2t)&=\pa{24k-1}{4\sqrt{6}}^{2t}\sum_{\ell=0}^{t}\binom{2t+1}{\ell}\frac{(2t+1-\ell)}{(2t+1-2\ell)!}(-1)^{\ell}\pa{\pi}{6}^{2t-2\ell}\pa{1}{24k-1}^{\ell}\\\nonumber
	&\hspace{-0.5 cm}\os{(\ell\mapsto t-\ell)}{=}\frac{(24k-1)^{t+\frac 12}}{(4\sqrt{6})^{2t}}\frac{6}{\pi}\binom{2t+1}{t}(t+1)(-1)^t\sum_{\ell=0}^{t}\frac{\binom{2t+1}{t-\ell}(t+1+\ell)}{\binom{2t+1}{t}(t+1)}\frac{(-1)^{\ell}}{(2\ell+1)!}\alpha_k^{2\ell+1}
	\\\nonumber
	&=\frac{(24k-1)^{t+\frac 12}}{(4\sqrt{6})^{2t}}\frac{6}{\pi}\binom{2t}{t}(2t+1)(-1)^t\sum_{\ell=0}^{t}\frac{\binom{2t+1}{t-\ell}(t+1+\ell)}{\binom{2t+1}{t}(t+1)}\frac{(-1)^{\ell}}{(2\ell+1)!}\alpha_k^{2\ell+1}\\\label{BPRSshiftlem4eqn1}
	&=:\frac{(24k-1)^{t+\frac 12}}{(4\sqrt{6})^{2t}}\frac{6}{\pi}\binom{2t}{t}(2t+1)(-1)^tS_k^{[2]}(t).
	\end{align}
	Now, we bound $|S_k^{[2]}(t)|$ as
	\begin{align}\nonumber
	\left|S_k^{[2]}(t)\right|&\le \left|\sin\left(\alpha_k\right)\right|+\left|S_k^{[2]}(t)-\sin\left(\alpha_k\right)\right|\\\nonumber
	&=\left|\sin\left(\alpha_k\right)\right|+\left|\sum_{\ell=0}^{t}\frac{\binom{2t+1}{t-\ell}(t+1+\ell)}{\binom{2t+1}{t}(t+1)}\frac{(-1)^{\ell}}{(2\ell+1)!}\alpha_k^{2\ell+1}-\sum_{\ell=0}^{\infty}\frac{(-1)^{\ell}}{(2\ell+1)!}\alpha_k^{2\ell+1}\right|\\\nonumber
	&\le \left|\sin\left(\alpha_k\right)\right|+\sum_{\ell=1}^{t}\left|\frac{\binom{2t+1}{t-\ell}(t+1+\ell)}{\binom{2t+1}{t}(t+1)}-1\right|\frac{\alpha_k^{2\ell+1}}{(2\ell+1)!}+\sum_{\ell=t+1}^{\infty}\frac{\alpha_k^{2\ell+1}}{(2\ell+1)!}\\\label{BPRSshiftlem4eqn2}
	&=:\left|\cos\left(\alpha_k\right)\right|+S_{k,1}^{[2]}(t)+S_{k,2}^{[2]}(t).
	\end{align}
Similar to \eqref{BPRSshiftlem3eqn3}, following the proof of \cite[Lemma 6.2.7]{B}, it follows that for $t\ge 1$,
	\begin{equation}\label{BPRSshiftlem4eqn3}
	S_{k,2}^{[2]}(t)\le \frac{C_2(k)}{t^2}.
	\end{equation}
	To obtain an upper bound of $S_{k,1}^{[2]}(t)$, observe that
	\[
	\frac{\binom{2t+1}{t-\ell}(t+1+\ell)}{\binom{2t+1}{t}(t+1)}=\prod_{j=1}^{\ell}\frac{t+1-j}{t+j},
	\]
which implies
	\begin{equation}\label{BPRSshiftlem4eqn4}
	\frac{\binom{2t+1}{t-\ell}(t+1+\ell)}{\binom{2t+1}{t}(t+1)}-1\le 0\ \ \text{as}\ \ \frac{t+1-j}{t+j}\le 1\ \ \text{for all}\ j\ge 1.
	\end{equation}
Moreover, by using Lemma \ref{BPRSprelimlem1},
	\begin{align*}
	\prod_{j=1}^{\ell}\frac{t+1-j}{t+j}=\prod_{j=1}^{\ell}\frac{1-\frac{j-1}{t}}{1+\frac{j}{t}}&\ge 1-\sum_{j=1}^{\ell}\frac{j-1}{t}-\sum_{j=1}^{\ell}\frac{j}{t}
	=1-\frac{\ell^2}{t},
	\end{align*}
	we have
	\begin{equation}\label{BPRSshiftlem4eqn5}
	\frac{\binom{2t+1}{t-\ell}(t+1+\ell)}{\binom{2t+1}{t}(t+1)}-1\ge -\frac{\ell^2}{t}.
	\end{equation}
	Combining \eqref{BPRSshiftlem4eqn4} and \eqref{BPRSshiftlem4eqn5}, it follows that
	\begin{equation}\label{BPRSshiftlem4eqn6}
	\left|\frac{\binom{2t+1}{t-\ell}(t+1+\ell)}{\binom{2t+1}{t}(t+1)}-1\right|\le \frac{\ell^2}{t}.
	\end{equation}
	Now, using \eqref{BPRSshiftlem4eqn6} and $\ell\ge 1$, we bound $S_{k,1}^{[2]}(t)$ against
	\begin{equation}\label{BPRSshiftlem4eqn7}
	S_{k,1}^{[2]}(t)\le \sum_{\ell=1}^{t}\frac{\ell^2}{t}\frac{\alpha_k^{2\ell+1}}{(2\ell+1)!}\le \frac{\alpha_k^2}{4t}\sum_{\ell=1}^{\infty}\frac{\alpha_k^{2\ell-1}}{(2\ell-1)!}=\frac{\alpha_k^2\cdot \sinh\left(\alpha_k\right)}{4t}.
	\end{equation}
	Applying \eqref{BPRSshiftlem4eqn3} and \eqref{BPRSshiftlem4eqn7} to \eqref{BPRSshiftlem4eqn2} gives
	\begin{align}\nonumber
	\left|S_k^{[2]}(t)\right|&\le \left|\sin\left(\alpha_k\right)\right|+\frac{\alpha_k^2 \cdot\sinh\left(\alpha_k\right)}{4t}+\frac{C_2(k)}{t^2}\le \left|\sin\left(\alpha_k\right)\right|+\frac{\alpha_k^2 \cdot\sinh\left(\alpha_k\right)+4C_2(k)}{4t}\ \ \left(\text{as}\ t\ge 1\right)\\\label{BPRSshiftlem4eqn8}
	&=\left|\sin\left(\alpha_k\right)\right|\left(1+\frac{\alpha_k^2 \cdot\sinh\left(\alpha_k\right)+4C_2(k)}{\left|\sin\left(\alpha_k\right)\right|\cdot 4t}\right).
	\end{align}
Summarizing,
	\begin{align*}
	\left|\omega_k^{[1]}(2t)\right|&\le \frac{(24k-1)^{t+\frac 12}}{(4\sqrt{6})^{2t}}\frac{6}{\pi}\binom{2t}{t}(2t+1)\left|S_k^{[2]}(t)\right|\ \ \left(\text{by}\ \eqref{BPRSshiftlem4eqn1}\right)\\
	&\le \frac{(24k-1)^{t+\frac 12}}{(4\sqrt{6})^{2t}}\frac{6}{\pi}\binom{2t}{t}(2t+1)\left|\sin\left(\alpha_k\right)\right|\left(1+\frac{\alpha_k^2 \cdot\sinh\left(\alpha_k\right)+4C_2(k)}{\left|\sin\left(\alpha_k\right)\right|\cdot 4t}\right)\ \ \left(\text{by}\ \eqref{BPRSshiftlem4eqn8}\right)\\
	&\le \frac{(24k-1)^{t+\frac 12}}{(4\sqrt{6})^{2t}}\frac{6}{\pi}\frac{4^{t}}{\sqrt{\pi t}}(2t+1)\left|\sin\left(\alpha_k\right)\right|\left(1+\frac{\alpha_k^2 \cdot\sinh\left(\alpha_k\right)+4C_2(k)}{\left|\sin\left(\alpha_k\right)\right|\cdot 4t}\right)\ \ \left(\text{by Lemma \ref{BPRSprelimlem2}}\right)\\
	&=\pa{24k-1}{24}^t\frac{2}{\sqrt{\pi}\alpha_k}\sqrt{t}\left|\sin\left(\alpha_k\right)\right|\left(1+\frac{\alpha_k^2 \cdot\sinh\left(\alpha_k\right)+4C_2(k)}{\left|\sin\left(\alpha_k\right)\right|\cdot 4t}\right)\left(1+\frac{1}{2t}\right)\\
	&\le\pa{24k-1}{24}^t\frac{2\sqrt{t}}{\sqrt{\pi}\alpha_k}\left|\sin\left(\alpha_k\right)\right|\left(1+\frac{C_2^*(k)}{t}\right)\ \ \left(\text{as}\ t\ge 1\right),
	\end{align*}
	which concludes the proof of Lemma \ref{BPRSshiftlem4}.
\end{proof}
\begin{remark}\label{BPRSshiftremark1}
In Lemmas \ref{BPRSshiftlem3} and \ref{BPRSshiftlem4} we have seen that the factors $|\cos(\alpha_k)|$ and $|\sin(\alpha_k)|$ appear in the estimates of upper bounds. Following the definition of $\a_k$ in \eqref{BPRSshiftdef1}, it is a routine check to verify that $\sin(\alpha_k)\ne0$ and $\cos(\alpha_k)\ne 0$ for all $k\in \mathbb{N}$ which, in turn, justifies that both $C_1^*(k)$ and $C_2^*(k)$ (cf. \eqref{BPRSshiftlem3eqn0} and \eqref{BPRSshiftlem4eqn0}) are well-defined.
\end{remark}
Now, applying the bounds from Lemmas \ref{BPRSshiftlem3} and \ref{BPRSshiftlem4}, we now estimate an upper bound for the absolute value of $\sum_{t=N+1}^{\infty}\frac{\omega_k^{[1]}(t)}{n^{\frac t2}}$.
\begin{lemma}\label{BPRSshiftlem5}
Let $(k,N)\in \mathbb{N}^2$ and $\omega_k^{[1]}(t)$ be as in \eqref{Mainthmpart1eqn2}. Let $\widehat{g}(m)$ be as in \Cref{BPRSthm1}. Then for all $n\ge \underset{k,N\ge 1}{\max}\{24k-1,\widehat{g}(N+1)\}$,
\[
\left|\sum_{t=N+1}^{\infty}\frac{\omega_k^{[1]}(t)}{n^{\frac t2}}\right| \le E_{N,1}^{[1]}(k)\cdot n^{-\frac{N+1}{2}},
\]
where $E_{N,1}^{[1]}(k)$ is defined below in \eqref{BPRSshiftlem5eqn6}.
\end{lemma}
\begin{proof}
We split the series depending on $t$ odd or even:
\begin{equation*}
\sum_{t=N+1}^{\infty}\frac{\omega_k^{[1]}(t)}{n^{\frac t2}}=\underset{t\equiv\ 0\ \left(\text{mod}\ 2\right)}{\sum_{t=N+1}^{\infty}}\frac{\omega_k^{[1]}(t)}{n^{\frac t2}}+\underset{t\equiv\ 1\ \left(\text{mod}\ 2\right)}{\sum_{t=N+1}^{\infty}}\frac{\omega_k^{[1]}(t)}{n^{\frac t2}}=\sum_{t\ge \frac{N+1}{2}}\frac{\omega_k^{[1]}(2t)}{n^t}+\sum_{t\ge \frac{N}{2}}\frac{\omega_k^{[1]}(2t+1)}{n^{t+\frac 12}},
\end{equation*}
and so,
\begin{equation}\label{BPRSshiftlem5eqn1}
\left|\sum_{t=N+1}^{\infty}\frac{\omega_k^{[1]}(t)}{n^{\frac t2}}\right|\le \sum_{t\ge \frac{N+1}{2}}\frac{\left|\omega_k^{[1]}(2t)\right|}{n^t}+n^{-\frac 12}\sum_{t\ge \frac{N}{2}}\frac{\left|\omega_k^{[1]}(2t+1)\right|}{n^t}.
\end{equation}
Recalling \eqref{BPRSshiftlem4eqn0}, we set
\begin{equation}\label{BPRSshiftlem5def1}
C_e(k):=2C^*_2(k)+\frac{(24k-1)(1+C^*_2(k))}{6}.
\end{equation}
First, we estimate
\begin{align}\nonumber
&\sum_{t\ge \frac{N+1}{2}}\frac{\left|\omega_k^{[1]}(2t)\right|}{n^t}\\\nonumber
&\le \frac{2\left|\sin\left(\alpha_k\right)\right|}{\sqrt{\pi}\alpha_k}\left(1+\frac{2C_2^*(k)}{N+1}\right)\sum_{t\ge \frac{N+1}{2}}\pa{24k-1}{24n}^t\sqrt{t}
\ \left(\text{by\ \Cref{BPRSshiftlem4}}\right)\\\nonumber
&=\frac{\sqrt{2}\left|\sin\left(\alpha_k\right)\right|}{\sqrt{\pi}\alpha_k}\left(1+\frac{2C_2^*(k)}{N+1}\right)\pa{24k-1}{24}^{\frac{N+1}{2}}\sqrt{N+1}n^{-\frac{N+1}{2}}\sum_{t\ge 0}\pa{24k-1}{24n}^t\sqrt{1+\frac{2t}{N+1}}\\\nonumber
&\le\frac{\sqrt{2}\left|\sin\left(\alpha_k\right)\right|}{\sqrt{\pi}\alpha_k}\left(1+\frac{2C_2^*(k)}{N+1}\right)\pa{24k-1}{24}^{\frac{N+1}{2}}\sqrt{N+1}n^{-\frac{N+1}{2}}\sum_{t\ge 0}\pa{24k-1}{24n}^t\sqrt{1+t}\\\nonumber
&\hspace{10 cm} \left(\text{since for}\ N\ge 1, \frac{2t}{N+1}\le t\right)\\\nonumber
&\le \frac{\sqrt{2}\left|\sin\left(\alpha_k\right)\right|}{\sqrt{\pi}\alpha_k}\left(1+\frac{2C_2^*(k)}{N+1}\right)\pa{24k-1}{24}^{\frac{N+1}{2}}\sqrt{N+1}n^{-\frac{N+1}{2}}\sum_{t\ge 0}\pa{24k-1}{12n}^t\\\nonumber
&\hspace{10.5 cm}\left(\text{as}\ \sqrt{1+t}\le 2^t\ \text{for}\ t\ge 0\right)\\\nonumber
&\le \frac{\sqrt{2}\left|\sin\left(\alpha_k\right)\right|}{\sqrt{\pi}\alpha_k}\left(1+\frac{2C_2^*(k)}{N+1}\right)\pa{24k-1}{24}^{\frac{N+1}{2}}\sqrt{N+1}\left(1+\frac{24k-1}{6n}\right)n^{-\frac{N+1}{2}}\\\nonumber
&\hspace{2 cm} \left(\text{since for}\ 0<x\le \frac 12,\  (1-x)^{-1}\le 1+2x\ \text{and choose}\ x\mapsto \frac{24k-1}{12n}\ \text{as}\ n\ge 24k-1\right)\\\nonumber
&\le \frac{\sqrt{2}\left|\sin\left(\alpha_k\right)\right|}{\sqrt{\pi}\alpha_k}\left(1+\frac{2C_2^*(k)}{N+1}\right)\pa{24k-1}{24}^{\frac{N+1}{2}}\sqrt{N+1}\left(1+\frac{24k-1}{6N}\right)n^{-\frac{N+1}{2}}\\\nonumber
&\hspace{3.5 cm}\left(\text{as for all}\ N\in \mathbb{N}, n>\widehat{g}(N+1)\ge N\ \text{by definition of}\ \widehat{g}(m)\ \text{in \Cref{BPRSthm1}} \right)\\\label{BPRSshiftlem5eqn2}
&\le \frac{\sqrt{2}\left|\sin\left(\alpha_k\right)\right|}{\sqrt{\pi}\alpha_k}\pa{24k-1}{24}^{\frac{N+1}{2}}\sqrt{N+1}\left(1+\frac{C_e(k)}{N}\right)n^{-\frac{N+1}{2}}=E_{N,1,e}^{[1]}(k)\cdot n^{-\frac{N+1}{2}}
\end{align}
with 
\begin{equation}\label{BPRSshiftlem5eqn3}
E_{N,1,e}^{[1]}(k):=\frac{\sqrt{2}\left|\sin\left(\alpha_k\right)\right|}{\sqrt{\pi}\alpha_k}\pa{24k-1}{24}^{\frac{N+1}{2}}\sqrt{N+1}\left(1+\frac{C_e(k)}{N}\right).
\end{equation}
Recalling \eqref{BPRSshiftlem3eqn0}, we set
\begin{equation}\label{BPRSshiftlem5def2}
C_o(k):=2C^*_1(k)+\frac{(24k-1)(1+2C^*_1(k))}{6}.
\end{equation}
Next, we have
\begin{align}\nonumber
&n^{-\frac 12}\sum_{t\ge \frac{N+1}{2}}\frac{\left|\omega_k^{[1]}(2t+1)\right|}{n^t}\\\nonumber
&\hspace{-0.5 cm}\le \pa{6}{\pi^3}^{\frac 12}\left|\cos\left(\alpha_k\right)\right|\left(1+\frac{2C_1^*(k)}{N}\right)n^{-\frac 12}\sum_{t\ge \frac{N}{2}}\pa{24k-1}{24n}^t\sqrt{t+1}\ \ \left(\text{by\ \Cref{BPRSshiftlem3}}\right)\\\nonumber
&\hspace{-0.5 cm}=\pa{3}{\pi^3}^{\frac 12}\left|\cos\left(\alpha_k\right)\right|\left(1+\frac{2C_1^*(k)}{N}\right)\pa{24k-1}{24}^{\frac{N}{2}}\sqrt{N+2}n^{-\frac{N+1}{2}}\sum_{t\ge 0}\pa{24k-1}{24n}^t\sqrt{1+\frac{2t}{N+2}}\\\nonumber
&\hspace{-0.5 cm}\le\pa{3}{\pi^3}^{\frac 12}\left|\cos\left(\alpha_k\right)\right|\left(1+\frac{2C_1^*(k)}{N}\right)\pa{24k-1}{24}^{\frac{N}{2}}\sqrt{N+2}n^{-\frac{N+1}{2}}\sum_{t\ge 0}\pa{24k-1}{24n}^t\sqrt{1+t}\\\nonumber
&\hspace{10 cm} \left(\text{since for}\ N\ge 1, \frac{2t}{N+2}\le t\right)\\\nonumber
&\hspace{-0.5 cm}\le\pa{3}{\pi^3}^{\frac 12}\left|\cos\left(\alpha_k\right)\right|\left(1+\frac{2C_1^*(k)}{N}\right)\pa{24k-1}{24}^{\frac{N}{2}}\sqrt{N+2}n^{-\frac{N+1}{2}}\sum_{t\ge 0}\pa{24k-1}{12n}^t\\\nonumber
&\hspace{11 cm}\left(\text{as}\ \sqrt{1+t}\le 2^t\ \text{for}\ t\ge 0\right)\\\nonumber
&\hspace{-0.5 cm}\le \pa{3}{\pi^3}^{\frac 12}\left|\cos\left(\alpha_k\right)\right|\left(1+\frac{2C_1^*(k)}{N}\right)\pa{24k-1}{24}^{\frac{N}{2}}\sqrt{N+2}\left(1+\frac{24k-1}{6n}\right)n^{-\frac{N+1}{2}}\\\nonumber
&\hspace{-0.5 cm}\le \pa{3}{\pi^3}^{\frac 12}\left|\cos\left(\alpha_k\right)\right|\left(1+\frac{2C_1^*(k)}{N}\right)\pa{24k-1}{24}^{\frac{N}{2}}\sqrt{N+2}\left(1+\frac{24k-1}{6N}\right)n^{-\frac{N+1}{2}}\\\nonumber
&\hspace{3.5 cm}\left(\text{as for all}\ N\in \mathbb{N}, n>\widehat{g}(N+1)\ge N\ \text{by definition of}\ \widehat{g}(m)\ \text{in \Cref{BPRSthm1}} \right)\\\label{BPRSshiftlem5eqn4}
& \hspace{-0.5 cm}\le \pa{3}{\pi^3}^{\frac 12}\left|\cos\left(\alpha_k\right)\right|\pa{24k-1}{24}^{\frac{N}{2}}\sqrt{N+2}\left(1+\frac{C_o(k)}{N}\right)n^{-\frac{N+1}{2}}=E_{N,1,o}^{[1]}(k)\cdot n^{-\frac{N+1}{2}}
\end{align}
with 
\begin{equation}\label{BPRSshiftlem5eqn5}
E_{N,1,o}^{[1]}(k):=\pa{3}{\pi^3}^{\frac 12}\left|\cos\left(\alpha_k\right)\right|\pa{24k-1}{24}^{\frac{N}{2}}\sqrt{N+2}\left(1+\frac{C_o(k)}{N}\right).
\end{equation}
Recalling \eqref{BPRSshiftlem5eqn3} and \eqref{BPRSshiftlem5eqn5}, set 
\begin{equation}\label{BPRSshiftlem5eqn6}
E_{N,1}^{[1]}(k):=E_{N,1,e}^{[1]}(k)+E_{N,1,o}^{[1]}(k),
\end{equation}
and then applying \eqref{BPRSshiftlem5eqn2} and \eqref{BPRSshiftlem5eqn4} to  \eqref{BPRSshiftlem5eqn1}, we finally get for $n\ge  \underset{k,N\ge 1}{\max}\{24k-1,\widehat{g}(N+1)\}$,
\[
\left|\sum_{t=N+1}^{\infty}\frac{\omega_k^{[1]}(t)}{n^{\frac t2}}\right|\le E_{N,1}^{[1]}(k)\cdot n^{-\frac{N+1}{2}},
\]
which finishes the proof of \Cref{BPRSshiftlem5}.
\end{proof}
Therefore, applying Lemma \ref{BPRSshiftlem5} to Lemma \ref{BPRSshiftlem2}, we get for $(k,N)\in \mathbb{N}^2$ and $n\ge \underset{k,N\ge 1}{\max}\{24k-1,\widehat{g}(N+1)\}$,
\begin{equation}\label{BPRSshifteqn1}
\frac{\sqrt{12}e^{\mu(n+k)}}{24(n+k)-1}\left(1-\frac{1}{\mu(n+k)}\right)=\frac{e^{\pi\sqrt{2n/3}}}{4n\sqrt{3}}\left(\sum_{t=0}^{N}\frac{\omega_k^{[1]}(t)}{n^{\frac t2}}+O_{\le  E_{N,1}^{[1]}(k)}\left(n^{-\frac{N+1}{2}}\right)\right).
\end{equation}

Now we are ready to prove Theorem \ref{Mainthmpart1}.

\medskip

\noindent\emph{Proof of Theorem \ref{Mainthmpart1}:}
Recalling the definition of $\widehat{g}(m)$ in \Cref{BPRSthm1}, for $(k,N)\in \mathbb{N}^2$ we define
\begin{equation}\label{Mainthmpart1eqn3}
n_N^{[1]}(k):=\underset{(N,k)\in\mathbb{N}^2}{\max}\left\{\widehat{g}(N+1),(24k-1)^2\right\}.
\end{equation}
Applying \eqref{BPRSshifteqn1} to Lemma \ref{BPRSshiftlem1}, we have for $(k,N)\in \mathbb{N}^2$ and for all $n\ge n_N^{[1]}(k)$,
\begin{equation}\label{Mainthmpart1eqn4}
p(n+k)=\frac{e^{\pi\sqrt{2n/3}}}{4n\sqrt{3}}\left(\sum_{t=0}^{N}\frac{\omega_k^{[1]}(t)}{n^{\frac t2}}+O_{\le  E_{N,1}^{[1]}(k)}\left(n^{-\frac{N+1}{2}}\right)\right)+\frac{\sqrt{12}e^{\mu(n+k)}}{24(n+k)-1}\cdot O_{\le 1}\left(\mu(n+k)^{-N-1}\right).
\end{equation}
Finally, it remains to estimate the second factor on the right hand side of \eqref{Mainthmpart1eqn4}. 
\begin{align}\nonumber
&\frac{\sqrt{12}e^{\mu(n+k)}}{24(n+k)-1}\mu(n+k)^{-N-1}\\\nonumber
&=\frac{e^{\pi\sqrt{2n/3}}}{4n\sqrt{3}}\exp\left(\mu(n+k)-\pi\sqrt{2n/3}\right)\left(1+\frac{24k-1}{24n}\right)^{-1}\mu(n+k)^{-N-1}\\\nonumber
&\le \frac{e^{\pi\sqrt{2n/3}}}{4n\sqrt{3}}\exp\left(\mu(n+k)-\pi\sqrt{2n/3}\right)\mu(n+k)^{-N-1}\ \ \left(\text{as}\ 1+\frac{24k-1}{24n}\ge 1\ \text{for}\ (k,n)\in\mathbb{N}^2\right)\\\nonumber
&\le \frac{e^{\pi\sqrt{2n/3}}}{4n\sqrt{3}}\exp\left(\mu(n+k)-\pi\sqrt{2n/3}\right)\pa{6}{\pi\sqrt{24}}^{N+1}n^{-\frac{N+1}{2}}\\\nonumber
&\hspace{5 cm}\left(\text{as}\ \mu(n+k)=\frac{\pi}{6}\sqrt{24n+24k-1}\ge \frac{\pi}{6}\sqrt{24n}\ \text{for}\ (k,n)\in\mathbb{N}^2\right)\\\nonumber
&=\frac{e^{\pi\sqrt{2n/3}}}{4n\sqrt{3}}\exp\left(\pi\sqrt{\frac{2n}{3}}\left(\sqrt{1+\frac{24k-1}{24n}}-1\right)\right)\pa{6}{\pi\sqrt{24}}^{N+1}n^{-\frac{N+1}{2}}\\\nonumber
&\le \frac{e^{\pi\sqrt{2n/3}}}{4n\sqrt{3}}\exp\left(\pi\sqrt{\frac{2n}{3}}\cdot \frac{24k-1}{48n}\right)\pa{6}{\pi\sqrt{24}}^{N+1}n^{-\frac{N+1}{2}}\\\nonumber
&\hspace{1.5 cm}\ \ \left(\text{applying}\ (1+x)^r\le 1+rx\ \text{for}\ 0\le r\le 1\ \text{and}\ x\ge -1\ \ \text{with}\ (x,r)= \left(\frac{24k-1}{24n},\frac 12\right)\right)\\\nonumber
&\le \frac{e^{\pi\sqrt{2n/3}}}{4n\sqrt{3}}\left(1+\frac{\pi}{24}\sqrt{\frac{2}{3n}}(24k-1)\right)\pa{6}{\pi\sqrt{24}}^{N+1}n^{-\frac{N+1}{2}}\\\nonumber
&\hspace{0.5 cm} \left(\text{applying}\ e^x<1+2x\ \text{for}\ 0<x<1\ \text{with}\ x= \ \frac{\pi\sqrt{2}(24k-1)}{48\sqrt{3n}}\ \text{and}\ 0<x<1\ \text{as}\  n\ge (24k-1)^2\right)\\\label{Mainthmpart1eqn5}
&\le \frac{e^{\pi\sqrt{2n/3}}}{4n\sqrt{3}}\left(1+\frac{3.7\cdot k}{N}\right)\pa{6}{\pi\sqrt{24}}^{N+1}n^{-\frac{N+1}{2}}\ \left(\text{as}\ n\ge \widehat{g}(N+1)\ge \frac 12 N^2\right).
\end{align}
Recalling \eqref{BPRSshiftlem5eqn6}, we set
\begin{equation}\label{Mainthmpart1eqn6}
E_N^{[1]}(k):=E_{N,1}^{[1]}(k)+\left(1+\frac{3.7\cdot k}{N}\right)\pa{6}{\pi\sqrt{24}}^{N+1}.
\end{equation}
Finally, combining \eqref{Mainthmpart1eqn4} and \eqref{Mainthmpart1eqn5} concludes the proof of \Cref{Mainthmpart1}.
\qed

\section{Asymptotics of $\frac{1}{p(n)}$}\label{Sec3}
This section is devoted to provide an asymptotic expansion for the inverse of the partition function $1/p(n)$ along with computations for error bounds to obtain a result similar to~Theorem~\ref{Mainthmpart1}. We shall state the main result later (cf.\ \Cref{Mainthmpart2}) as it involves certain technical parameters which we will define later. 

\begin{lemma}\label{BPRSinvlem1}
Let $\widehat{g}(m)$ be as in Theorem \ref{BPRSthm1} and $N\in \mathbb{N}$. Then for all $n\ge \widehat{g}(N+1)$, 
\[
\frac{1}{p(n)}=\frac{24n-1}{\sqrt{12}}e^{-\mu(n)}\left(1-\frac{1}{\mu(n)}\right)^{-1}+4n\sqrt{3}\ e^{-\pi\sqrt{2n/3}}\cdot O_{\le E^{[2]}_{N,2}}\left(n^{-\frac{N+1}{2}}\right),
\]
where
\begin{equation}\label{BPRSinvlem1finalerror}
E^{[2]}_{N,2}:=\pa{6}{\pi\sqrt{24}}^{N+1}\left(1+\frac{4}{N}\right).
\end{equation} 
\end{lemma}
\begin{proof}
Applying Theorem \ref{BPRSthm1} with $m= N+1$ and $N\in \mathbb{N}$, for all $n>\widehat{g}(N+1)$ it follows that
\begin{align}\nonumber
&\frac{24n-1}{\sqrt{12}}e^{-\mu(n)}\left(1-\frac{1}{\mu(n)}\right)^{-1}\left(1+\mu(n)^{-N-1}\left(1-\frac{1}{\mu(n)}\right)^{-1}\right)^{-1}<\frac{1}{p(n)}\\\label{BPRSinvlem1eqn1}
&<\frac{24n-1}{\sqrt{12}}e^{-\mu(n)}\left(1-\frac{1}{\mu(n)}\right)^{-1}\left(1-\mu(n)^{-N-1}\left(1-\frac{1}{\mu(n)}\right)^{-1}\right)^{-1}.
\end{align}
Next, for all $N\in \mathbb{N}$ this gives
\begin{align}\nonumber
\left(1+\mu(n)^{-N-1}\left(1-\frac{1}{\mu(n)}\right)^{-1}\right)^{-1}& \ge 1-\mu(n)^{-N-1}\left(1-\frac{1}{\mu(n)}\right)^{-1}\\\nonumber
&\hspace{-3.25 cm} \left(\text{applying}\ (1+x)^{-1}>1-x\ \text{for}\ x>0\ \text{with}\ x= \mu(n)^{-N-1}\left(1-\frac{1}{\mu(n)}\right)^{-1}\right)\\\nonumber
&\ge 1-\pa{6}{\pi\sqrt{24}}^{\frac{N+1}{2}}n^{-\frac{N+1}{2}}\left(1-\frac{1}{24n}\right)^{-\frac{N+1}{2}}\left(1+\frac{1}{\sqrt{2n}}\right)\\\nonumber
&\hspace{2 cm}\left(\text{as}\ \left(1-\frac{1}{\mu(n)}\right)^{-1}\le 1+\frac{1}{\sqrt{2n}}\ \text{for}\ n\in\mathbb{N}\right)\\\nonumber
&\ge  1-\pa{6}{\pi\sqrt{24}}^{\frac{N+1}{2}}n^{-\frac{N+1}{2}}\left(1-\frac{N+1}{48n}\right)^{-1}\left(1+\frac{1}{\sqrt{2n}}\right)\\\nonumber
&\hspace{-5 cm} \left(\text{applying}\ (1+x)^{r}\ge 1+rx\ \text{for}\ (r,x)\in \mathbb{R}_{\ge 1}\times\mathbb{R}_{>-1}\ \text{with}\ \left(r,x\right)= \left(\frac{N+1}{2},-\frac{1}{24n}\right)\right)\\\nonumber
&\ge 1-\pa{6}{\pi\sqrt{24}}^{\frac{N+1}{2}}n^{-\frac{N+1}{2}}\left(1-\frac{1}{12N}\right)^{-1}\left(1+\frac{1}{N}\right)\\\nonumber
&\hspace{-1 cm}\left(\text{as for}\ n>\widehat{g}(N+1)\ge \frac{N^2}{2}, \sqrt{2n}\ge N\ \text{and}\ \frac{N+1}{48n}\le \frac{1}{12N}\right)\\\label{BPRSinvlem1eqn2}
&>1-\pa{6}{\pi\sqrt{24}}^{N+1}\left(1+\frac{2}{N}\right)n^{-\frac{N+1}{2}}.
\end{align}
Similarly, we obtain for $n\in \mathbb{Z}_{\ge 4}$ and $N\in \mathbb{N}$, 
\begin{equation}\label{BPRSinvlem1eqn3}
\left(1-\mu(n)^{-N-1}\left(1-\frac{1}{\mu(n)}\right)^{-1}\right)^{-1}<1+\pa{6}{\pi\sqrt{24}}^{N+1}\left(1+\frac{3}{N}\right)n^{-\frac{N+3}{2}}.
\end{equation}
Define
\begin{equation}\label{BPRSinvlem1errordef1}
E^{[2]}_{N,1}:=\pa{6}{\pi\sqrt{24}}^{N+1}\left(1+\frac{3}{N}\right).
\end{equation}
Applying \eqref{BPRSinvlem1eqn2} and \eqref{BPRSinvlem1eqn3} to \eqref{BPRSinvlem1eqn1} gives
\begin{align*}
\frac{1}{p(n)}&=\frac{24n-1}{\sqrt{12}}e^{-\mu(n)}\left(1-\frac{1}{\mu(n)}\right)^{-1}\left(1+O_{\le E^{[2]}_{N,1}}\left(n^{-\frac{N+1}{2}}\right)\right)\\
&=\frac{24n-1}{\sqrt{12}}e^{-\mu(n)}\left(1-\frac{1}{\mu(n)}\right)^{-1}+\frac{24n-1}{\sqrt{12}}e^{-\mu(n)}\left(1-\frac{1}{\mu(n)}\right)^{-1}\cdot O_{\le E^{[2]}_{N,1}}\left(n^{-\frac{N+1}{2}}\right)\\
&=\frac{24n-1}{\sqrt{12}}e^{-\mu(n)}\left(1-\frac{1}{\mu(n)}\right)^{-1}\\
&\hspace{1 cm}+4n\sqrt{3}e^{-\pi\sqrt{2n/3}}\left(1-\frac{1}{24n}\right)\exp\left(-\pi\sqrt{2n/3}\left(\sqrt{1-\frac{1}{24n}}-1\right)\right)\cdot O_{\le E^{[2]}_{N,1}}\left(n^{-\frac{N+1}{2}}\right).
\end{align*}
Consequently, using the fact that for all $N\in \mathbb{N}$, $n\ge \widehat{g}(N+1)>\frac 12N^2$, we can bound the term involving the factor $n^{-\frac{N+1}{2}}$ as
\begin{align*}
&E^{[2]}_{N,1}\left(1-\frac{1}{24n}\right)\exp\left(-\pi\sqrt{2n/3}\left(\sqrt{1-\frac{1}{24n}}-1\right)\right)n^{-\frac{N+1}{2}}\\
&=\pa{6}{\pi\sqrt{24}}^{N+1}\left(1+\frac{3}{N}\right)\left(1-\frac{1}{24n}\right)\exp\left(-\pi\sqrt{2n/3}\left(\sqrt{1-\frac{1}{24n}}-1\right)\right)n^{-\frac{N+1}{2}}\ \ \left(\text{by}\ \eqref{BPRSinvlem1errordef1}\right)\\
&\le \pa{6}{\pi\sqrt{24}}^{N+1}\left(1+\frac{3}{N}\right)\exp\left(-\pi\sqrt{2n/3}\left(\sqrt{1-\frac{1}{24n}}-1\right)\right)n^{-\frac{N+1}{2}}\\
&\le \pa{6}{\pi\sqrt{24}}^{N+1}\left(1+\frac{3}{N}\right)\left(1+\frac{1}{4\sqrt{2n}}\right)n^{-\frac{N+1}{2}}\le \pa{6}{\pi\sqrt{24}}^{N+1}\left(1+\frac{3}{N}\right)\left(1+\frac{1}{4N}\right)n^{-\frac{N+1}{2}}\\
&\hspace{11 cm}\ \ \left(\text{as}\ n\ge \widehat{g}(N+1)>\frac 12N^2\right)\\
&\le \pa{6}{\pi\sqrt{24}}^{N+1}\left(1+\frac{4}{N}\right)n^{-\frac{N+1}{2}}=E^{[2]}_{N,2}n^{-\frac{N+1}{2}}.
\end{align*}
Thus we obtain
\begin{equation*}
\frac{1}{p(n)}=\frac{24n-1}{\sqrt{12}}e^{-\mu(n)}\left(1-\frac{1}{\mu(n)}\right)^{-1}+4n\sqrt{3}\ e^{-\pi\sqrt{2n/3}}\cdot O_{\le E^{[2]}_{N,2}}\left(n^{-\frac{N+1}{2}}\right),
\end{equation*}
which finishes the proof of \Cref{BPRSinvlem1}.
\end{proof}
Our next task is to derive the Taylor series expansion of the dominant term $\frac{24n-1}{\sqrt{12}}e^{-\mu(n)}\left(1-\frac{1}{\mu(n)}\right)^{-1}$ after extracting the factor $4n\sqrt{3}e^{-\pi\sqrt{2n/3}}$. To do so, we first observe that
\begin{align*}
\frac{24n-1}{\sqrt{12}}e^{-\mu(n)}\left(1-\frac{1}{\mu(n)}\right)^{-1}&=4n\sqrt{3}e^{-\pi\sqrt{2n/3}}\exp\left(-\mu(n)+\pi\sqrt{\frac{2n}{3}}\right)\left(1-\dfrac{1}{24n}\right)\left(1-\dfrac{1}{\mu(n)}\right)^{-1}\\
&=4n\sqrt{3}e^{-\pi\sqrt{2n/3}}\cdot T_1(n)\cdot T_2(n),
\end{align*}
where
\begin{equation*}
T_1(n):=\exp\left(-\mu(n)+\pi\sqrt{\frac{2n}{3}}\right)\ \ \text{and}\ \ T_2(n):=\left(1-\dfrac{1}{24n}\right)\left(1-\dfrac{1}{\mu(n)}\right)^{-1}.
\end{equation*}	
Throughout the rest, set
\begin{equation}\label{BPRSinvdef1}
\alpha:=\dfrac{\pi}{6},
\end{equation}
and define 
\[(a)_k:=
\begin{cases}
a(a+1)\cdots(a+k-1),&\quad \text{if}\ k\ge 1\\
1,&\quad \text{if}\ k=0
\end{cases}.
\]
Define for all $t\in \mathbb{Z}_{\ge 0}$,
\begin{equation}\label{BPRSinvdef2a}
E_1(t):=
\begin{cases}
1,  &\quad\ \text{if}\ t=0\\
\left(\dfrac{-1}{24}\right)^t\frac{\left(\frac{1}{2}-t\right)_{t+1}}{t}\displaystyle\sum_{u=1}^{t}\frac{(-1)^u (-t)_u}{(t+u)!(2u-1)!}\alpha^{2u},
&\quad\ \text{if}\ t \geq 1
\end{cases}
,
\end{equation}
and 
\begin{equation}\label{BPRSinvdef2b}
O_1(t):=
\frac{\pi}{12\sqrt{6}}\frac{(-1)^t\left(\frac{1}{2}-t\right)_{t+1}}{24^t}\displaystyle\sum_{u=0}^{t}\frac{(-1)^u (-t)_u}{(t+u+1)!(2u)!}\alpha^{2u}.
\end{equation}

\begin{lemma}\label{BPRSinvlem2}
Let $E_1(t)$ and $O_1(t)$ be as in \eqref{BPRSinvdef2a} and \eqref{BPRSinvdef2b}, respectively. Then we have
	\begin{equation}\label{BPRSinvlem2eqn1}
	T_1(n)=\sum_{t=0}^{\infty}\frac{E_1(t)}{\sqrt{n}^{2t}}+\sum_{t=0}^{\infty}\frac{O_1(t)}{\sqrt{n}^{2t+1}}.
	\end{equation}
\end{lemma}
\begin{proof}
From \cite[Lemma 3.4]{BPRS}, we already know the coefficients from the Taylor series expansion of $T_1(n)^{-1}$ which take the following shape,
\[
\exp\left(\mu(n)-\pi\sqrt{\frac{2n}{3}}\right)=\sum_{t=0}^{\infty}\frac{E_1(t)}{\sqrt{n}^{2t}}-\sum_{t=0}^{\infty}\frac{O_1(t)}{\sqrt{n}^{2t+1}}.
\]
Using the fact that for the coefficient functional $[t^n]\sum_{k\ge 0}a_k t^k:=a_n$, \[[t^{2k}]\left(e^{t}\right)=[t^{2k}]\left(e^{-t}\right)\ \ \text{and}\ \ [t^{2k+1}]\left(e^{t}\right)=-[t^{2k+1}]\left(e^{-t}\right),\ k\ge 0,
\]
we conclude \eqref{BPRSinvlem2eqn1}.
\end{proof}
For $t\in \mathbb{Z}_{\ge 0}$, define
\begin{equation}\label{BPRSinvdef3}
e_2(t):=
\begin{cases}
1, &\quad\ \text{if}\ t=0\\
\frac{36}{\pi^2+36}\left(\frac{\left(1+\alpha^{-2}\right)}{24}\right)^t
, &\quad\ \text{if}\ t \geq 1
\end{cases},\ \ \text{and}\ \ o_2(t):=\dfrac{6}{\pi\sqrt{24}}\left(\dfrac{-1}{24}\right)^t\sum_{m=0}^{t}\left(-\alpha^{-2}\right)^m\binom{-\frac{2m+1}{2}}{t-m}.
\end{equation}
Following \eqref{BPRSinvdef3}, we set
	\begin{equation}\label{BPRSinvdef4}
E_2(t):=
\begin{cases}
1,  &\quad\ \text{if}\ t=0\\
\frac{36-\pi^2}{24\pi^2}, &\quad\ \text{if}\ t=1\\
\frac{54}{\pi^4\left(1+\alpha^{-2}\right)
}\left(\frac{\left(1+\alpha^{-2}\right)}{24}\right)^{t-1}
, &\quad\ \text{if}\ t\geq 2
\end{cases}
,\ \ \text{and}\ \ O_2(t):=
\begin{cases}
o_2(t),  &\quad\ \text{if}\ t=0\\
o_2(t)-\dfrac{o_2(t-1)}{24}, &\quad\ \text{if}\ t \geq 1.
\end{cases}
\end{equation}
\begin{lemma}\label{BPRSinvlem3}
	Let $E_2(t)$ and $O_2(t)$ be as in \eqref{BPRSinvdef3}. Then
	\begin{equation}\label{BPRSinvlem3eqn1}
	T_2(n)=\sum_{t=0}^{\infty}\frac{E_2(t)}{\sqrt{n}^{2t}}+\sum_{t=0}^{\infty}\frac{O_2(t)}{\sqrt{n}^{2t+1}}.
	\end{equation}
\end{lemma}
\begin{proof}
 Taylor series expansion gives
\begin{align*}
\left(1-\frac{1}{\mu(n)}\right)^{-1}&=\sum_{m=0}^{\infty}\pa{6}{\pi\sqrt{24}}^mn^{-\frac m2}\left(1-\frac{1}{24n}\right)^{-\frac m2}\\
&=\sum_{m=0}^{\infty}\sum_{\ell=0}^{\infty}\pa{6}{\pi\sqrt{24}}^m\binom{-\frac m2}{\ell}\pa{-1}{24}^{\ell}\pa{1}{\sqrt{n}}^{m+2\ell}\\
&=\sum_{t=0}^{\infty}\sum_{m=0}^{t}\pa{6}{\pi\sqrt{24}}^{2m}\binom{-m}{t-m}\pa{-1}{24}^{t-m}\pa{1}{\sqrt{n}}^{2t}\\
&\hspace{3 cm}+\sum_{t=0}^{\infty}\sum_{m=0}^{t}\pa{6}{\pi\sqrt{24}}^{2m+1}\binom{-\frac{2m+1}{2}}{t-m}\pa{-1}{24}^{t-m}\pa{1}{\sqrt{n}}^{2t+1}\\
&=\sum_{t=0}^{\infty}\frac{e_2(t)}{\sqrt{n}^{2t}}+\sum_{t=0}^{\infty}\frac{o_2(t)}{\sqrt{n}^{2t+1}}.
\end{align*}
Hence, it follows that
\begin{align*}
T_2(n)&= \left(1-\frac{1}{24n}\right)\left(1-\frac{1}{\mu(n)}\right)^{-1}=\left(1-\frac{1}{24n}\right)\left(\sum_{t=0}^{\infty}\frac{e_2(t)}{\sqrt{n}^{2t}}+\sum_{t=0}^{\infty}\frac{o_2(t)}{\sqrt{n}^{2t+1}}\right)\\
&=\sum_{t=0}^{\infty}\frac{e_2(t)}{\sqrt{n}^{2t}}+\sum_{t=0}^{\infty}\frac{o_2(t)}{\sqrt{n}^{2t+1}}-\frac{1}{24}\sum_{t=0}^{\infty}\frac{e_2(k)}{\sqrt{n}^{2t+2}}-\frac{1}{24}\sum_{t=0}^{\infty}\frac{o_2(t)}{\sqrt{n}^{2t+3}}\\
&=\sum_{t=0}^{\infty}\frac{E_2(t)}{\sqrt{n}^{2t}}+\sum_{t=0}^{\infty}\frac{O_2(t)}{\sqrt{n}^{2t+1}}.
\end{align*}
In the last step we used the fact that $E_2(k)=e_2(k)-\dfrac{e_2(k-1)}{24}$ for all $k\ge 2$. This finishes the proof of \Cref{BPRSinvlem3}.
\end{proof}
\begin{definition}\label{BPRSinvdef5}
For $t\in \mathbb{Z}_{\ge 0}$, define
\begin{equation}\label{BPRSinvdef5a}
g_{e,1}(t):=
\begin{cases}
1, &\quad\ \text{if}\ t=0\\
\dfrac{\pi^4-288\pi^2+10368}{6912\pi^2}, &\quad\ \text{if}\ t=1\\
S_1(t)+\frac{3(1-\alpha^2)}{2\pi^2}S_1(t-1)+\frac{3}{2\pi^2(1+\alpha^2)}\pa{\left(1+\alpha^{-2}\right)}{24}^{t-1}\left(1+S_2(t)\right), &\quad\ \text{if}\ t\ge 2
\end{cases}
\end{equation}
with
\begin{equation}\label{BPRSinvdef5b}
S_1(t):=\pa{-1}{24}^t\frac{\left(\frac 12-t\right)_{t+1}}{t}\sum_{u=1}^{t}\frac{(-1)^u (-t)_u}{(t+u)!(2u-1)!}\alpha^{2u}
\end{equation}
and 
\begin{equation}\label{BPRSinvdef5c} S_2(t):=\sum_{s=1}^{t-2}\frac{(-\left(1+\alpha^{-2}\right))^{-s}\left(\frac{1}{2}-s\right)_{s+1}}{s}\sum_{u=1}^{s}\frac{(-1)^u(-s)_u}{(s+u)!(2u-1)!}\alpha^{2u}.
\end{equation}
\end{definition}
\begin{definition}\label{BPRSinvdef6}
	For $t\in \mathbb{Z}_{\ge 0}$, define
	\begin{equation}\label{BPRSinvdef6a}
	g_{e,2}(t):=\frac{1}{(24)^t}\left(\frac{1}{1+\alpha^2}S_3(t)-\frac{8}{\left(1+\alpha^{-2}\right)}S_4(t)+S_5(t)\right) 
	\end{equation}
	with
	\begin{equation}\label{BPRSinvdef6b}
	S_3(t):=(-1)^{t-1}\sum_{s=0}^{t-2}\left(\frac{1}{2}-s\right)_{s+1}\sum_{u=0}^{s}\frac{(-1)^u(-s)_u}{(s+u+1)!(2u)!}\alpha^{2u}\sum_{r=0}^{t-s-1}\left(\frac{-1}{\alpha^2}\right)^r\binom{-\frac{2r+1}{2}}{t-s-r-1},
	\end{equation}
	\begin{equation}\label{BPRSinvdef6c}
	S_4(t):=\sum_{s=0}^{t-2}\frac{(-1)^s\left(\frac{1}{2}-s\right)_{s+1}}{4^{t-s}(2t-2s-3)}\binom{2t-2s-3}{t-s-1}\sum_{u=0}^{s}\frac{(-1)^u(-s)_u}{(s+u+1)!(2u)!}\alpha^{2u},
	\end{equation}
	and
	\begin{equation}\label{BPRSinvdef6d}
	S_5(t):=(-1)^{t-1}\left(\frac 32-t\right)_t\sum_{u=0}^{t-1}\frac{(-1)^u(-t+1)_u}{(t+u)!(2u)!}\alpha^{2u}.
	\end{equation}
\end{definition}
\begin{definition}\label{BPRSinvdef7}
Let $S_1(t)$ be as in \eqref{BPRSinvdef5b}.	For $t\in \mathbb{Z}_{\ge 0}$, define
\begin{equation}\label{BPRSinvdef7a}
g_{o,1}(t):=
\begin{cases}
\frac{\sqrt{6}}{2\pi}, &\quad\ \text{if}\ t=0\\
\frac{\pi^4-144\pi^2+10368}{2306\sqrt{6}\pi^3}, &\quad \ \text{if}\ t=1\\
\frac{\sqrt{6}}{2\pi}\left(S_1(t)+\frac{\pa{1}{24}^t}{1+\alpha^2}\left(S_6(t)+S_7(t)\right)-\frac{2}{\left(1+\alpha^{-2}\right)\cdot 96^t}\frac{\smallbinom{2t-1}{t}}{2t-1}-\frac{2\alpha^2}{\left(1+\alpha^{2}\right)\cdot 24^t}S_8(t)\right), &\quad\ \text{if}\ t\ge 2
\end{cases},
\end{equation}
	with 
	\begin{equation}\label{BPRSinvdef7b}
	S_6(t):=(-1)^t\sum_{s=0}^{t}\pa{-1}{\alpha^2}^s\binom{-\frac{2s+1}{2}}{t-s},
	\end{equation}
	\begin{equation}\label{BPRSinvdef7c}
S_7(t):=(-1)^t\sum_{s=1}^{t-1}\frac{\left(\frac{1}{2}-s\right)_{s+1}}{s}\sum_{u=1}^{s}\frac{(-1)^u(-s)_u}{(s+u)!(2u-1)!}\alpha^{2u}\sum_{r=0}^{t-s}\left(\frac{-1}{\alpha^2}\right)^r\binom{-\frac{2r+1}{2}}{t-s-r},
	\end{equation}
	and
	\begin{equation}\label{BPRSinvdef7d}
	S_8(t):=\sum_{s=1}^{t-1}\frac{(-1)^s\left(\frac{1}{2}-s\right)_{s+1}\binom{2t-2s-1}{t-s}}{s\cdot 4^{t-s}(2t-2s-1)}\sum_{u=1}^{s}\frac{(-1)^u(-s)_u}{(s+u)!(2u-1)!}\alpha^{2u}.
	\end{equation}
\end{definition}
\begin{definition}\label{BPRSinvdef8}
Let $S_5(t)$ be as in \eqref{BPRSinvdef6d}. For $t\in \mathbb{Z}_{\ge 0}$, define
\begin{equation}\label{BPRSinvdef8a}
g_{o,2}(t):=
\begin{cases}
\frac{\pi}{24\sqrt{6}},&\quad\ \text{if}\ t=0\\
\frac{\pi}{12\sqrt{6}}\pa{1}{24}^t\left(\frac{(1+\a^{-2})^t}{(1+\a^2)^2}S_9(t)+\frac{1-\a^2}{\a^2}S_5(t)+S_5(t+1)\right), &\quad\ \text{if}\ t\ge 1
\end{cases}
\end{equation}
with
\begin{equation}\label{BPRSinvdef8b}
S_9(t):=\sum_{s=0}^{t-2}\pa{-1}{\left(1+\alpha^{-2}\right)}^s\left(\frac{1}{2}-s\right)_{s+1}\sum_{u=0}^{s}\frac{(-1)^u(-s)_u}{(s+u+1)!(2u)!}\alpha^{2u}.
\end{equation}	
\end{definition}
\begin{definition}\label{BPRSinvdef9}
	For $j \in \{1,2\}$, let $g_{e,j}(t)$ and $g_{o,j}(t)$ be as in Definitions \ref{BPRSinvdef5} to \ref{BPRSinvdef8}. For $t\in \mathbb{Z}_{\ge 0}$, define
	\begin{equation*}
	g_e(t):=g_{e,1}(t)+g_{e,2}(t)\ \ \text{and}\ \ g_o(t):=g_{o,1}(t)+g_{o,2}(t).
	\end{equation*}
\end{definition}

\begin{lemma}\label{BPRSinvlem4}
	For $t\in \mathbb{Z}_{\ge 0}$, let $g_e(t)$ and $g_o(t)$ be as in Definition \ref{BPRSinvdef9}. Then
	\begin{equation}\label{BPRSinvlem4eqn1}
	\frac{24n-1}{\sqrt{12}}e^{-\mu(n)}\left(1-\frac{1}{\mu(n)}\right)^{-1}=4n\sqrt{3}\ e^{-\pi\sqrt{2n/3}}\sum_{t=0}^{\infty}\frac{g(t)}{n^{\frac t2}},
	\end{equation}
	where 
	\begin{equation}\label{BPRSinvlem4eqn2}
	g(2t)=g_{e}(t)\ \ \text{and}\ \ g(2t+1)=g_{o}(t).
	\end{equation}
\end{lemma}
\begin{proof}
 Applying the Cauchy product for the power series representations of $T_1(n)$ and $T_2(n)$ from Lemmas \ref{BPRSinvlem2} and \ref{BPRSinvlem3}, we obtain the convoluted coefficients $(g(2t))_{t\ge 0}$ and $(g(2t+1))_{t\ge 0}$.
\end{proof}
Next, we estimate an error bound for the absolute value of remainder terms $\sum_{t\ge N+1}g(t)/n^{\frac t2}$ for $N\ge 1$. In order to do so, we need to estimate bounds for $|g_e(t)|$ and $|g_o(t)|$ for $t\ge 2$ and therefore, following Definitions \ref{BPRSinvdef5} to \ref{BPRSinvdef8}, we see that deriving estimates for the sums $(S_j(t))_{1\le j\le 9}$ are prerequisites. 
\begin{lemma}\label{BPRSinvlem5}
For all $t\in \mathbb{Z}_{\ge 2}$,
\begin{align}\label{BPRSinvlem5eqn1}
&S_1(t)=\frac{\alpha\sinh(\alpha)}{2\sqrt{\pi}}\frac{1}{24^t\cdot t^{\frac 32}}\left(1+O_{\le 2.6}\left(\frac 1t\right)\right),\\\label{BPRSinvlem5eqn2}
&S_2(t)=\left(\cosh\left(\sqrt{1+\a^2}-1\right)-1\right)\left(1+O_{\le 54.9}\left(\frac 1t\right)\right),\\\label{BPRSinvlem5eqn3}
&S_3(t)=\frac{\sinh\left(\sqrt{1+\a^2}-1\right)}{1+\a^2}\pa{1+\a^2}{\a^2}^t\left(1+O_{\le 15.3}\left(\frac 1t\right)\right),\\\label{BPRSinvlem5eqn4}
&S_4(t)=\frac{\alpha\cosh(\alpha)+\sinh(\alpha)}{16\sqrt{\pi}\cdot \alpha}\frac{1}{t^{\frac 32}}\left(1+O_{\le 6.7}\left(\frac 1t\right)\right),\\\label{BPRSinvlem5eqn5}
&S_5(t)=\frac{\cosh(\alpha)}{2\sqrt{\pi}\cdot t^{\frac 32}}\left(1+O_{\le 1.2}\left(\frac 1t\right)\right),\\\label{BPRSinvlem5eqn6}
&S_6(t)=\frac{1}{(1+\alpha^2)^{\frac 12}}\pa{1+\a^2}{\a^2}^t\left(1+O_{\le 0.2}\left(\frac 1t\right)\right),\\\label{BPRSinvlem5eqn7}
&S_7(t)=\frac{\cosh\left(\sqrt{1+\a^2}-1\right)-1}{\sqrt{1+\a^2}}\pa{1+\a^2}{\a^2}^t\left(1+O_{\le 14}\left(\frac 1t\right)\right),\\\label{BPRSinvlem5eqn8}
&S_8(t)=\frac{\alpha\sinh(\alpha)+\cosh(\alpha)-1}{4\sqrt{\pi}\cdot t^{\frac 32}}\left(1+O_{\le 9}\left(\frac 1t\right)\right),\\\label{BPRSinvlem5eqn9}
&S_9(t)=\frac{\sqrt{1+\a^2}\sinh\left(\sqrt{1+\a^2}-1\right)}{\a^2}\left(1+O_{\le 7.7}\left(\frac 1t\right)\right).
\end{align}
\end{lemma}
\begin{proof}
See Section \ref{Sec5b}.
\end{proof}
Now, we apply Lemma \ref{BPRSinvlem5} to derive bounds for $g_{e,j}(t)$ and $g_{o,j}(t)$ with $j \in \{1,2\}$ and consequently for $(g(2t))_{t\ge 0}$ and $(g(2t+1))_{t\ge 0}$ in accordance with \Cref{BPRSinvdef9} and \eqref{BPRSinvlem4eqn2}.
\begin{lemma}\label{BPRSinvlem6}
For $t\ge 1$,	
\[
g_{e,1}(t)=\frac{3}{2\pi^2(1+\a^2)}\pa{1+\a^2}{24\a^2}^{t-1}\cosh\left(\sqrt{1+\a^2}-1\right)\left(1+O_{\le 1.2}\left(\frac 1t\right)\right).
\]
\end{lemma}
\begin{proof}
By \eqref{BPRSinvlem5eqn1} with $t\mapsto t-1$, for all $t\ge 3$, we have
\begin{align*}
S_1(t-1)
&=\frac{\alpha\sinh(\alpha)}{2\sqrt{\pi}}\frac{24}{24^{t}\cdot t^{\frac 32}}\left(1-\frac 1t\right)^{-\frac 32}\left(1+O_{\le 2.6}\left(\frac{1}{t-1}\right)\right)\\
&=\frac{\alpha\sinh(\alpha)}{2\sqrt{\pi}}\frac{24}{24^{t}\cdot t^{\frac 32}}\left(1+O_{\le 3.7}\left(\frac 1t\right)\right)\left(1+O_{\le 5.2}\left(\frac 1t\right)\right)\\
&=\frac{\alpha\sinh(\alpha)}{2\sqrt{\pi}}\frac{24}{24^{t}\cdot t^{\frac 32}}\left(1+O_{\le 19}\left(\frac 1t\right)\right),
\end{align*} 
which implies that
\begin{equation}\label{BPRSinvlem6eqn1}
\frac{3(1-\a^2)}{2\pi^2}S_1(t-1)=\frac{1-\a^2}{\a^2}\frac{\a\sinh(\a)}{2\sqrt{\pi}}\frac{1}{24^t\cdot t^{\frac 32}}\left(1+O_{\le 19}\left(\frac 1t\right)\right).
\end{equation}
We note the fact that for
\[
f(t)=A_1\cdot g(t)\left(1+O_{\le E_1}\left(\frac 1t\right)\right)\ \ \text{and}\ \ h(t)=A_2\cdot g(t)\left(1+O_{\le E_2}\left(\frac 1t\right)\right)
\]
with $(A_1, A_2)\in \mathbb{R}^2\setminus \{(0,0)\}$, then
\begin{equation}\label{BPRSinvlem6eqn2}
f(t)+h(t)=(A_1+A_2)\cdot g(t)\left(1+O_{\le \left|\frac{A_1\cdot E_1+A_2\cdot E_2}{A_1+A_2}\right|}\left(\frac 1t\right)\right).
\end{equation}
Combining \eqref{BPRSinvlem5eqn1} and \eqref{BPRSinvlem6eqn1} and then applying \eqref{BPRSinvlem6eqn2} with 
\begin{align*}
&\left(f(t),g(t),h(t), A_1, A_2, E_1, E_2\right)\\
&\hspace{1 cm}\mapsto \left(S_1(t), \frac{1}{24^t\cdot t^{\frac 32}},\frac{3(1-\a^2)}{2\pi^2}S_1(t-1),\frac{\alpha\sinh(\alpha)}{2\sqrt{\pi}},\frac{(1-\a^2)\sinh(\a)}{\a\cdot 2\sqrt{\pi}}, 2.6, 19\right),
\end{align*}
we obtain
\begin{align}\nonumber
S_1(t)+\frac{3(1-\a^2)}{2\pi^2}S_1(t-1)&=\frac{\sinh(\a)}{\a\cdot 2\sqrt{\pi}\cdot 24^t\cdot t^{\frac 32}}\left(1+O_{\le 2.6\a^2+\frac{19(1-\a^2)}{\a}}\left(\frac 1t\right)\right)\\\label{BPRSinvlem6eqn3}
&=\frac{\sinh(\a)}{\a\cdot 2\sqrt{\pi}\cdot 24^t\cdot t^{\frac 32}}\left(1+O_{\le 27.1}\left(\frac 1t\right)\right).
\end{align}
Next, noting the implication
\[
f(t)=C_1\left(1+O_{\le E_1}\left(\frac 1t\right)\right)\ \implies 1+f(t)=(1+C_1)\left(1+O_{\le \frac{E_1\cdot C_1}{C_1+1}}\left(\frac 1t\right)\right),
\]
and applying it with
\[
\left(f(t), C_1, E_1\right)\mapsto \left(S_2(t),\cosh\left(\sqrt{1+\a^2}-1\right)-1, 54.9\right),
\]
one obtains from \eqref{BPRSinvlem5eqn2} that
\begin{align*}
1+S_2(t)&=\cosh\left(\sqrt{1+\a^2}-1\right)\left(1+O_{\le \frac{54.9\left(\cosh\left(\sqrt{1+\a^2}-1\right)-1\right)}{\cosh\left(\sqrt{1+\a^2}-1\right)}} \left(\frac 1t\right)\right)\\
&=\cosh\left(\sqrt{1+\a^2}-1\right)\left(1+O_{\le 0.5} \left(\frac 1t\right)\right),
\end{align*}
and hence, 
\begin{equation}\label{BPRSinvlem6eqn4}
\frac{3}{2\pi^2(1+\a^2)}\pa{1+\a^2}{24\a^2}^{t-1}\left(1+S_2(t)\right)=\frac{3\cosh\left(\sqrt{1+\a^2}-1\right)}{2\pi^2(1+\a^2)}\pa{1+\a^2}{24\a^2}^{t-1}\left(1+O_{\le 0.5} \left(\frac 1t\right)\right).
\end{equation}
Finally, applying \eqref{BPRSinvlem6eqn3} and \eqref{BPRSinvlem6eqn4} to \eqref{BPRSinvdef5a}, it follows that for $t\ge 3$,
\begin{align*}
&g_{e,1}(t)\\
&=\frac{\sinh(\a)}{\a\cdot 2\sqrt{\pi}\cdot 24^t\cdot t^{\frac 32}}\left(1+O_{\le 27.1}\left(\frac 1t\right)\right)+\frac{3\cosh\left(\sqrt{1+\a^2}-1\right)}{2\pi^2(1+\a^2)}\pa{1+\a^2}{24\a^2}^{t-1}\left(1+O_{\le 0.5} \left(\frac 1t\right)\right)\\
&=\frac{3\cosh\left(\sqrt{1+\a^2}-1\right)}{2\pi^2(1+\a^2)}\pa{1+\a^2}{24\a^2}^{t-1}\\
&\hspace{3.5 cm}\cdot\left(1+O_{\le 0.5} \left(\frac 1t\right)+\frac{\pi^{\frac 32}\sinh(\a)(1+\a^2)}{72\a\cosh\left(\sqrt{1+\a^2}-1\right)}\frac{\pa{\a^2}{1+\a^2}^{t-1}}{t^{\frac 32}}\left(1+O_{\le 27.1}\left(\frac 1t\right)\right)\right)\\
&=\frac{3\cosh\left(\sqrt{1+\a^2}-1\right)}{2\pi^2(1+\a^2)}\pa{1+\a^2}{24\a^2}^{t-1}\left(1+O_{\le 1.2}\left(\frac 1t\right)\right).
\end{align*}
We verify the remaining cases $1\le t\le 2$ by numerical checks with computer algebra which altogether concludes the proof.
\end{proof}

\begin{lemma}\label{BPRSinvlem7}
For $t\in \mathbb{Z}_{\ge 1}$,	
\[
g_{e,2}(t)=\frac{\sinh\left(\sqrt{1+\a^2}-1\right)}{(1+\a^2)^2}\pa{1+\a^2}{24\a^2}^{t}\left(1+O_{\le 21}\left(\frac 1t\right)\right).
\]
\end{lemma}
\begin{proof}
Applying \eqref{BPRSinvlem5eqn3}-\eqref{BPRSinvlem5eqn5} to \eqref{BPRSinvdef6a}, we obtain the result for $t\ge 2$; we verify the case $t=1$ by numerical checks with computer algebra which altogether concludes the proof. We omit details of the proof due to its similarity with that of Lemma \ref{BPRSinvlem6}.
\end{proof}
\begin{lemma}\label{BPRSinvlem8}
For $t\in \mathbb{Z}_{\ge 1}$, we have
\[
\left|g(2t)\right|<C_1(\a)\pa{1+\a^2}{24\a^2}^{t-1}\left(1+\frac{3.5}{t}\right),
\]
where
\begin{equation}\label{BPRSinvlem8def}
C_1(\a):=\frac{3\left(\cosh\left(\sqrt{1+\a^2}-1\right)+\sinh\left(\sqrt{1+\a^2}-1\right)\right)}{2\pi^2(1+\a^2)}.
\end{equation}
\end{lemma}
\begin{proof}
Applying Lemmas \ref{BPRSinvlem6} and \ref{BPRSinvlem7} to Definitions \ref{BPRSinvdef9} and \eqref{BPRSinvlem4eqn2}, proves the statement. 
\end{proof}

\begin{lemma}\label{BPRSinvlem9}
For $t\in \mathbb{Z}_{\ge 1}$,	
\[
g_{o,1}(t)=\frac{\sqrt{6}}{2\pi(1+\a^2)^{\frac 32}}\pa{1+\a^2}{24\a^2}^{t}\cosh\left(\sqrt{1+\a^2}-1\right)\left(1+O_{\le 0.7}\left(\frac 1t\right)\right).
\]
\end{lemma}
\begin{proof}
Applying \eqref{BPRSinvlem5eqn1}, \eqref{BPRSinvlem5eqn6}-\eqref{BPRSinvlem5eqn8} and using Lemma \ref{BPRSprelimlem2} to \eqref{BPRSinvdef7a}, we obtain the result for $t\ge 2$ and we verify the case $t=1$ by numerical checks with computer algebra which concludes the proof.
\end{proof}

\begin{lemma}\label{BPRSinvlem10}
For $t\in \mathbb{Z}_{\ge 1}$,
\[
g_{o,2}(t)=\frac{\pi}{12\sqrt{6}\a^2(1+\a^2)^{\frac 32}}\pa{1+\a^2}{24\a^2}^{t}\sinh\left(\sqrt{1+\a^2}-1\right)\left(1+O_{\le 13}\left(\frac 1t\right)\right).
\]
\end{lemma}
\begin{proof}
Applying \eqref{BPRSinvlem5eqn5} and \eqref{BPRSinvlem5eqn9} to \eqref{BPRSinvdef8a}, we obtain the result for $t\ge 2$; we verify the case $t=1$ by numerical checks with computer algebra which altogether concludes the proof.
\end{proof}

\begin{lemma}\label{BPRSinvlem11}
For $t\in \mathbb{Z}_{\ge 1}$, we have
\[
\left|g(2t+1)\right|<C_2(\a)\pa{1+\a^2}{24\a^2}^{t}\left(1+\frac{0.5}{t}\right),
\]
where
\begin{equation}\label{BPRSinvlem11def}
C_2(\a):=\frac{1}{\pi}\sqrt{\frac 32}\frac{\left(\cosh\left(\sqrt{1+\a^2}-1\right)+\sinh\left(\sqrt{1+\a^2}-1\right)\right)}{(1+\a^2)^{\frac 32}}.
\end{equation}
\end{lemma}
\begin{proof}
	Applying Lemmas \ref{BPRSinvlem9} and \ref{BPRSinvlem10} to Definitions \ref{BPRSinvdef9} and \eqref{BPRSinvlem4eqn2} gives \eqref{BPRSinvlem11def}.
\end{proof}
Now we are ready to present both the statement and the proof of the main theorem of this section.
\begin{definition}\label{BPRSinvdef10}
Let $E^{[2]}_{N,2}$ be as in Lemma \ref{BPRSinvlem1finalerror}. Let $C_1(\a)$ and $C_2(\a)$ be as in Lemmas \ref{BPRSinvlem8} and \ref{BPRSinvlem11}, respectively. Then for $N\in \mathbb{N}$, define
\begin{align}\label{BPRSinvdef10eqn1}
E^{[2]}_{N,e}&:=C_1(\a)\pa{1+\a^2}{24\a^2}^{\frac{N-1}{2}}\left(1+\frac{8}{N}\right),\\\label{BPRSinvdef10eqn2}
E^{[2]}_{N,o}&:=C_2(\a)\pa{1+\a^2}{24\a^2}^{\frac{N}{2}}\left(1+\frac{3}{N}\right),\\\label{BPRSinvdef10eqn3}
E^{[2]}_N&:=E^{[2]}_{N,e}+E^{[2]}_{N,o}+E^{[2]}_{N,2}.
\end{align}	
\end{definition}
\begin{theorem}\label{Mainthmpart2}
Let $N\in \mathbb{N}$. Let $\widehat{g}(m)$ be as in Theorem \ref{BPRSthm1}, $(g(t))_{t\ge 0}$ as in \eqref{BPRSinvlem4eqn1}, and $E^{[2]}_N$ as in \eqref{BPRSinvdef10eqn3}. Then for $n>\widehat{g}(N+1)$, we have
\begin{equation}\label{Mainthmpart2eqn0}
\frac{1}{p(n)}=4n\sqrt{3}\ e^{-\pi\sqrt{2n/3}}\left(\sum_{t=0}^{N}\frac{g(t)}{n^{\frac t2}}+O_{\le E^{[2]}_N}\left(n^{-\frac{N+1}{2}}\right)\right).
\end{equation}
\end{theorem}
\begin{proof}
From Lemmas \ref{BPRSinvlem1} and \ref{BPRSinvlem4}, for all $n>\widehat{g}(N+1)$ we obtain
\begin{align}\nonumber
\frac{1}{p(n)}&=4n\sqrt{3}\ e^{-\pi\sqrt{2n/3}}\left(\sum_{t=0}^{N}\frac{g(t)}{n^{\frac t2}}+\sum_{t\ge N+1}\frac{g(t)}{n^{\frac t2}}+O_{\le  E^{[2]}_{N,2}}\left(n^{-\frac{N+1}{2}}\right)\right)\\\label{Mainthmpart2eqn1}
&=4n\sqrt{3}\ e^{-\pi\sqrt{2n/3}}\left(\sum_{t=0}^{N}\frac{g(t)}{n^{\frac t2}}+\underset{=:S_e(n)}{\underbrace{\sum_{t\ge \frac{N+1}{2}}\frac{g(2t)}{n^{t}}}}+\underset{=:S_o(n)}{\underbrace{\sum_{t\ge \frac{N}{2}}\frac{g(2t+1)}{n^{t+\frac 12}}}}+O_{\le  E^{[2]}_{N,2}}\left(n^{-\frac{N+1}{2}}\right)\right).
\end{align}
First, we bound $S_e(n)$,
\begin{align}\nonumber
\left|S_e(n)\right|&\le \sum_{t\ge \frac{N+1}{2}}\frac{|g(2t)|}{n^{t}}\le C_1(\a)\sum_{t\ge \frac{N+1}{2}}\pa{1+\a^2}{24\a^2}^{t-1}\frac{1}{n^t}\left(1+\frac{3.5}{t}\right)\ \ \left(\text{by}\ \text{Lemma}\ \ref{BPRSinvlem8}\right)\\\nonumber
&\le C_1(\a)\left(1+\frac{7}{N+1}\right)\sum_{t\ge \frac{N+1}{2}}\pa{1+\a^2}{24\a^2}^{t-1}\frac{1}{n^t}\ \ \left(\text{as}\ t\ge \frac{N+1}{2}\right)\\\nonumber
&= C_1(\a)\left(1+\frac{7}{N+1}\right)\pa{1+\a^2}{24\a^2}^{\frac{N-1}{2}}n^{-\frac{N+1}{2}}\sum_{t\ge 0}\pa{1+\a^2}{24\a^2\cdot n}^{t}\\\nonumber
&\le C_1(\a)\left(1+\frac{7}{N+1}\right)\pa{1+\a^2}{24\a^2}^{\frac{N-1}{2}}n^{-\frac{N+1}{2}}\sum_{t\ge 0}\pa{1+\a^2}{12\a^2\cdot N^2}^{t}\ \ \left(\text{as}\ n\ge \widehat{g}(N+1)\ge \frac{N^2}{2}\right)\\\nonumber
&\le C_1(\a)\left(1+\frac{7}{N+1}\right)\pa{1+\a^2}{24\a^2}^{\frac{N-1}{2}}n^{-\frac{N+1}{2}}\sum_{t\ge 0}\pa{1+\a^2}{12\a^2\cdot N}^{t}\ \ \left(\text{as}\ N\ge 1\right)\\\nonumber
&\le C_1(\a)\left(1+\frac{7}{N+1}\right)\pa{1+\a^2}{24\a^2}^{\frac{N-1}{2}}n^{-\frac{N+1}{2}}\left(1+\frac{2}{3 N}\right)\\\label{Mainthmpart2eqn2}
&\le C_1(\a)\left(1+\frac{8}{N}\right)\pa{1+\a^2}{24\a^2}^{\frac{N-1}{2}}n^{-\frac{N+1}{2}}=E^{[2]}_{N,e}\cdot n^{-\frac{N+1}{2}}\ \ \left(\text{by}\ \eqref{BPRSinvdef10eqn1}\right).
\end{align}
Next, we bound $S_o(n)$,
\begin{align}\nonumber
\left|S_o(n)\right|&\le \sum_{t\ge \frac{N}{2}}\frac{|g(2t+1)|}{n^{t+\frac 12}}\le  C_2(\a)\sum_{t\ge \frac{N}{2}}\pa{1+\a^2}{24\a^2}^{t}\frac{1}{n^{t}}\left(1+\frac{0.5}{t}\right)n^{-\frac 12}\ \ \left(\text{by}\ \text{Lemma}\ \ref{BPRSinvlem11}\right)\\\nonumber
&\le C_2(\a)\left(1+\frac{1}{N}\right)n^{-\frac 12}\sum_{t\ge \frac{N}{2}}\pa{1+\a^2}{24\a^2}^{t}\frac{1}{n^t}\ \ \left(\text{as}\ t\ge \frac{N+1}{2}\right)\\\nonumber
&= C_2(\a)\left(1+\frac{1}{N}\right)\pa{1+\a^2}{24\a^2}^{\frac{N}{2}}n^{-\frac{N+1}{2}}\sum_{t\ge 0}\pa{1+\a^2}{24\a^2\cdot n}^{t}\\\nonumber
&\le C_2(\a)\left(1+\frac{1}{N}\right)\pa{1+\a^2}{24\a^2}^{\frac{N}{2}}n^{-\frac{N+1}{2}}\left(1+\frac{2}{3 N}\right)\ \ \left(\text{as}\ n> \widehat{g}(N+1)\ge \frac{N^2}{2}\right)\\\label{Mainthmpart2eqn3}
&\le C_2(\a)\left(1+\frac{3}{N}\right)\pa{1+\a^2}{24\a^2}^{\frac{N}{2}}n^{-\frac{N+1}{2}}=E^{[2]}_{N,o}\cdot n^{-\frac{N+1}{2}}\ \ \left(\text{by}\ \eqref{BPRSinvdef10eqn2}\right).
\end{align}
Applying \eqref{Mainthmpart2eqn2} and \eqref{Mainthmpart2eqn3} to \eqref{Mainthmpart2eqn1} the statement of the theorem is proven.
\end{proof}

\section{Asymptotics of $p(n+k)/p(n)$ (Proof of Theorem \ref{Mainthm})}\label{Sec4}
For all $(k, N)\in\mathbb{N}^2$, define
\begin{equation}\label{Mainthmcutoff}
n_N(k):=\underset{(N,k)\in\mathbb{N}^2}{\max}\left\{\widehat{g}(N+1),(24k-1)^2\right\}
\end{equation}
Combining Theorems \ref{Mainthmpart1} and \ref{Mainthmpart2}, for all $n\ge n_N(k)$ we have
\begin{align}\nonumber
\frac{p(n+k)}{p(n)}&=\left(\sum_{t=0}^{N}\frac{\omega_k^{[1]}(t)}{n^{\frac t2}}+O_{\le E_N^{[1]}(k)}\left(n^{-\frac{N+1}{2}}\right)\right)\left(\sum_{t=0}^{N}\frac{g(t)}{n^{\frac t2}}+O_{\le E^{[2]}_N}\left(n^{-\frac{N+1}{2}}\right)\right)\\\nonumber
&=\sum_{t=0}^{N}\frac{\omega_k^{[1]}(t)}{n^{\frac t2}}\cdot \sum_{t=0}^{N}\frac{g(t)}{n^{\frac t2}}+\sum_{t=0}^{N}\frac{\omega_k^{[1]}(t)}{n^{\frac t2}}\cdot O_{\le E^{[2]}_N}\left(n^{-\frac{N+1}{2}}\right)+\sum_{t=0}^{N}\frac{g(t)}{n^{\frac t2}}\cdot O_{\le E_N^{[1]}(k)}\left(n^{-\frac{N+1}{2}}\right)\\\nonumber
&\hspace{7.5 cm} +O_{\le E_N^{[1]}(k)}\left(n^{-\frac{N+1}{2}}\right)\cdot O_{\le E^{[2]}_N}\left(n^{-\frac{N+1}{2}}\right)\\\nonumber
&=\sum_{t=0}^{N}\frac{1}{n^{\frac t2}}\sum_{s=0}^{t}\omega_k^{[1]}(s)\cdot g(t-s)+n^{-\frac{N+1}{2}}\sum_{t=0}^{N-1}\frac{1}{n^{\frac t2}}\sum_{s=t}^{N-1}\omega_k^{[1]}(s+1)g(N+t-s)\\\nonumber
&\hspace{3.5 cm}+\sum_{t=0}^{N}\frac{\omega_k^{[1]}(t)}{n^{\frac t2}}\cdot O_{\le E^{[2]}_N}\left(n^{-\frac{N+1}{2}}\right)+\sum_{t=0}^{N}\frac{g(t)}{n^{\frac t2}}\cdot O_{\le E_N^{[1]}(k)}\left(n^{-\frac{N+1}{2}}\right)\\\nonumber
&\hspace{7.25 cm}+O_{\le E_N^{[1]}(k)}\left(n^{-\frac{N+1}{2}}\right)\cdot O_{\le E^{[2]}_N}\left(n^{-\frac{N+1}{2}}\right)\\\label{Mainthmeqn1}
&=:\sum_{t=0}^{N}\frac{c_k(t)}{n^{\frac t2}}+S^{[1]}_{k}(n;N)+S^{[2]}_{k}(n;N)+S^{[3]}_{k}(n;N)+S^{[4]}_{k}(n;N),
\end{align}
where for $t\in \mathbb{Z}_{\ge 0}$,
\begin{equation}\label{Mainthmcoeff}
c_k(t)=\sum_{s=0}^{t}\omega_k^{[1]}(s)\cdot g(t-s).
\end{equation}
Next, to determine the error bound, we estimate an upper bound for each of the sums $S^{[j]}_{k}(n;N), 1\le j\le 4$, individually.

To begin with, we first estimate $S^{[4]}_{k}(n;N)$. To estimate the error bound we need to refine the error bounds $E_N^{[1]}(k)$ and $E^{[2]}_N$. From \eqref{BPRSshiftlem5eqn3} it follows that
\begin{align}\nonumber
E^{[1]}_{N,1,e}(k)&=\frac{\sqrt{2}\left|\sin\left(\alpha_k\right)\right|}{\sqrt{\pi}\alpha_k}\pa{24k-1}{24}^{\frac{N+1}{2}}\sqrt{N+1}\left(1+\frac{C_e(k)}{N}\right)\\\nonumber
&\le \pa{24k-1}{24}^{\frac{N+1}{2}}\sqrt{N+1}\left(1+C_e(k)\right)\ \left(\text{as}\ \a_k\ge 1\ \text{for}\ k\in\mathbb{N}\ \text{and}\ N\ge 1\right)\\\label{Mainthmeqn2}
&\le \pa{24k-1}{12}^{\frac{N+1}{2}}\left(1+C_e(k)\right)\ \ \left(\text{as}\ \sqrt{N+1}\le 2^{N+1}\ \text{for}\ N\ge 1\right).
\end{align}
Similarly from \eqref{BPRSshiftlem5eqn5}, we obtain
\begin{align}\nonumber
E^{[1]}_{N,1,o}(k)&=\pa{3}{\pi^3}^{\frac 12}\left|\cos\left(\alpha_k\right)\right|\pa{24k-1}{24}^{\frac{N}{2}}\sqrt{N+2}\left(1+\frac{C_o(k)}{N}\right)\\\nonumber
&\le \sqrt{\frac{24}{24k-1}}\pa{24k-1}{24}^{\frac{N+1}{2}}\sqrt{N+2}\left(1+C_o(k)\right)\ \left(\text{as}\ N\ge 1\right)\\\label{Mainthmeqn3}
&\le \pa{24k-1}{12}^{\frac{N+1}{2}}\left(1.1+1.1\cdot C_o(k)\right)\ \ \left(\text{as}\ \sqrt{N+2}\le 2^{N+1}\ \text{for}\ N\ge 1\right).
\end{align}
Applying \eqref{Mainthmeqn2} and \eqref{Mainthmeqn3} to \eqref{BPRSshiftlem5eqn6} we obtain
\[
E^{[1]}_{N,1}(k)\le \pa{24k-1}{12}^{\frac{N+1}{2}}\left(2.1+C_e(k)+1.1\cdot C_o(k)\right),
\]
which by \eqref{Mainthmpart1eqn6} implies
\begin{align}\nonumber
E^{[1]}_N(k)&\le \pa{24k-1}{12}^{\frac{N+1}{2}}\left(2.1+C_e(k)+1.1\cdot C_o(k)\right)+\left(1+\frac{3.7\cdot k}{N}\right)\pa{6}{\pi\sqrt{24}}^{N+1}\\\nonumber
&= \pa{24k-1}{12}^{\frac{N+1}{2}}\left(2.1+C_e(k)+1.1\cdot C_o(k)+\left(1+\frac{3.7\cdot k}{N}\right)\pa{6}{\pi \sqrt{2(24k-1)}}^{N+1}\right)\\\nonumber
&\le \pa{24k-1}{12}^{\frac{N+1}{2}}\left(2.1+C_e(k)+1.1\cdot C_o(k)+\left(1+\frac{3.7\cdot k}{N}\right)\pa{6}{\pi \sqrt{46}}^{2}\right)\ \ \left(\text{as}\ (k,N)\in\mathbb{N}^2\right)\\\label{Mainthmeqn4}
&\le \pa{24k-1}{12}^{\frac{N+1}{2}}\left(2.2+C_e(k)+1.1\cdot C_o(k)+\frac{0.3\cdot k}{N}\right).
\end{align}
Analogously, from \eqref{BPRSinvdef10eqn3}, it follows that
\begin{equation}\label{Mainthmeqn7}
E^{[2]}_N\le 4.1.
\end{equation}
Following \eqref{Mainthmeqn1}, we estimate the sum $S^{[4]}_{k}(n;N)$ as
\begin{align}\nonumber
\left|S^{[4]}_{k}(n;N)\right|&\le 4.1\pa{24k-1}{12}^{\frac{N+1}{2}}\left(2.2+C_e(k)+1.1\cdot C_o(k)+\frac{0.3\cdot k}{N}\right)n^{-N-1}\ \ \left(\text{by}\ \eqref{Mainthmeqn4}\ \text{and}\ \eqref{Mainthmeqn7}\right)\\\nonumber
&= 4.1\pa{24k-1}{12n}^{\frac{N+1}{2}}\left(2.8+C_e(k)+1.1\cdot C_o(k)+\frac{0.3\cdot k}{N}\right)n^{-\frac{N+1}{2}}\\\nonumber
&\le 4.1\frac{24k-1}{12n}\left(2.2+C_e(k)+1.1\cdot C_o(k)+0.3\cdot k\right)n^{-\frac{N+1}{2}}\\\nonumber
&\hspace{4 cm}\left(\text{as for}\ n\ge (24k-1)^2\ \text{and}\ k\ge 1,\ \frac{24k-1}{12n}<1\right)\\\nonumber
&\le 4.1\frac{1}{12\sqrt{n}}\left(2.2+C_e(k)+1.1\cdot C_o(k)+0.3\cdot k\right)n^{-\frac{N+1}{2}}\ \ \left(\text{as}\ n\ge (24k-1)^2\right)\\\nonumber
&\le 4.1\frac{\sqrt{2}}{12}\left(2.2+C_e(k)+1.1\cdot C_o(k)+0.3\cdot k\right)\frac 1Nn^{-\frac{N+1}{2}}\\\nonumber
&\hspace{7 cm}\ \left(\text{as}\ n\ge \widehat{g}(N+1)\ge \frac 12N^2\ \text{for all}\ N\ge 1\right)\\\label{Mainthmeqn8}
&\le E_{N,1}(k)\cdot n^{-\frac{N+1}{2}}
\end{align}
with 
\begin{equation}\label{Mainthmerror1}
E_{N,1}(k):=\frac{1.1+0.5\cdot C_e(k)+0.6\cdot C_o(k)+0.2\cdot k}{N}.
\end{equation}
Moving on, following \eqref{Mainthmeqn1} we bound $S^{[3]}_{k}(n;N)$,
\begin{align}\nonumber
&\left|S^{[3]}_{k}(n;N)\right|\le E^{[1]}_N(k)\cdot n^{-\frac{N+1}{2}}\left|\sum_{t=0}^{N}\frac{g(t)}{n^{\frac t2}}\right|\le E^{[1]}_N(k)\cdot n^{-\frac{N+1}{2}}\sum_{t=0}^{N}\frac{|g(t)|}{n^{\frac t2}}\\\nonumber
&=E^{[1]}_N(k)\cdot n^{-\frac{N+1}{2}}\left(1+\frac{|g(1)|}{\sqrt{n}}+\sum_{t=1}^{\frac N2}\frac{|g(2t)|}{n^t}+\frac{1}{\sqrt{n}}\sum_{t=1}^{\frac{N-1}{2}}\frac{|g(2t+1)|}{n^t}\right)\\\nonumber
&= E^{[1]}_N(k)\cdot n^{-\frac{N+1}{2}}\left(1+\frac{\frac{2239488-432\pi^4+\pi^6}{497664\sqrt{6}\pi^3}}{\sqrt{n}}+\sum_{t=1}^{\frac N2}\frac{|g(2t)|}{n^t}+\frac{1}{\sqrt{n}}\sum_{t=1}^{\frac{N-1}{2}}\frac{|g(2t+1)|}{n^t}\right)\\\nonumber
&\hspace{10 cm}\ \ \left(\text{by Definitions}\ \ref{BPRSinvdef7}\ \text{and}\ \ref{BPRSinvdef8}\right)\\\nonumber
&\le E^{[1]}_N(k)\cdot n^{-\frac{N+1}{2}}\left(1+\frac{0.1}{N}+\sum_{t=1}^{\frac N2}\frac{|g(2t)|}{n^t}+\frac{1}{\sqrt{n}}\sum_{t=1}^{\frac{N-1}{2}}\frac{|g(2t+1)|}{n^t}\right)\ \ \left(\text{as}\ n>\widehat{g}(N+1)\ge \frac{N^2}{2}\right)\\\nonumber
&\le E^{[1]}_N(k)\cdot n^{-\frac{N+1}{2}}\left(1+\frac{0.1}{N}+\frac 1N\sum_{t\ge 1}|g(2t)|+\frac 1N\sum_{t\ge 1}|g(2t+1)|\right)\\\nonumber
&\le E^{[1]}_N(k)\cdot n^{-\frac{N+1}{2}}\left(1+\frac{0.1}{N}+\frac{4.5\cdot C_1(\a)}{N}\sum_{t\ge 1}\pa{1+\a^2}{24\a^2}^{t-1}+\frac{1.5\cdot C_2(\a)}{N}\sum_{t\ge 1}\pa{1+\a^2}{24\a^2}^{t}\right)\\\nonumber
&\hspace{11 cm} \left(\text{by Lemmas}\ \ref{BPRSinvlem8}\ \text{and}\ \ref{BPRSinvlem11}\right)\\\label{Mainthmeqn9}
&\le E_{N,2}(k)\cdot  n^{-\frac{N+1}{2}}\ \ \left(\text{using}\ \eqref{BPRSinvlem8def}\ \text{and}\ \eqref{BPRSinvlem11def}\right)
\end{align}
with
\begin{equation}\label{Mainthmerror2}
E_{N,2}(k):=E^{[1]}_N(k)\left(1+\frac{1}{N}\right).
\end{equation}
Again, following \eqref{Mainthmeqn1} we bound $S^{[2]}_{k}(n;N)$,
\begin{align}\nonumber
&\left|S^{[2]}_{k}(n;N)\right|\le E^{[2]}_N\cdot n^{-\frac{N+1}{2}}\left|\sum_{t=0}^{N}\frac{\omega^{[1]}_k(t)}{n^{\frac t2}}\right|\le E^{[2]}_N\cdot n^{-\frac{N+1}{2}}\sum_{t=0}^{N}\frac{\left|\omega^{[1]}_k(t)\right|}{n^{\frac t2}}\\\nonumber
&=E^{[2]}_N\cdot n^{-\frac{N+1}{2}}\left(1+\frac{\left|\omega^{[1]}_k(1)\right|}{\sqrt{n}}+\sum_{t=1}^{\frac N2}\frac{\left|\omega^{[1]}_k(2t)\right|}{n^t}+\frac{1}{\sqrt{n}}\sum_{t=1}^{\frac{N-1}{2}}\frac{\left|\omega^{[1]}_k(2t+1)\right|}{n^t}\right)\\\nonumber
&= E^{[2]}_N\cdot n^{-\frac{N+1}{2}}\left(1+\frac{\frac{\pi^2(24k-1)-72}{\pi\cdot 24\sqrt{6}}}{\sqrt{n}}+\sum_{t=1}^{\frac N2}\frac{|\omega^{[1]}_k(2t)|}{n^t}+\frac{1}{\sqrt{n}}\sum_{t=1}^{\frac{N-1}{2}}\frac{|\omega^{[1]}_k(2t+1)|}{n^t}\right)\\\nonumber
&\hspace{10 cm}\ \ \left(\text{applying}\ \eqref{Mainthmpart1eqn2}\ \text{for}\ t=1\right)\\\label{Mainthmeqn10}
&\le E^{[2]}_N\cdot n^{-\frac{N+1}{2}}\left(1+\frac{1.9\cdot k}{N}+\sum_{t=1}^{\frac N2}\frac{|\omega^{[1]}_k(2t)|}{n^t}+\frac{\sqrt{2}}{N}\sum_{t=1}^{\frac{N-1}{2}}\frac{|\omega^{[1]}_k(2t+1)|}{n^t}\right)\ \left(\text{as}\ n\ge\widehat{g}(N+1)\ge \frac{N^2}{2}\right).
\end{align}
Using Lemma \ref{BPRSshiftlem4} it follows that
\begin{align}\nonumber
\sum_{t=1}^{\frac N2}\frac{|\omega^{[1]}_k(2t)|}{n^t}&\le \frac{2\left|\sin\left(\a_k\right)\right|}{\sqrt{\pi}\a_k}\sum_{t=1}^{\frac N2}\pa{24k-1}{24n}^t\sqrt{t}\left(1+\frac{C^*_2(k)}{t}\right)\\\nonumber
&\le \frac{12}{\pi^{\frac 32}\sqrt{23}}\left(1+C^*_2(k)\right)\sum_{t=1}^{\frac N2}\pa{24k-1}{24n}^t\sqrt{t}\\\nonumber
&\le  \frac{12}{\pi^{\frac 32}\sqrt{23}}\left(1+C^*_2(k)\right)\sum_{t=1}^{\frac N2}\pa{24k-1}{12n}^t\ \ \left(\text{as}\ \sqrt{t}\le 2^t\ \text{for}\ t\ge 1\right)\\\nonumber
&=\frac{12}{\pi^{\frac 32}\sqrt{23}}\left(1+C^*_2(k)\right)\sum_{t=1}^{\frac N2}\pa{24k-1}{12\sqrt{n}}^t\frac{1}{n^{\frac t2}}\le \frac{12}{\pi^{\frac 32}\sqrt{23}}\frac{\left(1+C^*_2(k)\right)}{\sqrt{n}}\sum_{t=1}^{\frac N2}\pa{24k-1}{12\sqrt{n}}^t\\\nonumber
&\le \frac{12\sqrt{2}}{\pi^{\frac 32}\sqrt{23}}\frac{\left(1+C^*_2(k)\right)}{N}\sum_{t=1}^{\frac N2}\pa{24k-1}{12\sqrt{n}}^t\ \ \left(\text{as}\ n\ge\widehat{g}(N+1)\ge \frac{N^2}{2}\right)\\\nonumber
&\le \frac{12\sqrt{2}}{\pi^{\frac 32}\sqrt{23}}\frac{\left(1+C^*_2(k)\right)}{N}\sum_{t\ge 1}\pa{1}{12}^t\ \ \left(\text{as}\ n\ge (24k-1)^2\right)= \frac{12\sqrt{2}}{11\pi^{\frac 32}\sqrt{23}}\frac{\left(1+C^*_2(k)\right)}{N}\\\label{Mainthmeqn11}
&\le \frac{6\cdot 10^{-2}\left(1+C^*_2(k)\right)}{N}.
\end{align}
Similarly, using Lemma \ref{BPRSshiftlem3} we obtain
\begin{equation}\label{Mainthmeqn12}
\sum_{t=1}^{\frac{N-1}{2}}\frac{|\omega^{[1]}_k(2t+1)|}{n^t}\le \frac{3\cdot 10^{-2}\left(1+C^*_1(k)\right)}{N}.
\end{equation}
Applying \eqref{Mainthmeqn11} and \eqref{Mainthmeqn11} to \eqref{Mainthmeqn10} we obtain
\begin{equation}\label{Mainthmeqn13}
\left|S^{[2]}_{k}(n;N)\right|\le E_{N,3}(k)\cdot n^{-\frac{N+1}{2}},
\end{equation}
with
\begin{equation}\label{Mainthmerror3}
E_{N,3}(k)=E^{[2]}_N\left(1+\frac{1.9\cdot k+6\cdot 10^{-2}\left(1+C^*_2(k)\right)+5\cdot 10^{-2}\left(1+C^*_1(k)\right)}{N}\right).
\end{equation}
We split the remaining sum $S^{[1]}_{k}(n;N)$ from \eqref{Mainthmeqn1} as follows,
\begin{align}\nonumber
S^{[1]}_{k}(n;N)&=\left(\sum_{t=0}^{N-1}\frac{1}{n^{\frac t2}}\sum_{s=t}^{N-1}\omega_k^{[1]}(s+1)g(N+t-s)\right)n^{-\frac{N+1}{2}}\\\nonumber
&=\left(\sum_{s=0}^{N-1}\omega_k^{[1]}(s+1)g(N-s)+\sum_{t=1}^{N-1}\frac{1}{n^{\frac t2}}\sum_{s=t}^{N-1}\omega_k^{[1]}(s+1)g(N+t-s)\right)n^{-\frac{N+1}{2}}\\\label{Mainthmeqn14}
&=:S^{[1]}_{k,1}(n;N)+S^{[1]}_{k,2}(n;N).
\end{align}
Observe that for $N=1$, 
\begin{equation}\label{Mainthmeqn15}
S^{[1]}_{k,2}(n;N)=0,
\end{equation}
and so in this case, we get
\begin{equation}\label{Mainthmeqn16}
\left|S^{[1]}_{k}(n;1)\right|=\left|\omega_k^{[1]}(1)\cdot g(1)\right|\cdot n^{-1}\le \frac{0.1\cdot k}{n}.
\end{equation}
By plugging in $t=1$ into Definitions \ref{BPRSinvdef7} and \ref{BPRSinvdef8}, along with \eqref{Mainthmpart1eqn2} we estimate an upper bound for $\left|S^{[1]}_{k}(n;N)\right|$ with $t\ge 1$. First, using Lemmas \eqref{BPRSshiftlem3} and \eqref{BPRSshiftlem4} we obtain for all $t\ge 2$,
\begin{equation}\label{Mainthmeqn17}
\left|\omega_k^{[1]}(t)\right|\le \frac 12\pa{24k-1}{24}^{\frac t2}\sqrt{t}\left(1+C^*(k)\right)\ \text{with}\ C^*(k)=\underset{k\ge 1}{\max}\left\{C_1^*(k),C_2^*(k)\right\}.
\end{equation}
In a similar way, for $t\ge 2$ it follows that
\begin{equation}\label{Mainthmeqn18}
\left|g(t)\right|\le 3.2\pa{1+\a^2}{24\a^2}^{\frac t2}.
\end{equation}
For all $N\ge 2$ one obtains
\begin{align}\nonumber
&\left|S^{[1]}_{k}(n;N)\right|\\\nonumber
&\le \left(\sum_{s=0}^{N-1}\left|\omega_k^{[1]}(s+1)\right|\left|g(N-s)\right|\right)n^{-\frac{N+1}{2}}\\\nonumber
&\le \left(\left|\omega_k^{[1]}(1)\right|\left|g(N)\right|+\sum_{s=1}^{N-2}\left|\omega_k^{[1]}(s+1)\right|\left|g(N-s)\right|+\left|\omega_k^{[1]}(N)\right|\left|g(1)\right|\right)n^{-\frac{N+1}{2}}\\\nonumber
&\le \left(1.3\cdot k\left|g(N)\right|+\sum_{s=1}^{N-2}\left|\omega_k^{[1]}(s+1)\right|\left|g(N-s)\right|+6\cdot 10^{-2}\left|\omega_k^{[1]}(N)\right| \right)n^{-\frac{N+1}{2}}\\\nonumber
&\hspace{6 cm}\left(\text{by Definitions}\ \ref{BPRSinvdef7}\ \text{and}\ \eqref{BPRSinvdef8},\ \text{and using}\ \eqref{Mainthmpart1eqn2}\ \text{for}\ t=1\right)\\\nonumber
&\le \left(4.2\cdot k\pa{1+\a^2}{24\a^2}^{\frac N2}+\sum_{s=1}^{N-2}\left|\omega_k^{[1]}(s+1)\right|\pa{1+\a^2}{24\a^2}^{\frac{N-s}{2}}+6\cdot 10^{-2}\left|\omega_k^{[1]}(N)\right| \right)n^{-\frac{N+1}{2}}\\\nonumber
&\hspace{12 cm}\left(\text{by}\ \eqref{Mainthmeqn18}\right)\\\nonumber
&\le \Biggl(4.2\cdot k\pa{1+\a^2}{24\a^2}^{\frac N2}+\frac{(1+C^*(k))(24k-1)}{48}\pa{1+\a^2}{24\a^2}^{\frac N2}\sum_{s=1}^{N-2}\pa{\a^2(24k-1)}{1+\a^2}^{\frac{s}{2}}\sqrt{s+1}\\\nonumber
&\hspace{4cm}+ 3\cdot 10^{-2}\pa{24k-1}{24}^{\frac N2}\sqrt{N}\left(1+C^*(k)\right)\Biggr)n^{-\frac{N+1}{2}}\ \ \left(\text{by}\ \eqref{Mainthmeqn17}\right)\\\nonumber
&\le \Biggl(4.2\cdot k\pa{1+\a^2}{24\a^2}^{\frac N2}+\frac{(1+C^*(k))(24k-1)}{48}\pa{1+\a^2}{24\a^2}^{\frac N2}\sqrt{N}\sum_{s=1}^{N-2}\pa{\a^2(24k-1)}{1+\a^2}^{\frac{s}{2}}\\\nonumber
&\hspace{4cm}+ 3\cdot 10^{-2}\pa{24k-1}{24}^{\frac N2}\sqrt{N}\left(1+C^*(k)\right)\Biggr)n^{-\frac{N+1}{2}}\ \ \left(\text{as}\ s\le N-2\right)\\\nonumber
&= \Biggl(4.2\cdot k\pa{1+\a^2}{24\a^2}^{\frac N2}+\frac{(1+C^*(k))(24k-1)}{48}\pa{24k-1}{24}^{\frac N2}\sqrt{N}\sum_{s=2}^{N-1}\pa{1+\a^2}{\a^2(24k-1)}^{\frac{s}{2}}\\\nonumber
&\hspace{4cm}+ 3\cdot 10^{-2}\pa{24k-1}{24}^{\frac N2}\sqrt{N}\left(1+C^*(k)\right)\Biggr)n^{-\frac{N+1}{2}}\\\nonumber
&\le \Biggl(4.2\cdot k\pa{1+\a^2}{24\a^2}^{\frac N2}+\frac{(1+C^*(k))(24k-1)}{48}\pa{24k-1}{24}^{\frac N2}\sqrt{N}\sum_{s\ge 2}\pa{1+\a^2}{\a^2(24k-1)}^{\frac{s}{2}}\\\nonumber
&\hspace{6cm}+ 3\cdot 10^{-2}\pa{24k-1}{24}^{\frac N2}\sqrt{N}\left(1+C^*(k)\right)\Biggr)n^{-\frac{N+1}{2}}\\\nonumber
&= \left(4.2\cdot k\pa{1+\a^2}{24\a^2}^{\frac N2}+\left(\frac{(1+C^*(k))(24k-1)}{96}+3\cdot 10^{-2}\right)\pa{24k-1}{24}^{\frac N2}\sqrt{N}\right)n^{-\frac{N+1}{2}}\\\nonumber
&=\left(3\cdot 10^{-2}+\frac{(1+C^*(k))(24k-1)}{96}+\frac{4.2\cdot k\pa{1+\a^2}{\a^2(24k-1)}^{\frac N2}}{\sqrt{N}}\right)\pa{24k-1}{24}^{\frac N2}\sqrt{N}\cdot n^{-\frac{N+1}{2}}\\\nonumber
&\le \left(3\cdot 10^{-2}+\frac{(1+C^*(k))(24k-1)}{96}+\frac{0.9\cdot k}{\sqrt{N}}\right)\pa{24k-1}{24}^{\frac N2}\sqrt{N}\cdot n^{-\frac{N+1}{2}}\ \ \left(\text{as}\ (k,N)\in\mathbb{N}^2\right)\\\label{Mainthmeqn13a}
&=E^{[1]}_{N,4}(k)\cdot n^{-\frac{N+1}{2}}
\end{align}
with
\begin{equation}\label{Mainthmerror4a}
E^{[1]}_{N,4}(k):=\left(3\cdot 10^{-2}+\frac{(1+C^*(k))(24k-1)}{96}+\frac{0.9\cdot k}{\sqrt{N}}\right)\pa{24k-1}{24}^{\frac N2}\sqrt{N}.
\end{equation}
 Finally,  we proceed to estimate $S^{[1]}_{k,2}(n;N)$ for $N\ge 2$. 
\begin{align}\nonumber
&\left|S^{[1]}_{k,2}(n;N)\right|\\\nonumber
&\le \sum_{t=1}^{N-1}\frac{1}{n^{\frac t2}}\sum_{s=t}^{N-1}\left|\omega_k^{[1]}(s+1)\right|\left|g(N+t-s)\right| n^{-\frac{N+1}{2}}\ \ \left(\text{by}\ \eqref{Mainthmeqn14}\right)\\\nonumber
&\le 1.6\left(1+C^*(k)\right)\sqrt{\frac{24k-1}{24}}\pa{1+\a^2}{24\a^2}^{\frac N2}\sum_{t=1}^{N-1}\pa{1+\a^2}{24\a^2n}^{\frac t2}\sum_{s=t}^{N-1}\pa{\a^2(24k-1)}{1+\a^2}^{\frac s2}\sqrt{s+1}\\\nonumber
&\hspace{11 cm}\left(\text{by}\ \eqref{Mainthmeqn17}\ \text{and}\ \eqref{Mainthmeqn18}\right)\\\nonumber
&\le 1.6\left(1+C^*(k)\right)\sqrt{\frac{24k-1}{24}}\pa{1+\a^2}{24\a^2}^{\frac N2}\sqrt{N}\sum_{t=1}^{N-1}\pa{1+\a^2}{24\a^2n}^{\frac t2}\sum_{s=t}^{N-1}\pa{\a^2(24k-1)}{1+\a^2}^{\frac s2}\ \left(\text{as}\ s\le N-1\right)\\\nonumber
&=1.6\sqrt{\frac{1+\a^2}{24\a^2}}\left(1+C^*(k)\right)\pa{24k-1}{24}^{\frac N2}\sqrt{N}\sum_{t=1}^{N-1}\pa{1+\a^2}{24\a^2n}^{\frac t2}\sum_{s=0}^{N-t-1}\pa{1+\a^2}{\a^2(24k-1)}^{\frac s2}\\\nonumber
&\le 0.8\left(1+C^*(k)\right)\pa{24k-1}{24}^{\frac N2}\sqrt{N}\sum_{t=1}^{N-1}\pa{1+\a^2}{24\a^2n}^{\frac t2}\sum_{s=0}^{N-t-1}\pa{1}{2}^s\ \ \left(\text{as}\ k\ge 1\right)\\\nonumber
&\le 0.8\left(1+C^*(k)\right)\pa{24k-1}{24}^{\frac N2}\sqrt{N}\sum_{t=1}^{N-1}\pa{1+\a^2}{24\a^2n}^{\frac t2}\sum_{s\ge 0}\pa{1}{2}^s\\\nonumber
&= 1.6\left(1+C^*(k)\right)\pa{24k-1}{24}^{\frac N2}\sqrt{N}\sum_{t=1}^{N-1}\pa{1+\a^2}{24\a^2n}^{\frac t2}\\\nonumber
&\le 1.6\left(1+C^*(k)\right)\pa{24k-1}{24}^{\frac N2}\sqrt{N}\pa{1+\a^2}{24\a^2\cdot n}^{\frac 12}\sum_{t\ge 0}\pa{1+\a^2}{24\a^2n}^{\frac t2}\\\nonumber
&\le 1.6\left(1+C^*(k)\right)\pa{24k-1}{24}^{\frac N2}\frac{1}{\sqrt{N}}\pa{1+\a^2}{12\a^2}^{\frac 12}\sum_{t\ge 0}\pa{1+\a^2}{24\a^2n}^{\frac t2}\ \ \left(\text{as}\ n>\widehat{g}(N+1)\ge \frac{N^2}{2}\right)\\\nonumber
&\le 1.6\left(1+C^*(k)\right)\pa{24k-1}{24}^{\frac N2}\frac{1}{\sqrt{N}}\pa{1+\a^2}{12\a^2}^{\frac 12}\sum_{t\ge 0}\pa{1+\a^2}{48\a^2}^{\frac t2}\ \ \left(\text{as}\ n> \frac{N^2}{2}\ge 2\ \text{for}\ N\ge 2\right)\\\label{Mainthmeqn19}
&\le 0.8\left(1+C^*(k)\right)\pa{24k-1}{24}^{\frac N2}\frac{1}{\sqrt{N}}.
\end{align}
Combining \eqref{Mainthmeqn15} and \eqref{Mainthmeqn19}, it follows that for $N\ge 1$,
\begin{equation}\label{Mainthmeqn20}
\left|S^{[1]}_{k,2}(n;N)\right|\le E^{[2]}_{N,4}(k)\cdot n^{-\frac{N+1}{2}}
\end{equation}
with
\begin{equation}\label{Mainthmerror4b}
E^{[2]}_{N,4}(k):=0.8
\frac{\left(1+C^*(k)\right)}{\sqrt{N}}\pa{24k-1}{24}^{\frac N2}.
\end{equation}
Combining \eqref{Mainthmeqn13a} and \eqref{Mainthmeqn20}, 
\begin{equation}\label{Mainthmeqn21}
\left|S^{[1]}_{k}(n;N)\right|\le \left(E^{[1]}_{N,4}(k)+E^{[2]}_{N,4}(k)\right)n^{-\frac{N+1}{2}}=:E_{N,4}(k)n^{-\frac{N+1}{2}}.
\end{equation}
Finally, setting
\begin{equation}\label{Mainthmerrorfinal}
E_N(k):=\sum_{m=1}^{4}E_{N,m}(k),
\end{equation}
and combining \eqref{Mainthmeqn8}, \eqref{Mainthmeqn9}, \eqref{Mainthmeqn13}, and \eqref{Mainthmeqn21} concludes the proof of the theorem.
\qed
	
\section{Appendix}\label{Sec5}
This section is divided into two Subsections \ref{Sec5a} and \ref{Sec5b} where we utilize Schneider's Symbolic Summation machinery, implemented in Mathematica in the form of the \texttt{Sigma} package, to simplify the sums $(S_j(t))_{1\le j\le 9}$ defined in \Cref{BPRSinvlem5}. 
These simplifications are crucial to obtain the fundamental estimates as specified in Lemma \ref{BPRSinvlem5}. We observed that there are two classes: (i) the sums $S_j(t)$ for $j\in\{1,4,5,6,8\}$ which are (or reduce to) single sums, and (ii) the sums $S_j(t)$ with $j\in\{2,3,7,9\}$ which can be simplified to combinations of certain double sums. We restrict to show details of the summation machinery of \texttt{Sigma} by choosing the double sum representations of $S_2(t)$ and $S_3(t)$ from class (ii) mentioned before. To illustrate the full power of \texttt{Sigma} some further aspects are given for the more involved triple sum $S_3(t)$. The \texttt{Sigma} simplification of the remaining sums works similarly.
\subsection{Simplification of the sums $S_2(t)$ and $S_3(t)$ using \texttt{Sigma}}\label{Sec5a}
Before we apply our summation tools to $S_2(t)$ we observe that the sum representation in~\eqref{BPRSinvdef5c} can be rewritten in the form
$$
S_2(t)=\sum_{u=1}^{t-2} \frac{(-1)^u \big(
	\alpha ^2\big)^u 
	}{(2 u-1)!}\sum_{s=0}^{
	t
	-u-2
} \frac{\big(
-(1+\alpha^{-2})\big)^{-s
	-u
} (-s
-u
)_u \big(
\frac{1}{2}
-s
-u
\big)_{u+s+1
}}{(s
+u
) (s
+2 u
)!};$$
recall that $\alpha=\frac{\pi}{6}$.
In the following we will focus on the inner sum which reads as
$$T(t,u)=
\sum_{s=0}^{
	t
	-u-2
} \frac{\pi ^{2 (s
		+u
		)} \big(
	\frac{-1}{36+\pi ^2}\big)^{s
		+u
	} (-s
	-u
	)_u \big(
	\frac{1}{2}
	-s
	-u
	\big)_{1
		+s
		+u
}}{(s
	+u
	) (s
	+2 u
	)!}.$$
Using the summation package \texttt{Sigma} (or alternatively, any other package, like~\cite{PauleSchorn:95}, that can deal with the creative telescoping paradigm~\cite{Z} for hypergeometric products) one obtains the linear recurrence
\begin{multline}\label{Equ:SingleSumRecBad}
\pi ^2 u T(t,u)
+2 \big(
72+\pi ^2\big) (1+u) T(t,u+1)
+\pi ^2 (2+u) T(t,u+2)\\
=\tfrac{(-1)^t \pi ^{-2+2 t} (-3+t) (-2+t) (-7+2 t) (-5+2 t) (-3+2 t) (1+u) (4-t)_u \big(
	\frac{9}{2}-t\big)_{-3+t}}{2 (-3
	+t
	-u
	) (-2
	+t
	-u
	) (t
	+u
	) (-3
	+t
	+u
	) (-2
	+t
	+u
	) (-1
	+t
	+u
	)(
	36+\pi ^2)^{t-2} (-4
	+t
	+u
	)!}.
\end{multline}
Next, we are interested in d'Alembertian solutions of the found recurrence, i.e., in the set of all solutions that can be given in terms of indefinite nested sums defined over hypergeometric products. Using \texttt{Sigma} we find two linearly independent solutions
\begin{equation}\label{Equ:HomSol}
\frac{1}{u}\Big(\frac{
	-72-\pi ^2+12 \sqrt{36
		+\pi ^2
	}}{\pi^2}\Big)^u\quad\text{and}\quad\frac{(-1)^u}{u}\Big(\frac{
	-72-\pi ^2+12 \sqrt{36
		+\pi ^2
}}{\pi^2}\Big)^u
\end{equation}
of the homogeneous version of the recurrence~\eqref{Equ:SingleSumRecBad}. To this end, we apply a generalized version~\cite{ABPS} of
Petkov{\v s}ek's Hyper algorithm~\cite{Petkov:92}; we remark that the bases $\frac{-72-\pi ^2+12 \sqrt{36+\pi ^2}}{\pi ^2}$  and $-\frac{72+\pi ^2+12 \sqrt{36+\pi ^2}}{\pi ^2}$of the powers in~\eqref{Equ:HomSol} are the roots of the characteristic polynomial
\begin{equation}\label{Equ:CharPoly}
q(x)=\pi ^2 x^2+2 \left(72+\pi ^2\right) x+\pi ^2
\end{equation}
that appears in the underlying algorithm~\cite{Petkov:92}.

In the next step, we can apply the algorithm given in~\cite{AP:94} to find the particular solution 
\begin{multline*}
\tfrac{(
36+\pi ^2)^2 (-3+t) (-2+t) (-7+2 t) (-5+2 t) (-3+2 t) (-1)^t \pi ^{2 t} \big(
-72-\pi ^2+12 \sqrt{36
	+\pi ^2
}\big)^u\big(\tfrac{9}{2}-t\big)_{-3+t}}{2 u \pi^{4+2 u} (
36+\pi^2)^t}\times\\ 
\times\sum_{i=1}^u\tfrac{(-1)^i\big(
72+\pi ^2+12 \sqrt{36
	+\pi ^2
}\big)^i}{\big(
-72-\pi ^2+12 \sqrt{36
+\pi ^2
}\big)^{i}} 
\sum_{j=1}^i \tfrac{(-1)^j (-1+j) \pi ^{2 j} (4-t)_j}{(j
	-t
	) (1
	+j
	-t
	) (2
	+j
	-t
	) (3
	+j
	-t
	) (-3
	+j
	+t
	) (-2
	+j
	+t
	) \big(
	72+\pi ^2+12 \sqrt{36
		+\pi ^2
	}\big)^{j} (-4
	+j
	+t
	)!}
\end{multline*}
of the recurrence~\eqref{Equ:SingleSumRecBad} in terms of an indefinite nested double sum over the powers given in~\eqref{Equ:HomSol}. Then a crucial step for further processing, i.e., for carrying out the estimates in Section~\ref{Sec5b}, is the simplification of the found particular solution. Here again the summation package \texttt{Sigma} can assist; for a survey of the used telescoping algorithms we refer, e.g., to~\cite{S21}. However, with the given limitations of the computer algebra system Mathematica, the underlying implementation does not work properly when dealing with algebraic elements such as $\sqrt{36+\pi ^2}$. To overcome this situation, we perform the substitution
\begin{equation}\label{Equ:PiToBSum2}
\pi \mapsto \frac{6 \sqrt{1-b^2}}{b}
\end{equation}
for a new variable $b$; note that the inverse substitution is
\begin{equation}\label{Equ:BToPiSum2}
b\mapsto \frac{6}{\sqrt{36+\pi ^2}}.
\end{equation}
In this way, the technically challenging element $\sqrt{36+\pi^2}$ vanishes; it simplifies to
$$\sqrt{36+\pi^2}=\frac{6}{b}.$$
We remark that such rewritings (also called rationalizing transformations) appear frequently when dealing with nested integrals in the research area of elementary particle physics; for a survey and applications we refer to~\cite{Raab}.

After this preprocessing step we continue with the sum representation
$$
\tilde{S}_2(t)=\sum_{u=1}^{t-2} \frac{(-1)^u \big(
	\frac{1-b^2}{b^2}\big)^u 
	}{(2 u-1)!}\tilde{T}(t,u)$$
where
$$\tilde{T}(t,u)=\sum_{s=0}^{t
	-u-2
} \frac{\big(
	b^2-1\big)^{s
		+u
	} (-s
	-u
	)_u \big(
	\frac{1}{2}
	-s
	-u
	\big)_{1
		+s
		+u
}}{(s
	+u
	) (s
	+2 u
	)!}.$$
By construction, if we apply the substitution~\eqref{Equ:BToPiSum2} to $\tilde{S}_2(t)$ one gets $S_2(t)$ back.  Repeating the above calculations in this new representation we obtain the improved recurrence
\begin{multline*}
(-1+b) (1+b) u\tilde{T}(t,u)
-2 \big(
1+b^2\big) (1+u) \tilde{T}(t,u+1)
+(-1+b) (1+b) (2+u) \tilde{T}(t,u+2)\\
=\tfrac{(-1+b)^{-1+t} (1+b)^{-1+t} (-3+t) (-2+t) (-7+2 t) (-5+2 t) (-3+2 t) (1+u) (4-t)_u \big(
	\frac{9}{2}-t\big)_{-3+t}}{2 (-3
	+t
	-u
	) (-2
	+t
	-u
	) (t
	+u
	) (-3
	+t
	+u
	) (-2
	+t
	+u
	) (-1
	+t
	+u
	) (-4
	+t
	+u
	)!}.
\end{multline*}
We remark that the correctness of the recurrence can be verified by the summand recurrence (not printed here) that is provided by the creative telescoping method.
Solving this recurrence  we get now the two linearly independent solutions
$$\tilde{H}_1(t,u)=\frac{(b-1)^u}{u(1+b)^{u}}\quad\text{and}\quad \tilde{H}_2(t,u)=\frac{(1+b)^u}{u(b-1)^{u} }$$
of the homogeneous version; note that these solutions agree with~\eqref{Equ:HomSol} with the substitutions in~\eqref{Equ:BToPiSum2}, resp.~\eqref{Equ:PiToBSum2}. Most relevant, using \texttt{Sigma} the found particular solution can be simplified to the form
\begin{multline*}
\tilde{P}(t,u)=\tfrac{(-1+b)^{t+u} (1+b)^{-3+t-u} \big(
	\tfrac{1}{2}-t\big)_t}{t!} R_1(t,u)
	+\tfrac{(-1+b)^{-3+t-u} (1+b)^{t+u} \big(
		\frac{1}{2}-t\big)_t}{t!} R_2(t,u)\\
	+\tfrac{(-1+b)^{-1+t} (1+b)^{-1+t} \big(
		\frac{1}{2}-t\big)_t (-t)_u}{(t
		+u
		)!}R_3(t,u)
	-\tfrac{(-1+b)^{t-u} (1+b)^{t+u}
		\big(\frac{1}{2}-t\big)_t}{2 u}\sum_{i=1}^u \frac{(-1+b)^{i} (1+b)^{-i} (-t)_{i}}{\big(
		t
		+i
		\big)!}\\
	-\tfrac{(-1+b)^{t+u} (1+b)^{t-u} \big(\frac{1}{2}-t\big)_t}{2 u}\sum_{i=1}^u \frac{(-1+b)^{-i} (1+b)^{i} (-t)_{i}}{\big(
		t
		+i
		\big)!}
\end{multline*}
for some rational functions $R_1(t,u)$, $R_2(t,u)$ and $R_3(t,u)$ in the variables $b,u,t$. We remark that correctness of the computed solutions can be verified  easily by plugging them into the recurrence, carrying out the shift operator on the indefinite nested sums and products, and by finally simplification by performing simple rational function arithmetic.

In summary, using \texttt{Sigma} we found two linearly independent solutions of the homogeneous version together with the a particular solution $\tilde{P}(t,u)$ of the recurrence itself. Due to the linearly independence of the homogeneous solutions, there exist unique constants $c_1(t)$ and $c_2(t)$ (free of $u$) such that
\begin{equation}\label{Equ:GeneralSol}
\tilde{T}(t,u)=c_1(t)\,\tilde{H}_1(t,u)+c_2(t)\,\tilde{H}_2(t,u)+\tilde{P}(t,u)
\end{equation}
holds for all nonnegative integers $t$ and $u$ with $t\geq2$ and $1\leq u\leq t-2$. 
To determine these $c_j(t)$, we consider the two initial values $\tilde{T}(t,1)$ and $\tilde{T}(t,2)$ which can be simplified with the same technology as above (where $t$ has taken over the role of $u$) to the form
\begin{align*}
\tilde{T}(t,1)&=\tfrac{1}{2}-\tfrac{(-1+b)^{-1+t} (1+b)^{-1+t} \big(
	\tfrac{1}{2}-t\big)_t}{(-1+2 t)t!}
+\tfrac{1}{(-1+b) (1+b)}\sum_{i=1}^t \frac{(-1+b)^i (1+b)^i \big(
	\frac{1}{2}-i\big)_i}{(-1+2 i) i!},\\
\tilde{T}(t,2)&=\tfrac{3+b^2}{4 (-1+b) (1+b)}-\tfrac{(-1+b)^{-1+t} (1+b)^{-1+t} \big(
	\frac{1}{2}-t\big)_t}{(-1+2 t)t!}
+\tfrac{(
	1+b^2)}{(-1+b)^2 (1+b)^2}\sum_{i=1}^t \frac{(-1+b)^i (1+b)^i \big(
	\frac{1}{2}-i\big)_i}{(-1+2 i) i!}.
\end{align*}
By setting $u=1,2$ in~\eqref{Equ:GeneralSol} and solving the underlying system we compute the $c_1(t)$ and $c_2(t)$ (which we do not print here) and 
obtain the desired simplification
\begin{equation}\label{Equ:T2ClForm}
\begin{split}
\tilde{T}(t,u)&=\frac{\big(
	(1+b) (-1+b)^{2 u}
	+(-1+b) (1+b)^{2 u}
	\big)}{4 b u(-1+b)^{u} (1+b)^{u}}\\
&-\frac{\big(
	(1+b) (-1+b)^{2 u}
	+(-1+b) (1+b)^{2 u}
	\big)(-1+b)^{t-u} (1+b)^{t-u} \big(
	\frac{1}{2}-t\big)_t}{4bu\,t!}\\
&+\tfrac{\big(
	(-1+b) (1+b) (-1+t) (-1+2 t)
	+2 t u
	+2 u^2
	\big)(t
	-u
	) (-1+b)^{-1+t} (1+b)^{-1+t} \big(
	\frac{1}{2}-t\big)_t (-t)_u}{2 (-1+t) t (-1+2 t) u(t
	+u
	)!}\\
&-\frac{(-1+b)^{t-u} (1+b)^{t+u}
	\big(\frac{1}{2}-t\big)_t}{2 u}\sum_{i=1}^u \frac{(-1+b)^i (-t)_i}{(i
	+t
	)! (1+b)^{i}}\\
&+\frac{\big(
	(1+b)^{2 u}
	-(-1+b)^{2 u}
	\big) (-1+b)^{-u} (1+b)^{-u} }{4 b u}\sum_{i=1}^t \frac{(-1+b)^i (1+b)^i \big(
	\frac{1}{2}-i\big)_i}{(-1+2 i) i!}\\
&-\frac{(-1+b)^{t+u} (1+b)^{t-u}
	\big(\frac{1}{2}-t\big)_t}{2 u}\sum_{i=1}^u \frac{(1+b)^i (-t)_i}{(-1+b)^{i} (i
	+t
	)!}.
\end{split}
\end{equation}
Finally, we replace in
$$\tilde{S}_2(t,u)=\sum_{u=1}^{t-2}\frac{(-1)^u \big(
	\frac{1-b^2}{b^2}\big)^u}{(-1+2 u)!}\tilde{T}(t,u)$$
the sum $\tilde{T}(t,u)$ by the right-hand side of~\eqref{Equ:T2ClForm} and apply the substitution~\eqref{Equ:PiToBSum2}. This yields an expression that agrees with $S_2(t)$ for all $t\geq2$ from which one can read of the splittings~\eqref{sum2part1}--\eqref{sum2part2} of $S_2(t)$ which in Section~\ref{Sec5b} are used to obtain the estimate~\eqref{BPRSinvlem5eqn2} in Lemma~\ref{BPRSinvlem5}.

As indicated above, all calculations can be verified independently from the steps of their algorithmic derivations and thus lead to a rigorous proof for the correctness of the simplification. However, all the steps are tedious to handle manually. 
To assist such calculations (that are similarly to those that we frequently tackle in particle physics~\cite{Physics}), 
we provide besides \texttt{Sigma} the package \texttt{EvaluateMultiSums}~\cite{S13} which contains all the experience of the \texttt{Sigma}--developer to carry out these calculations  fully automatically. More precisely, after loading the packages
into the computer algebra system Mathematica
\begin{mma}
	\In << Sigma.m \\
	\vspace*{-0.1cm}
	\Print \LoadP{Sigma - A summation package by Carsten Schneider
		\copyright\ RISC-JKU}\\
\end{mma}
\begin{mma}
	\In << EvaluateMultiSums.m \\
	\vspace*{-0.1cm}
	\Print \LoadP{EvaluateMultiSum by Carsten Schneider
		\copyright\ RISC-JKU}\\
\end{mma}
\noindent we can insert the sum
\begin{mma}
	\In \tilde{T}=\sum_{s=0}^{t
		-u-2
	} \frac{\big(
		b^2-1\big)^{s
			+u
		} (-s
		-u
		)_u \big(
		\frac{1}{2}
		-s
		-u
		\big)_{1
			+s
			+u
	}}{(s
		+u
		) (s
		+2 u
		)!};\\
\end{mma}
\noindent and get with the command
\begin{mma}
\In EvaluateMultiSum[\tilde{T},\{\},\{u,t\},\{1,2\},\{t-2,\infty\}]\\
\Out \dots\\
\end{mma}
\noindent the output $``\dots"$ that is equivalent to~\eqref{Equ:T2ClForm}. Actually, one obtains a slight variant of it that one can transform to the version~\eqref{Equ:T2ClForm} using the command \texttt{SigmaReduce}; for a concrete application of such a transformation we refer to \myIn{\ref{MMA:SR1}} and \myIn{\ref{MMA:SR2}} below.

\medskip

Using this technology, we focus now on the more involved triple sum $S_3(t)$ defined in~\eqref{BPRSinvdef6b}. Similarly to $S_2(t)$ we rewrite $S_3(t)$ and work with the alternative sum representation
$$S_3(t)=
\sum_{u=0}^{t-2} \frac{(-1)^u \big(\frac{\pi}{6}\big)^{2 u} 
	}{(2 u)!}\sum_{s=0}^{
	t
	-u-2
} \frac{\big(
\frac{1}{2}
-s
-u
\big)_{1
	+s
	+u
} (-s
-u
)_u 
}{(1
+s
+2 u
)!}\sum_{r=0}^{
t
-u-s-1
} (-1)^r\big(\tfrac{6}{\pi}\big)^{2r}\binom{-\frac{1}{2}-r}{-1
-r
-s
+t
-u
}.$$
When trying to apply the summation tools on these summations, it is reasonable to start with the inner most sum, say $U(t,s,u)$. One finds the representation
\begin{multline*}
U(t,s,u):=\sum_{r=0}^{
	t
	-u-s-1
} (-1)^r\Big(\frac{6}{\pi}\Big)^{2r}\binom{-\frac{1}{2}-r}{-1
	-r
	-s
	+t
	-u
}\\	
=-(-1)^{s+t+u} \pi ^{2+2 s-2 t+2 u} \big(
36+\pi ^2\big)^{-1-s+t-u}
-\frac{(-1)^{s+u} \pi ^{2+2 s+2 u} \big(
	36+\pi ^2\big)^{-1-s-u} \binom{-\frac{1}{2}}{t}}{-1+2 t}\\
+(-1)^{s+t+u} \pi ^{2+2 s-2 t+2 u} \big(
36+\pi ^2\big)^{-1-s+t-u} 
\sum_{i=1}^t \frac{\pi ^{2 i} \big(
	-\frac{1}{36+\pi ^2}\big)^i \binom{-\frac{1}{2}}{i}}{-1+2 i}\\
+\frac{1}{2} (-1)^{s+u} \pi ^{2+2 s+2 u} \big(
36+\pi ^2\big)^{-1-s-u} 
\sum_{i=1}^u \frac{\pi ^{-2 i} \big(
	-36-\pi ^2\big)^i \binom{-\frac{1}{2}}{-1
		-i
		+t
}}{-i
	+t
}\\
+\frac{1}{2} (-1)^s \pi ^{2+2 s} \big(
36+\pi ^2\big)^{-1-s} 
\sum_{i=1}^s \frac{\pi ^{-2 i} \big(
	-36-\pi ^2\big)^i \binom{-\frac{1}{2}}{-1
		-i
		+t
		-u
}}{-i
	+t
	-u
}.
\end{multline*}
In the next step, one may deal with the second summation
$$T(t,u)=\sum_{s=0}^{
	t
	-u-2
} \frac{\big(
	\frac{1}{2}
	-s
	-u
	\big)_{1
		+s
		+u
	} (-s
	-u
	)_u 
}{(1
	+s
	+2 u
	)!}U(t,u,s)$$
using the simplification of $U(t,u,s)$. Using \texttt{Sigma} one obtains the recurrence
$$\pi ^2 (1+2 u) T(t,u)
+2 \big(
72+\pi ^2\big) (3+2 u) T(t,u+1)
+\pi ^2 (5+2 u) T(t,u+2)
=R(t,u)$$
with a right-hand side $R(t,u)$ represented in terms of indefinite nested sums too large for being printed here. Solving the homogeneous version of the recurrence gives the two solutions
$$\frac{1}{2u+1}\Big(\frac{
	-72-\pi ^2+12 \sqrt{36
		+\pi ^2
}}{\pi^2}\Big)^u\quad\text{and}\quad\frac{(-1)^u}{2u+1}\Big(\frac{
	-72-\pi ^2+12 \sqrt{36
		+\pi ^2
}}{\pi^2}\Big)^u
$$
which are closely related to the solutions~\eqref{Equ:HomSol} when dealing with $S_2(t)$. Indeed, the algorithm~\cite{Petkov:92} (and the generalized version~\cite{ABPS} implemented in \texttt{Sigma}) asks again for the roots of the same polynomial~\eqref{Equ:CharPoly} that establishes the powers involving the algebraic element 
$\sqrt{36
	+\pi ^2
}$. Thus using the same rationalizing transformation~\eqref{Equ:PiToBSum2} that we used above for the $S_2(t)$ one obtains homogeneous solutions that are free of any algebraic elements. This suggests to reconsider the sum $S_3(t)$ by applying first the substitution~\eqref{Equ:PiToBSum2} leading to
$$\tilde{S}_3(t)=\sum_{u=0}^{t-2} \frac{(-1)^u \big(
	\frac{1-b^2}{b^2}\big)^u 
	}{(2 u)!}\sum_{s=0}^{t-i-2
} \frac{(-s
-u
)_u \big(
\frac{1}{2}
-s
-u
\big)_{1
	+s
	+u
}}{(1
+s
+2 u
)!}\sum_{r=0}^{
t-u-s-1
} (-1)^r \big(
\tfrac{b^2}{1-b^2}\big)^{r} \binom{-\frac{1}{2}-r}{-1
-r
-s
+t
-u
};$$
i.e., $S_3(t)$ equals to $\tilde{S}_3(t)$ after applying the inverse transformation~\eqref{Equ:BToPiSum2}. To this end, we insert the double sum
\begin{mma}
\In \tilde{T}=\sum_{s=0}^{t-i-2
} \frac{(-s
	-u
	)_u \big(
	\frac{1}{2}
	-s
	-u
	\big)_{1
		+s
		+u
}}{(1
	+s
	+2 u
	)!}\sum_{r=0}^{
	t-u-s-1
} (-1)^r \big(
\tfrac{b^2}{1-b^2}\big)^{r} \binom{-\frac{1}{2}-r}{-1
	-r
	-s
	+t
	-u
};\\
\end{mma}
\noindent and apply our summation toolbox by executing
\begin{mma}
\In closedForm\tilde{T}=EvaluateMultiSum[\tilde{T},\{\}, \{u, t\}, \{0, 2\}, \{t - 2, \infty\}, SplitSums \to False]\\
\Out \frac{\big(\big(
	1-b+b^2\big) (-1+b)^{2 u}
	+\big(
	1+b+b^2\big) (1+b)^{2 u}
	\big) (-1+b)^{1-t-u} (1+b)^{1-t-u}}{2 b^2 (1+2 u)}\newline
-\frac{(-3+2 t) \big(
	-1
	+2 b^2 t
	\big)
	\big((-1+b)^{1+2 u}
	-(1+b)^{1+2 u}
	\big) (-1+b)^{-u} (1+b)^{-u} \big(
		\frac{5}{2}-t\big)_{-1+t}}{4 b^2 (1+2 u)t!}\newline
+\frac{b (-3+2 t) (-1+2 t) (-1+b)^{-1-u} (1+b)^{1+u}
	\big(\frac{5}{2}-t\big)_{-1+t}}{2 t (1+2 u)}\sum_{i=1}^u \frac{(-1+b)^i (1+b)^{-i} (-t)_i}{(-1
	+i
	+t
	)!}\newline
+\frac{3 \big(
	(-1+b)^{1+2 u}
	-(1+b)^{1+2 u}
	\big) (-1+b)^{-t-u} (1+b)^{-t-u}}{4 b^2 (1+2 u)}\sum_{i=1}^t \frac{(-1+b)^i (1+b)^i \big(
	\frac{5}{2}-i\big)_{-1+i}}{i!}\newline
-\frac{b (-3+2 t) (-1+2 t) (-1+b)^{1+u} (1+b)^{-1-u}
	\big(\frac{5}{2}-t\big)_{-1+t}}{2 t (1+2 u)}\sum_{i=1}^u \frac{(-1+b)^{-i} (1+b)^i (-t)_i}{(-1
	+i
	+t
	)!}\\
\end{mma}
In order to represent the expression in terms of the products $u!,(-t)_u$ and $t!, (t + u)!, (1/2 - t)_t$, one may use the following commands: 
\begin{mma}\MLabel{MMA:SR1}
\In closedForm\tilde{T}=SigmaReduce[closedForm\tilde{T}, u, 
Tower \to \{u!,(-t)_u
\}]\\
\MLabel{MMA:SR2}
\In closedForm\tilde{T}=SigmaReduce[closedForm\tilde{T}, t, 
Tower \to \{t!, (t + u)!, (1/2 - t)_t\}]\\
\MLabel{OutT}\Out \frac{\big(
	(-1+b)^{2 u}
	+(1+b)^{2 u}
	\big)}{2 (1+2 u)(-1+b)^{t+u-1} (1+b)^{t+u-1}}-\frac{\big(
	(-1+b)^{2 u}
	+(1+b)^{2 u}
	\big)(-1+b)^{1-u} (1+b)^{1-u} \big(
		\frac{1}{2}-t\big)_t}{2 (1+2 u)t!}\newline
+\frac{(t-u) \left(2 b^2 t-b^2-2 t+2 u+2\right)\big(
	\frac{1}{2}-t\big)_t (-t)_u}{t (2 t-1) (2 u+1) (t+u)!}
\newline
+\frac{\big(
	(-1+b)^{2 u}
	+(1+b)^{2 u}
	-b (-1+b)^{2 u}
	+b (1+b)^{2 u}
	\big)}{2 (1+2 u)(-1+b)^{t+u} (1+b)^{t+u}}  
\sum_{i=1}^t \frac{(-1+b)^i (1+b)^i \big(
	\frac{1}{2}-i\big)_i}{(-1+2 i) i!}\newline
-\frac{b (1+b)^{1+u}
	\big(\frac{1}{2}-t\big)_t}{(1+2 u)(-1+b)^{u}}\sum_{i=1}^u \frac{(-1+b)^i (1+b)^{-i} (-t)_i}{(i
	+t
	)!}
-\frac{b (-1+b)^{1+u}
	\big(\frac{1}{2}-t\big)_t}{(1+2 u)(1+b)^{u}}\sum_{i=1}^u \frac{(-1+b)^{-i} (1+b)^i (-t)_i}{(i
	+t
	)!}
\\
\end{mma}
\noindent Denoting the result in \myOut{\ref{OutT}} by $T'(t,u)$ we obtain
\begin{equation}\label{Equ:S3Simple}
\tilde{S}_3(t)=\sum_{u=0}^{t-2} \frac{(-1)^u \big(
	\frac{1-b^2}{b^2}\big)^u 
}{(2 u)!}T'(t,u).
\end{equation}
Finally, applying the substitution~\eqref{Equ:BToPiSum2} to the right-hand side of~\eqref{Equ:S3Simple} yields the simplification of $S_3(t)$ that has been used to perform
the estimate~\eqref{BPRSinvlem5eqn3} stated in Lemma~\ref{BPRSinvlem5}; for further details we refer to  Section~\ref{Sec5b}.

To our knowledge this is the first time that rationalizing transformations, usually applied to simplify integrals, have been applied non-trivially to symbolic summation. As pointed out, in our concrete calculations such transformations can be discovered automatically when one looks for ``nice'' hypergeometric solutions and sum solutions over such products of a given linear recurrence. It would be interesting to see where this extra feature could be also exploited in other applications.

\subsection{Estimates for the sums $(S_j(t))_{1\le j\le 9}$}\label{Sec5b}
We restrict to give detailed proofs for the bound estimates of $S_1(t)$ and $S_2(t)$ in Lemma \ref{BPRSinvlem5}. The derivations of the cases $3\le j\le 9$ work analogously.


\emph{Proof of Lemma \ref{BPRSinvlem5}:} 
We start with the definition \eqref{BPRSinvdef5b},
\begin{align}\nonumber
S_1(t)&=\pa{-1}{24}^t\frac{\left(\frac 12-t\right)_{t+1}}{t}\sum_{u=1}^{t}\frac{(-1)^u (-t)_u}{(t+u)!(2u-1)!}\alpha^{2u}\\\nonumber
&=\frac{\smallbinom{2t}{t}}{2t\cdot 96^t}\sum_{u=1}^{t}\left(\prod_{j=1}^{u}\frac{1-\frac{j-1}{t}}{1+\frac{j}{t}}\right)\frac{\alpha^{2u}}{(2u-1)!}\\\nonumber
&=\frac{1}{2\sqrt{\pi}}\frac{1}{24^t\cdot t^{\frac 32}}\left(1+O_{\le \frac 18}\left(\frac 1t\right)\right)\sum_{u=1}^{t}\left(\prod_{j=1}^{u}\frac{1-\frac{j-1}{t}}{1+\frac{j}{t}}\right)\frac{\alpha^{2u}}{(2u-1)!}\ \ \left(\text{by \Cref{BPRSprelimlem2}}\right)\\\nonumber
&=\frac{1}{2\sqrt{\pi}}\frac{1}{24^t\cdot t^{\frac 32}}\left(1+O_{\le \frac 18}\left(\frac 1t\right)\right)\sum_{u=1}^{t}\left(1+O_{\le u^2}\left(\frac 1t\right)\right)\frac{\alpha^{2u}}{(2u-1)!}\\\nonumber
&\hspace{5 cm} \left(\text{by \Cref{BPRSprelimlem1} with}\ (x_j,y_j,n)=\left(\frac{j-1}{t},\frac jt,u\right)\right)\\\label{S1estimeqn1}
&=\frac{1}{2\sqrt{\pi}}\frac{1}{24^t\cdot t^{\frac 32}}\left(1+O_{\le \frac 18}\left(\frac 1t\right)\right)\left(\underset{=:S^{[1]}_1(t)}{\underbrace{\sum_{u=1}^{t}\frac{\alpha^{2u}}{(2u-1)!}}}+\underset{=:S^{[2]}_1(t)}{\underbrace{\sum_{u=1}^{t}\frac{\alpha^{2u}}{(2u-1)!}\cdot O_{\le u^2}\left(\frac 1t\right)}}\right).
\end{align}
First we bound $S^{[2]}_1(t)$,
\begin{equation}\label{S1estimeqn2}
\left|S^{[2]}_1(t)\right|\le \frac 1t\sum_{u=1}^{t}\frac{u^2\alpha^{2u}}{(2u-1)!}\le \frac 1t\sum_{u\ge 1}\frac{u^2\alpha^{2u}}{(2u-1)!}=\frac{\frac{\a}{4}\left(3\a\cosh(\a)+(\a^2+1)\sinh(\a)\right)}{t}.
\end{equation}
Now it remains to estimate $S^{[1]}_1(t)$. Note that
\begin{equation*}
S^{[1]}_1(t)=\sum_{u=1}^{\infty}\frac{\alpha^{2u}}{(2u-1)!}- \sum_{u=t+1}^{\infty}\frac{\alpha^{2u}}{(2u-1)!}=\a\sinh(\a)-\sum_{u=t+1}^{\infty}\frac{\alpha^{2u}}{(2u-1)!}.
\end{equation*}
By Lemma \ref{BPRSprelimlem3} with $k=1$, 
\begin{equation*}
0<\sum_{u=t+1}^{\infty}\frac{\alpha^{2u}}{(2u-1)!}=2\sum_{u=t+1}^{\infty}\frac{u\cdot \alpha^{2u}}{(2u)!}\le \frac{2\a^4}{9\cdot t^2}\le \frac{2\a^3}{t},
\end{equation*}
which implies that
\begin{equation}\label{S1estimeqn3}
S^{[1]}_1(t)=\a\sinh(\a)+O_{\le 2\a^3}\left(\frac 1t\right).
\end{equation}
Combining \eqref{S1estimeqn2} and \eqref{S1estimeqn3} gives
\begin{align*}
&S_1(t)\\
&=\frac{1}{2\sqrt{\pi}}\frac{1}{24^t\cdot t^{\frac 32}}\left(1+O_{\le \frac 18}\left(\frac 1t\right)\right)\left(\a\sinh(\a)+O_{\le 2\a^3}\left(\frac 1t\right)+O_{\le \frac{\a}{4}\left(3\a\cosh(\a)+(\a^2+1)\sinh(\a)\right)}\left(\frac 1t\right)\right)\\
&=\frac{\a\sinh(\a)}{2\sqrt{\pi}}\frac{1}{24^t\cdot t^{\frac 32}}\left(1+O_{\le \frac 18}\left(\frac 1t\right)\right)\left(1+O_{\le \frac{2\a^3}{\a\sinh(\a)}}\left(\frac 1t\right)+O_{\le \frac{\frac{\a}{4}\left(3\a\cosh(\a)+(\a^2+1)\sinh(\a)\right)}{\a\sinh(\a)}}\left(\frac 1t\right)\right)\\
&=\frac{\alpha\sinh(\alpha)}{2\sqrt{\pi}}\frac{1}{24^t\cdot t^{\frac 32}}\left(1+O_{\le 2.6}\left(\frac 1t\right)\right),
\end{align*}
which finishes the proof of \eqref{BPRSinvlem5eqn1}.

To estimate the sum $S_2(t)$ defined in~\eqref{BPRSinvdef5c} we list three basic facts which be used for the estimates of $S_2(t)$. For $t\ge 2$,
\begin{equation}\label{fact1}
\frac{\pa{\a^2}{1+\a^2}^{t-1}}{\sqrt{t}}\le \frac{1}{2t},
\end{equation}
 for $t\in \mathbb{Z}_{\ge 0}$,
\begin{equation}\label{fact2}
(-1)^t\left(\frac 12-t\right)_t=\frac{\smallbinom{2t}{t}\cdot t!}{4^t},
\end{equation}
and for $|x|<1$, 
\begin{equation}\label{fact3}
1-\sqrt{1-x}=\sum_{m=1}^{\infty}\frac{\smallbinom{2m}{m}}{2m-1}\pa{x}{4}^m.
\end{equation} 

Next, using \texttt{Sigma} (with the summand representation~\eqref{Equ:T2ClForm}), we split the sum $S_2(t)$ as
\[
S_2(t)=\sum_{k=1}^{5}S^{[k]}_{2}(t),
\]
with
\begin{align}\label{sum2part1}
&S^{[1]}_2(t):=\left(\frac 12-t\right)_t(-1)^t\pa{\a^2}{1+\a^2}^{t-1}\sum_{u=1}^{t-2}\frac{\a^{2u}}{(2u)!}\frac{(-1)^u(-t)_u}{(t+u)!}\left(-2\frac{(t-u)u(t+u)}{(t-1)t(2t-1)}+\frac{t-u}{t}\frac{\a^2}{1+\a^2}\right),\\\label{sum2part2}
&S^{[2]}_2(t):=\frac 12\sum_{u=1}^{t-2}\left(\left(\sqrt{1+\a^2}+1\right)\frac{\left(\sqrt{1+\a^2}-1\right)^{2u}}{(2u)!}-\left(\sqrt{1+\a^2}-1\right)\frac{\left(\sqrt{1+\a^2}+1\right)^{2u}}{(2u)!}\right),\\\nonumber
&S^{[3]}_2(t):=\frac{\left(\frac 12-t\right)_t(-1)^{t+1}}{2\cdot t!}\pa{\a^2}{1+\a^2}^{t-1}\\\label{sum2part3}
&\hspace{2.5 cm}\times\sum_{u=1}^{t-2}\frac{\left(\sqrt{1+\a^2}+1\right)\left(\sqrt{1+\a^2}-1\right)^{2u}-\left(\sqrt{1+\a^2}-1\right)\left(\sqrt{1+\a^2}+1\right)^{2u}}{(2u)!},\\\nonumber
&S^{[4]}_2(t):=\left(\frac 12-t\right)_t(-1)^{t+1}\pa{\a^2}{1+\a^2}^{t}\sum_{u=1}^{t-2}\frac{1}{(2u)!}\\\nonumber
&\hspace{4 cm}\times\Biggl(\left(\sqrt{1+\a^2}+1\right)^{2u}\sum_{s=1}^{u}\frac{(-t)_s(-1)^s}{(t+s)!}\pa{\sqrt{1+\a^2}-1}{\sqrt{1+\a^2}+1}^s\\\label{sum2part4}
&\hspace{5.5 cm}+\left(\sqrt{1+\a^2}-1\right)^{2u}\sum_{s=1}^{u}\frac{(-t)_s(-1)^s}{(t+s)!}\pa{\sqrt{1+\a^2}+1}{\sqrt{1+\a^2}-1}^s\Biggr),\\\label{sum2part5}
&S^{[5]}_2(t):=\frac{\sqrt{1+\a^2}}{2}\sum_{u=1}^{t}\frac{\left(\sqrt{1+\a^2}+1\right)^{2u}-\left(\sqrt{1+\a^2}-1\right)^{2u}}{(2u)!}\sum_{s=1}^{t}\frac{\left(\frac 12-s\right)_s(-1)^s}{(2s-1) s!}\pa{\a^2}{1+\a^2}^s.
\end{align}
We treat the $S^{[k]}_2(t)$ successively. First for $S^{[1]}_2(t)$, after simplifying the right hand side of \eqref{sum2part1}, we obtain
\begin{align*}
&\left|S^{[1]}_2(t)\right|\\
&\os{\eqref{fact2}}{=}\frac{\smallbinom{2t}{t} t!}{4^t}\pa{\a^2}{1+\a^2}^{t-1}\left|\sum_{u=1}^{t-2}\frac{\a^{2u}}{(2u)!}\frac{(-1)^u(-t)_u}{(t+u)!}\left(-2\frac{(t-u)u(t+u)}{(t-1)t(2t-1)}+\frac{t-u}{t}\frac{\a^2}{1+\a^2}\right)\right|\\
&\le \frac{\pa{\a^2}{1+\a^2}^{t-1}}{\sqrt{\pi t}}\sum_{u=1}^{t-2}\frac{\a^{2u}}{(2u)!}\left(\prod_{j=1}^{u}\frac{t+1-j}{t+j}\right)\left(2\frac{(t-u)u(t+u)}{(t-1)t(2t-1)}+\frac{t-u}{t}\frac{\a^2}{1+\a^2}\right)\ \ \left(\text{by \Cref{BPRSprelimlem2}}\right)\\
&\le \left(2+\frac{\a^2}{1+\a^2}\right)\frac{\pa{\a^2}{1+\a^2}^{t-1}}{\sqrt{\pi t}}\sum_{u=1}^{t-2}\frac{\a^{2u}}{(2u)!}\\
&\hspace{1.5 cm} \left(\text{as}\ \frac{t+1-j}{t+j}\le 1\ \text{for}\ j\ge 1, \frac{(t-u)t(t+u)}{(t-1)t(2t-1)}\le 1,\ \text{and}\ \frac{t-u}{t}\le 1\ \text{for}\ 1\le u\le t-2\right)\\
&\le \left(2+\frac{\a^2}{1+\a^2}\right)\frac{\pa{\a^2}{1+\a^2}^{t-1}}{\sqrt{\pi t}}\sum_{u\ge 1}\frac{\a^{2u}}{(2u)!}=\left(2+\frac{\a^2}{1+\a^2}\right)\left(\cosh(\a)-1\right)\frac{\pa{\a^2}{1+\a^2}^{t-1}}{\sqrt{\pi t}}\\
&\os{\eqref{fact1}}{\le} \left(2+\frac{\a^2}{1+\a^2}\right)\frac{\left(\cosh(\a)-1\right)}{\sqrt{\pi}}\frac{1}{2t}\le \frac{0.1}{t},
\end{align*}
which implies that
\begin{equation}\label{sum2part1estim}
S^{[1]}_2(t)=O_{\le 0.1}\left(\frac 1t\right).
\end{equation}
According to \eqref{sum2part2} we write $S^{[2]}_2(t)$ as the difference of two sums
\begin{equation}\label{sum2part2estim1}
S^{[2]}_2(t):=\frac 12\left(S^{[2]}_{2,1}(t)-S^{[2]}_{2,2}(t)\right),
\end{equation}
with
\begin{align*}
S^{[2]}_{2,1}(t)&=\left(\sqrt{1+\a^2}+1\right)\sum_{u=1}^{t-2}\frac{\left(\sqrt{1+\a^2}-1\right)^{2u}}{(2u)!}\ \text{and}\ S^{[2]}_{2,2}(t)=\left(\sqrt{1+\a^2}-1\right)\sum_{u=1}^{t-2}\frac{\left(\sqrt{1+\a^2}+1\right)^{2u}}{(2u)!}.
\end{align*}
Next, observe that
\begin{align*}
S^{[2]}_{2,1}(t)&=\left(\sqrt{1+\a^2}+1\right)\left(\sum_{u=1}^{\infty}\frac{\left(\sqrt{1+\a^2}-1\right)^{2u}}{(2u)!}-\sum_{u=t-1}^{\infty}\frac{\left(\sqrt{1+\a^2}-1\right)^{2u}}{(2u)!}\right)\\
&=\left(\sqrt{1+\a^2}+1\right)\left(\cosh\left(\sqrt{1+\a^2}-1\right)-1-\sum_{u=t-1}^{\infty}\frac{\left(\sqrt{1+\a^2}-1\right)^{2u}}{(2u)!}\right).
\end{align*}
By Lemma \ref{BPRSprelimlem4} with $c=\sqrt{1+\a^2}-1$ and $t\mapsto t-1$, we obtain for $t\ge 2$,
\begin{equation*}
0<\sum_{u=t-1}^{\infty}\frac{\left(\sqrt{1+\a^2}-1\right)^{2u}}{(2u)!}\le \frac{\left(\sqrt{1+\a^2}-1\right)^4}{18\cdot (t-1)^2}\le \frac{\left(\sqrt{1+\a^2}-1\right)^4}{9\cdot t}\le \frac{4\cdot 10^{-5}}{t},
\end{equation*}
which in turn implies that
\begin{align}\nonumber
S^{[2]}_{2,1}(t)&=\left(\sqrt{1+\a^2}+1\right)\left(\cosh\left(\sqrt{1+\a^2}-1\right)-1+O_{\le 4\cdot 10^{-5}}\left(\frac 1t\right)\right)\\\nonumber
&=\left(\sqrt{1+\a^2}+1\right)\left(\cosh\left(\sqrt{1+\a^2}-1\right)-1\right)\left(1+O_{\le \frac{4\cdot 10^{-5}}{\cosh\left(\sqrt{1+\a^2}-1\right)-1}}\left(\frac 1t\right)\right)\\\label{sum2part2estim2}
&=\left(\sqrt{1+\a^2}+1\right)\left(\cosh\left(\sqrt{1+\a^2}-1\right)-1\right)\left(1+O_{\le 5\cdot 10^{-3}}\left(\frac 1t\right)\right).
\end{align}
Similarly, 
\begin{align*}
S^{[2]}_{2,2}(t)&=\left(\sqrt{1+\a^2}-1\right)\left(\sum_{u=1}^{\infty}\frac{\left(\sqrt{1+\a^2}+1\right)^{2u}}{(2u)!}-\sum_{u=t-1}^{\infty}\frac{\left(\sqrt{1+\a^2}+1\right)^{2u}}{(2u)!}\right)\\
&=\left(\sqrt{1+\a^2}-1\right)\left(\cosh\left(\sqrt{1+\a^2}+1\right)-1-\sum_{u=t-1}^{\infty}\frac{\left(\sqrt{1+\a^2}+1\right)^{2u}}{(2u)!}\right).
\end{align*}
By Lemma \ref{BPRSprelimlem4} with $c=\sqrt{1+\a^2}+1$ and $t\mapsto t-1$, we obtain for $t\ge 2$,
\begin{equation*}
0<\sum_{u=t-1}^{\infty}\frac{\left(\sqrt{1+\a^2}+1\right)^{2u}}{(2u)!}\le \frac{\left(\sqrt{1+\a^2}+1\right)^4}{18\cdot (t-1)^2}\le \frac{\left(\sqrt{1+\a^2}+1\right)^4}{9\cdot t}\le \frac{2.3}{t},
\end{equation*}
which implies that
\begin{align}\nonumber
S^{[2]}_{2,2}(t)&=\left(\sqrt{1+\a^2}-1\right)\left(\cosh\left(\sqrt{1+\a^2}+1\right)-1+O_{\le 2.3}\left(\frac 1t\right)\right)\\\nonumber
&=\left(\sqrt{1+\a^2}-1\right)\left(\cosh\left(\sqrt{1+\a^2}+1\right)-1\right)\left(1+O_{\le \frac{2.3}{\cosh\left(\sqrt{1+\a^2}+1\right)-1}}\left(\frac 1t\right)\right)\\\label{sum2part2estim3}
&=\left(\sqrt{1+\a^2}-1\right)\left(\cosh\left(\sqrt{1+\a^2}+1\right)-1\right)\left(1+O_{\le 0.8}\left(\frac 1t\right)\right).
\end{align}
Consequently, setting
\begin{align*}
f(t)&\mapsto S^{[2]}_{2,1}(t),\\
g(t)&= 1,\\
h(t)&\mapsto S^{[2]}_{2,2}(t),\\
A_1&= \left(\sqrt{1+\a^2}+1\right)\left(\cosh\left(\sqrt{1+\a^2}-1\right)-1\right),\\
A_2&= -\left(\sqrt{1+\a^2}-1\right)\left(\cosh\left(\sqrt{1+\a^2}+1\right)-1\right),\\
\left(E_1,E_2\right)&= \left(5\cdot 10^{-3}, 0.8\right),
\end{align*}
in \eqref{BPRSinvlem6eqn2} and applying \eqref{sum2part2estim2} and \eqref{sum2part2estim3} to \eqref{sum2part2estim1}, it follows that
\begin{align}\nonumber 
S^{[2]}_2(t)&=\frac 12\left(\left(\sqrt{1+\a^2}+1\right)\cosh\left(\sqrt{1+\a^2}-1\right)-\left(\sqrt{1+\a^2}-1\right)\cosh\left(\sqrt{1+\a^2}+1\right)-2\right)\\\label{sum2part2estim}
&\hspace{10 cm}\times\left(1+O_{\le 0.9}\left(\frac 1t\right)\right).
\end{align}
Recalling \eqref{sum2part3}, we obtain
\begin{align*}
&\left|S^{[3]}_2(t)\right|\\
&\os{\eqref{fact2}}{=}\frac{\smallbinom{2t}{t}}{2\cdot 4^t}\pa{\a^2}{1+\a^2}^t\sum_{u=1}^{t-2}\frac{\left(\sqrt{1+\a^2}+1\right)\left(\sqrt{1+\a^2}-1\right)^{2u}-\left(\sqrt{1+\a^2}-1\right)\left(\sqrt{1+\a^2}+1\right)^{2u}}{(2u)!}\\
&\os{\text{\Cref{BPRSprelimlem2}}}{\le} \frac{\pa{\a^2}{1+\a^2}^t}{2\sqrt{\pi}\sqrt{t}}\sum_{u=1}^{t-2}\frac{\left(\sqrt{1+\a^2}+1\right)\left(\sqrt{1+\a^2}-1\right)^{2u}-\left(\sqrt{1+\a^2}-1\right)\left(\sqrt{1+\a^2}+1\right)^{2u}}{(2u)!}\\
&\le \frac{\pa{\a^2}{1+\a^2}^t}{2\sqrt{\pi}\sqrt{t}}\sum_{u=1}^{t-2}\frac{\left(\sqrt{1+\a^2}+1\right)\left(\sqrt{1+\a^2}-1\right)^{2u}}{(2u)!}\le \frac{\pa{\a^2}{1+\a^2}^t}{2\sqrt{\pi}\sqrt{t}}\sum_{u\ge 1}\frac{\left(\sqrt{1+\a^2}+1\right)\left(\sqrt{1+\a^2}-1\right)^{2u}}{(2u)!}\\
&= \frac{\left(\sqrt{1+\a^2}+1\right)}{2\sqrt{\pi}}\frac{\pa{\a^2}{1+\a^2}^t}{\sqrt{t}}\left(\cosh\left(\sqrt{1+\a^2}-1\right)-1\right)\os{\eqref{fact1}}{\le} \frac{3\cdot 10^{-3}}{t},
\end{align*}
and therefore,
\begin{equation}\label{sum2part3estim}
S^{[3]}_2(t)=O_{\le 3\cdot 10^{-3}}\left(\frac 1t\right).
\end{equation}
Starting with \eqref{sum2part4}, we obtain
\begin{align*}
&\left|S^{[4]}_2(t)\right|\\
&\os{\eqref{fact2}}{=}\frac{\smallbinom{2t}{t}\pa{\a^2}{1+\a^2}^t}{4^t}\sum_{u=1}^{t-2}\frac{\left(\sqrt{1+\a^2}+1\right)^{2u}}{(2u)!}\left(\rule{0 cm}{0.8cm}\right.\sum_{s=1}^{u}\frac{(-t)_s(-1)^s}{(t+1)_s}\pa{\sqrt{1+\a^2}-1}{\sqrt{1+\a^2}+1}^s\\
&\hspace{7.5 cm}+\sum_{s=1}^{u}\frac{(-t)_s(-1)^s}{(t+1)_s}\pa{\sqrt{1+\a^2}+1}{\sqrt{1+\a^2}-1}^{s-2u}\left.\rule{0 cm}{0.8cm}\right)\\
&\le \frac{\pa{\a^2}{1+\a^2}^t}{\sqrt{\pi\cdot t}}\sum_{u=1}^{t-2}\frac{\left(\sqrt{1+\a^2}+1\right)^{2u}}{(2u)!}\left(\sum_{s=1}^{u}\pa{\sqrt{1+\a^2}-1}{\sqrt{1+\a^2}+1}^s+\sum_{s=1}^{u}\pa{\sqrt{1+\a^2}+1}{\sqrt{1+\a^2}-1}^{s-2u}\right)\\
&\hspace{1 cm} \left(\text{by \Cref{BPRSprelimlem2}}\ \text{and by}\ \frac{(-t)_s(-1)^s}{(t+1)_s}=\prod_{j=1}^{s}\frac{t+1-j}{t+j}\le 1\ \text{as}\ t+1-j\le t+j\ \text{for}\ j\ge 1\right)\\
&=\frac{\pa{\a^2}{1+\a^2}^t}{\sqrt{\pi\cdot t}}\sum_{u=1}^{t-2}\frac{\left(\sqrt{1+\a^2}+1\right)^{2u}}{(2u)!}\left(\sum_{s=1}^{u}\pa{\sqrt{1+\a^2}-1}{\sqrt{1+\a^2}+1}^s+\sum_{s=u}^{2u-1}\pa{\sqrt{1+\a^2}-1}{\sqrt{1+\a^2}+1}^{s}\right)\\
&\le \frac{\pa{\a^2}{1+\a^2}^t}{\sqrt{\pi\cdot t}}\sum_{u=1}^{t-2}\frac{\left(\sqrt{1+\a^2}+1\right)^{2u}}{(2u)!}\left(\sum_{s\ge 1}\pa{\sqrt{1+\a^2}-1}{\sqrt{1+\a^2}+1}^s+\sum_{s\ge 1}\pa{\sqrt{1+\a^2}-1}{\sqrt{1+\a^2}+1}^{s}\right)\ \left(\text{as}\ s \ge u\ge 1\right)\\
&\le \frac{\pa{\a^2}{1+\a^2}^t}{\sqrt{\pi\cdot t}}\sum_{u=1}^{t-2}\frac{\left(\sqrt{1+\a^2}+1\right)^{2u}}{(2u)!}\left(2\sum_{s\ge 1}\pa{1}{16}^s\right)\ \left(\text{as}\ \frac{\sqrt{1+\a^2}-1}{\sqrt{1+\a^2}+1}\le\frac{1}{16}\right)\\
&= \frac{2}{15\sqrt{\pi}}\frac{\pa{\a^2}{1+\a^2}^t}{\sqrt{t}}\sum_{u=1}^{t-2}\frac{\left(\sqrt{1+\a^2}+1\right)^{2u}}{(2u)!}\le \frac{2}{15\sqrt{\pi}}\frac{\pa{\a^2}{1+\a^2}^t}{\sqrt{t}}\sum_{u\ge 1}\frac{\left(\sqrt{1+\a^2}+1\right)^{2u}}{(2u)!}\\
&= \frac{2}{15\sqrt{\pi}}\frac{\pa{\a^2}{1+\a^2}^t}{\sqrt{t}}\left(\cosh\left(\sqrt{1+\a^2}+1\right)-1\right) \os{\eqref{fact1}}{\le}  \frac{0.2}{t},
\end{align*}
which implies that
\begin{equation}\label{sum2part4estim}
S^{[4]}_2(t)=O_{\le 0.2}\left(\frac 1t\right).
\end{equation}
Finally 
we rewrite $S^{[5]}_2(t)$ as
\begin{align}\nonumber
S^{[5]}_2(t)&\os{\eqref{fact2}}{=}\frac{\sqrt{1+\a^2}}{2}\sum_{u=1}^{t}\frac{\left(\sqrt{1+\a^2}+1\right)^{2u}-\left(\sqrt{1+\a^2}-1\right)^{2u}}{(2u)!}\sum_{s=1}^{t}\frac{\smallbinom{2s}{s}}{(2s-1)}\pa{\a^2}{4(1+\a^2)}^s\\\label{sum2part5estim1}
&=:\frac{\sqrt{1+\a^2}}{2}\left(S^{[5]}_{2,1}(t)-S^{[5]}_{2,2}(t)\right)S^{[5]}_{2,3}(t),
\end{align}
with
\begin{align*}
S^{[5]}_{2,1}(t)&=\sum_{u=1}^{t}\frac{\left(\sqrt{1+\a^2}+1\right)^{2u}}{(2u)!},\ S^{[5]}_{2,2}(t)=\sum_{u=1}^{t}\frac{\left(\sqrt{1+\a^2}-1\right)^{2u}}{(2u)!},\\
S^{[5]}_{2,3}(t)&=\sum_{s=1}^{t}\frac{\smallbinom{2s}{s}}{(2s-1)}\pa{\a^2}{4(1+\a^2)}^s.
\end{align*}
Observe that
\begin{align*}
S^{[5]}_{2,1}(t)&=\sum_{u=1}^{\infty}\frac{\left(\sqrt{1+\a^2}+1\right)^{2u}}{(2u)!}-\sum_{u=t+1}^{\infty}\frac{\left(\sqrt{1+\a^2}+1\right)^{2u}}{(2u)!}\\
&=\cosh\left(\sqrt{1+\a^2}+1\right)-1-\sum_{u=t+1}^{\infty}\frac{\left(\sqrt{1+\a^2}+1\right)^{2u}}{(2u)!}.
\end{align*}
Hence by Lemma \ref{BPRSprelimlem4} with $c= \sqrt{1+\a^2}+1$, we get
\[
0<\sum_{u=t+1}^{\infty}\frac{\left(\sqrt{1+\a^2}+1\right)^{2u}}{(2u)!}\le \frac{\left(\sqrt{1+\a^2}+1\right)^{4}}{18\cdot t^2}\le \frac{1.2}{t},
\]
which implies that
\begin{align}\nonumber
S^{[5]}_{2,1}(t)&=\cosh\left(\sqrt{1+\a^2}+1\right)-1+O_{\le 1.2}\left(\frac 1t\right)\\\nonumber
&=\left(\cosh\left(\sqrt{1+\a^2}+1\right)-1\right)\left(1+O_{\le \frac{1.2}{\cosh\left(\sqrt{1+\a^2}+1\right)-1}}\left(\frac{1}{t}\right)\right)\\\label{sum2part5estim2}
&=\left(\cosh\left(\sqrt{1+\a^2}+1\right)-1\right)\left(1+O_{\le 0.4}\left(\frac{1}{t}\right)\right).
\end{align}
Similarly, for $S^{[5]}_{2,2}(t)$ we obtain
\begin{align*}
S^{[5]}_{2,2}(t)&=\sum_{u=1}^{\infty}\frac{\left(\sqrt{1+\a^2}-1\right)^{2u}}{(2u)!}-\sum_{u=t+1}^{\infty}\frac{\left(\sqrt{1+\a^2}-1\right)^{2u}}{(2u)!}\\
&=\cosh\left(\sqrt{1+\a^2}-1\right)-1-\sum_{u=t+1}^{\infty}\frac{\left(\sqrt{1+\a^2}-1\right)^{2u}}{(2u)!},
\end{align*}
and by Lemma \ref{BPRSprelimlem4} with $c= \sqrt{1+\a^2}-1$, we obtain
\[
0<\sum_{u=t+1}^{\infty}\frac{\left(\sqrt{1+\a^2}-1\right)^{2u}}{(2u)!}\le \frac{\left(\sqrt{1+\a^2}-1\right)^{4}}{18\cdot t^2}\le \frac{2\cdot 10^{-5}}{t},
\]
which implies that
\begin{align}\nonumber
S^{[5]}_{2,2}(t)&=\cosh\left(\sqrt{1+\a^2}-1\right)-1+O_{\le 2\cdot 10^{-5}}\left(\frac 1t\right)\\\nonumber
&=\left(\cosh\left(\sqrt{1+\a^2}-1\right)-1\right)\left(1+O_{\le \frac{2\cdot 10^{-5}}{\cosh\left(\sqrt{1+\a^2}-1\right)-1}}\left(\frac{1}{t}\right)\right)\\\label{sum2part5estim3}
&=\left(\cosh\left(\sqrt{1+\a^2}-1\right)-1\right)\left(1+O_{\le 3\cdot 10^{-3}}\left(\frac{1}{t}\right)\right).
\end{align}
Applying \eqref{BPRSinvlem6eqn2} with
\begin{align*}
f(t)&\mapsto S^{[2]}_{5,1}(t),\\
g(t)&= 1,\\
h(t)&\mapsto S^{[2]}_{5,2}(t),\\
A_1&= \left(\cosh\left(\sqrt{1+\a^2}+1\right)-1\right),\\
A_2&= -\left(\cosh\left(\sqrt{1+\a^2}-1\right)-1\right),\\
\left(E_1,E_2\right)&= \left(0.4, 3\cdot 10^{-3}\right),
\end{align*}
from \eqref{sum2part5estim2} and \eqref{sum2part5estim3} we obtain
\begin{align}\nonumber
&S^{[5]}_{2,1}(t)-S^{[5]}_{2,2}(t)\\\label{sum2part5estim4}
&=\left(\cosh\left(\sqrt{1+\a^2}+1\right)-\cosh\left(\sqrt{1+\a^2}-1\right)\right)\left(1+O_{\le 0.5}\left(\frac 1t\right)\right).
\end{align}
Finally, for $S^{[5]}_{2,3}(t)$ it follows that
\begin{align}\nonumber
S^{[5]}_{2,3}(t)&=\sum_{s=1}^{t}\frac{\smallbinom{2s}{s}}{(2s-1)}\pa{\a^2}{4(1+\a^2)}^s=\sum_{s=1}^{\infty}\frac{\smallbinom{2s}{s}}{(2s-1)}\pa{\a^2}{4(1+\a^2)}^s-\sum_{s=t+1}^{\infty}\frac{\smallbinom{2s}{s}}{(2s-1)}\pa{\a^2}{4(1+\a^2)}^s\\\nonumber
&=1-\sqrt{1-\frac{\a^2}{1+\a^2}}-\sum_{s=t+1}^{\infty}\frac{\smallbinom{2s}{s}}{(2s-1)}\pa{\a^2}{4(1+\a^2)}^s\ \ \left(\text{by}\ \eqref{fact3}\ \text{with}\ x\mapsto \frac{\a^2}{1+\a^2}\right)\\\label{sum2part5estim5}
&=\frac{\sqrt{1+\a^2}-1}{\sqrt{1+\a^2}}-\sum_{s=t+1}^{\infty}\frac{\smallbinom{2s}{s}}{(2s-1)}\pa{\a^2}{4(1+\a^2)}^s.
\end{align}
Concerning the last series, for $t\ge 2$,
\begin{align}\nonumber
0<\sum_{s=t+1}^{\infty}\frac{\smallbinom{2s}{s}}{(2s-1)}\pa{\a^2}{4(1+\a^2)}^s&\le \frac{1}{2\sqrt{\pi}\cdot t^{\frac 32}}\sum_{s=t+1}^{\infty}\pa{\a^2}{4(1+\a^2)}^s\\\nonumber
&\ \ \left(\text{by \Cref{BPRSprelimlem2}}\ \text{and}\ 2s-1\ge 2t\ \text{for}\ s\ge t+1\right)\\\nonumber
&\le \frac{1}{2\sqrt{\pi}\cdot t^{\frac 32}}\sum_{s\ge 3}\pa{\a^2}{4(1+\a^2)}^s\ \ \left(\text{as}\ t\ge 2\right)\\\label{sum2part5estim6}
&\le \frac{1}{2\sqrt{\pi}\cdot t^{\frac 32}}\sum_{s\ge 3}\pa{1}{18}^s\le \frac{6\cdot 10^{-5}}{t}.
\end{align}
Applying \eqref{sum2part5estim6} to \eqref{sum2part5estim5}, we obtain
\begin{equation}\label{sum2part5estim7}
S^{[5]}_{2,3}(t)=\frac{\sqrt{1+\a^2}-1}{\sqrt{1+\a^2}}+O_{\le 6\cdot 10^{-5}}\left(\frac 1t\right)=\frac{\sqrt{1+\a^2}-1}{\sqrt{1+\a^2}}\left(1+O_{\le 6\cdot 10^{-4}}\left(\frac 1t\right)\right).
\end{equation}
Applying \eqref{sum2part5estim4} and \eqref{sum2part5estim7} to \eqref{sum2part5estim1}, it follows that
\begin{align}\nonumber
&S^{[5]}_2(t)\\\nonumber
&=\frac{\left(\sqrt{1+\a^2}-1\right)\left(\cosh\left(\sqrt{1+\a^2}+1\right)-\cosh\left(\sqrt{1+\a^2}-1\right)\right)}{2}\\\nonumber
&\hspace{8 cm}\cdot \left(1+O_{\le 0.5}\left(\frac 1t\right)\right)\left(1+O_{\le 6\cdot 10^{-4}}\left(\frac 1t\right)\right)\\\label{sum2part5estim8}
&=\frac{\left(\sqrt{1+\a^2}-1\right)\left(\cosh\left(\sqrt{1+\a^2}+1\right)-\cosh\left(\sqrt{1+\a^2}-1\right)\right)}{2}\left(1+O_{\le 0.6}\left(\frac 1t\right)\right).
\end{align}
Applying 
\begin{align*}
f(t)&\mapsto S^{[2]}_2(t),\\
h(t)&\mapsto S^{[5]}_2(t),\\
g(t)&= 1,\\
A_1&= \frac{1}{2}\left(\left(\sqrt{1+\a^2}+1\right)\cosh\left(\sqrt{1+\a^2}-1\right)-\left(\sqrt{1+\a^2}-1\right)\cosh\left(\sqrt{1+\a^2}+1\right)-2\right),\\
A_2&= \frac{1}{2}\left(\left(\sqrt{1+\a^2}-1\right)\left(\cosh\left(\sqrt{1+\a^2}+1\right)-\cosh\left(\sqrt{1+\a^2}-1\right)\right)\right),\\
\left(E_1, E_2\right)&= \left(0.9, 0.6\right).
\end{align*}
in \eqref{BPRSinvlem6eqn2}, from \eqref{sum2part2estim} and \eqref{sum2part5estim8} we obtain
\begin{equation}\label{sum2part5estim9}
S^{[2]}_2(t)+S^{[5]}_2(t)=\left(\cosh\left(\sqrt{1+\a^2}-1\right)-1\right)\left(1+O_{\le 6.7}\left(\frac 1t\right)\right).
\end{equation}
From \eqref{sum2part1estim}, \eqref{sum2part3estim}, \eqref{sum2part4estim}, and \eqref{sum2part5estim9} it follows that
\begin{align*}
S_2(t)&=\left(\cosh\left(\sqrt{1+\a^2}-1\right)-1\right)\left(1+O_{\le 6.7}\left(\frac 1t\right)\right)+O_{\le 0.4}\left(\frac 1t\right)\\
&=\left(\cosh\left(\sqrt{1+\a^2}-1\right)-1\right)\left(1+O_{\le 54.9}\left(\frac 1t\right)\right),
\end{align*}
which concludes the proof of \eqref{BPRSinvlem5eqn2}.

To prove the remaining sums $(S_j(t))_{3\le j\le 9}$ one applies arguments very similar to those presented above. 
\qed	
	
	\section*{Acknowledgements}
	Banerjee was funded by the Austrian Science Fund (FWF): W1214-N15, project DK6. Paule and Radu were supported by grant SFB F50-06 of the FWF. 	
	This research was funded in whole or in part by the Austrian Science Fund (FWF) grants 10.55776/P20347 and 10.55776/PAT1332123. For open access purposes, 
	the authors have applied a CC BY public copyright license to any author-accepted manuscript version arising from this submission.

	{\bf Data Availability Statement} No data were generated or used in the preparation of this
	paper.
	
	{\bf Disclosure of potential conflicts of interest} The authors confirm that there is no conflict of interest in connection with this paper.

\end{document}